\documentclass[reqno]{amsart}
\usepackage{amsmath,amsthm,amssymb,amsfonts}
\usepackage{mathrsfs}
\numberwithin{equation}{section}
\usepackage{color}
\usepackage{enumerate}
\usepackage{multirow} 
\newtheoremstyle{mystyle}
{}
{}
{\normalfont}
{ }
{\bfseries}
{}
{10pt}
{ }
\theoremstyle{mystyle}
\newtheorem{theorem}{Theorem}

\newtheorem{lemma}{Lemma}
\newtheorem{remark}{Remark}

\usepackage{mathtools}
\usepackage[pdftex]{graphicx}

\def\Diag{\mathop{\rm Diag}\nolimits}
\def\tr{\mathop{\rm tr}\nolimits}
\def\vec{\mathop{\rm vec}\nolimits}
\def\vech{\mathop{\rm vech}\nolimits}
\def\det{\mathop{\rm det}\nolimits}
\def\rank{\mathop{\rm rank}\nolimits}

\newcommand{\dd}{\mathrm d}

\allowdisplaybreaks[4]
\usepackage[foot]{amsaddr}
\usepackage[top=3cm, bottom=3cm, left=2.5cm, right=2.5cm]{geometry}

\usepackage[hang,small,bf]{caption}
\usepackage[subrefformat=parens]{subcaption}
\usepackage{comment}
\large
\title[QAIC of SEM for diffusion processes]{Quasi-Akaike information criterion of structural equation modeling with latent variables for diffusion processes}
\author[S. Kusano]{Shogo Kusano $^{1}$}
\author[M. Uchida]{Masayuki Uchida $^{1,2}$}
\address{$^{1}$Graduate School of Engineering Science, Osaka University}
\address{$^{2}$Center for Mathematical Modeling and Data Science (MMDS), Osaka University and JST CREST}
\begin{document}
\begin{abstract}
\fontsize{8pt}{10pt}\selectfont
We consider a model selection problem for structural equation modeling (SEM) with latent variables for diffusion processes based on high-frequency data. First, we propose the quasi-Akaike information criterion of the SEM and study the asymptotic properties. Next, we consider the situation where the set of competing models includes some misspecified parametric models. It is shown that the probability of choosing the misspecified models converges to zero. Furthermore, examples and simulation results are given.
\end{abstract}
\maketitle

\section{Introduction}
\fontsize{10pt}{16pt}\selectfont
{\setlength{\abovedisplayskip}{12pt}
\setlength{\belowdisplayskip}{12pt}
We consider a model selection problem for structural equation modeling (SEM) with latent variables for diffusion processes. First, we define the true model of the SEM. 
The stochastic processes $\mathbb{X}_{1,0,t}$ and $\mathbb{X}_{2,0,t}$
are defined by the factor models as follows:
\begin{align}
    \mathbb{X}_{1,0,t}&={\bf{\Lambda}}_{x_1,0}\xi_{0,t}+\delta_{0,t}\label{X0},\\
    \mathbb{X}_{2,0,t}&={\bf{\Lambda}}_{x_2,0}\eta_{0,t}+\varepsilon_{0,t},
\end{align}
where $\{\mathbb{X}_{1,0,t}\}_{t\geq 0}$ 
and $\{\mathbb{X}_{2,0,t}\}_{t\geq 0}$ are
$p_1$ and $p_2$-dimensional observable vector processes,
$\{\xi_{0,t}\}_{t\geq 0}$ and  $\{\eta_{0,t}\}_{t\geq 0}$ are
$k_{1}$ and  $k_{2}$-dimensional latent common factor vector processes, 
$\{\delta_{0,t}\}_{t\geq 0}$ and $\{\varepsilon_{0,t}\}_{t\geq 0}$ are $p_{1}$ and $p_{2}$-dimensional latent unique factor vector processes, respectively.
${\bf{\Lambda}}_{x_1,0}\in\mathbb{R}^{p_1\times k_{1}}$ 
and ${\bf{\Lambda}}_{x_2,0}\in\mathbb{R}^{p_2\times k_{2}}$ are constant loading matrices.
Both $p_1$ and $p_2$ are not zero, $p_1$,  $p_2$, $k_{1}$ and  $k_{2}$ are fixed, 
$k_{1}\leq p_1$ and $k_{2}\leq p_2$. 
Let $p=p_1+p_2$. 
Suppose that $\{\xi_{0,t}\}_{t\geq 0}$, $\{\delta_{0,t}\}_{t\geq 0}$ and 
$\{\varepsilon_{0,t}\}_{t\geq 0}$  satisfy the following  stochastic differential equations:
\begin{align}
    \quad\dd \xi_{0,t}&=B_{1}(\xi_{0,t})\dd t+{\bf{S}}_{1,0}\dd W_{1,t}, \ \ 
    \xi_{0,0}=c_{1}, \\
    \quad\dd\delta_{0,t}&=B_{2}(\delta_{0,t})\dd t+{\bf{S}}_{2,0}\dd W_{2,t}, \ \ 
    \delta_{0,0}=c_{2},  \\
    \quad\dd\varepsilon_{0,t}&=B_{3}(\varepsilon_{0,t})\dd t+{\bf{S}}_{3,0}\dd W_{3,t}, \ \ 
     \varepsilon_{0,0}=c_{3},      
\end{align}
where $B_{1}:\mathbb{R}^{k_{1}}\rightarrow\mathbb{R}^{k_{1}}$, ${\bf{S}}_{1,0}\in\mathbb{R}^{k_{1}\times r_{1}}$, $c_{1}\in\mathbb{R}^{k_{1}}$,
$B_{2}:\mathbb{R}^{p_1}\rightarrow\mathbb{R}^{p_1}$, ${\bf{S}}_{2,0}\in\mathbb{R}^{p_1\times r_{2}}$, $c_{2}\in\mathbb{R}^{p_1}$,  
$B_{3}:\mathbb{R}^{p_2}\rightarrow\mathbb{R}^{p_2}$, ${\bf{S}}_{3,0}\in\mathbb{R}^{p_2\times r_{3}}$, $c_{3}\in\mathbb{R}^{p_2}$ 
and $W_{1,t}$, $W_{2,t}$ and $W_{3,t}$ are $r_{1}$, $r_{2}$ and $r_{3}$-dimensional standard Wiener processes, respectively. 
Moreover, we express the relationship between $\eta_{0,t}$ and  $\xi_{0,t}$ as follows:
\begin{align}
    \eta_{0,t}={\bf{B}}_0\eta_{0,t}+{\bf{\Gamma}}_0\xi_{0,t}+\zeta_{0,t},
\end{align}
where 
${\bf{B}}_0\in\mathbb{R}^{k_{2}\times k_{2}}$ is a constant loading matrix, whose diagonal elements are zero, and ${\bf{\Gamma}}_0\in\mathbb{R}^{k_{2}\times k_{1}}$ is a constant loading matrix. Define ${\bf{\Psi}}_0=\mathbb{I}_{k_2}-{\bf{B}}_0$, where $\mathbb{I}_{k_2}$ denotes the identity matrix of size $k_2$. We assume that ${\bf{\Lambda}}_{x_1,0}$ is a full column rank matrix and ${\bf{\Psi}}_0$ is non-singular.  
$\{\zeta_{0,t}\}_{t\geq 0}$ is a $k_{2}$-dimensional latent unique factor vector process
defined by the following stochastic differential equation: 
\begin{align}
    \quad\dd\zeta_{0,t}=B_{4}(\zeta_{0,t})\dd t+{\bf{S}}_{4,0}\dd W_{4,t},\ \ 
    \zeta_{0,0}=c_{4},\label{zeta0}
\end{align}
where $B_{4}:\mathbb{R}^{k_{2}}\rightarrow\mathbb{R}^{k_{2}}$, ${\bf{S}}_{4,0}\in\mathbb{R}^{k_{2}\times r_{4}}$, $c_{4}\in\mathbb{R}^{k_{2}}$ and $W_{4,t}$ is an $r_{4}$-dimensional standard Wiener process.
Set ${\bf{\Sigma}}_{\xi\xi,0}={\bf{S}}_{1,0}{\bf{S}}_{1,0}^{\top}$, ${\bf{\Sigma}}_{\delta\delta,0}={\bf{S}}_{2,0}{\bf{S}}_{2,0}^{\top}$, ${\bf{\Sigma}}_{\varepsilon\varepsilon,0}={\bf{S}}_{3,0}{\bf{S}}_{3,0}^{\top}$ and ${\bf{\Sigma}}_{\zeta\zeta,0}={\bf{S}}_{4,0}{\bf{S}}_{4,0}^{\top}$, where $\top$ denotes the transpose. It is supposed that ${\bf{\Sigma}}_{\delta\delta,0}$ and ${\bf{\Sigma}}_{\varepsilon\varepsilon,0}$ are positive definite matrices, 
and $W_{1,t}$, $W_{2,t}$, $W_{3,t}$ and $W_{4,t}$ are independent standard Wiener processes on a stochastic basis with usual conditions $(\Omega, \mathscr{F}, \{\mathscr{F}_t\}, {\bf{P}})$.
Let $\mathbb{X}_{0,t}=(\mathbb{X}_{1,0,t}^{\top},\mathbb{X}_{2,0,t}^{\top})^{\top}$. Set ${\bf{\Sigma}}_0$ as the variance of $\mathbb{X}_{0,t}$. If there is no misunderstanding, we simply write $\mathbb{X}_{0,t}$ as $\mathbb{X}_t$.
$\mathbb{X}_n=(\mathbb{X}_{t_{i}^n})_{0\leq i\leq n}=(\mathbb{X}_{0,t_{i}^n})_{0\leq i\leq n}$ are discrete observations, where $t_{i}^n=ih_n$, $h_n=\frac{T}{n}$, $T$ is fixed, and $p_1$, $p_2$, $k_{1}$ and $k_{2}$ are independent of $n$. We consider the situation where $h_n \longrightarrow 0$
as $n \longrightarrow \infty$.
We cannot estimate all the elements of ${\bf{\Lambda}}_{x_1,0}$, ${\bf{\Lambda}}_{x_2,0}$, ${\bf{\Gamma}}_0$, ${\bf{\Psi}}_0$,
${\bf{\Sigma}}_{\xi\xi,0}$, ${\bf{\Sigma}}_{\delta\delta,0}$, ${\bf{\Sigma}}_{\varepsilon\varepsilon,0}$ and ${\bf{\Sigma}}_{\zeta\zeta,0}$. Thus, some elements may be assumed to be zero to satisfy an identifiability condition; see, e.g., Everitt \cite{Everitt(1984)}. Note that these constraints and the number of factors $k_1$ and $k_2$ are determined from the theoretical viewpoint of each research field. 

A model selection problem among the following $M$ parametric models is considered.
We define the parametric model of Model $m\in\{1,\cdots,M\}$ as follows. Set $\theta_m\in\Theta_m\subset\mathbb{R}^{q_m}$ as the parameter of Model $m$, where $\Theta_m$ is a convex compact space. It is assumed that $\Theta_m$ has locally Lipschitz boundary; see, e.g., Adams and Fournier \cite{Adams(2003)}. 
The stochastic processes $\mathbb{X}^{\theta}_{1,m,t}$ and $\mathbb{X}^{\theta}_{2,m,t}$
are defined as the following factor models:
\begin{align}
    \mathbb{X}^{\theta}_{1,m,t}&={\bf{\Lambda}}^{\theta}_{x_1,m}\xi^{\theta}_{m,t}+\delta^{\theta}_{m,t}\label{Xm}, \\
    \mathbb{X}^{\theta}_{2,m,t}&={\bf{\Lambda}}^{\theta}_{x_2,m}\eta^{\theta}_{m,t}+\varepsilon^{\theta}_{m,t}\label{Ym}, 
\end{align}
where $\{\mathbb{X}^{\theta}_{1,m,t}\}_{t\geq 0}$ and $\{\mathbb{X}^{\theta}_{2,m,t}\}_{t\geq 0}$
are $p_1$ and $p_2$-dimensional observable vector processes,
$\{\xi^{\theta}_{m,t}\}_{t\geq 0}$ and $\{\eta^{\theta}_{m,t}\}_{t\geq 0}$
are $k_{1}$ and $k_{2}$-dimensional latent common factor vector processes, 
$\{\delta^{\theta}_{m,t}\}_{t\geq 0}$ and $\{\varepsilon^{\theta}_{m,t}\}_{t\geq 0}$
are $p_{1}$ and $p_{2}$-dimensional latent unique factor vector processes, respectively.
${\bf{\Lambda}}^{\theta}_{x_1,m}\in\mathbb{R}^{p_1\times k_{1}}$ and ${\bf{\Lambda}}^{\theta}_{x_2,m}\in\mathbb{R}^{p_2\times k_{2}}$
are constant loading matrices.
Assume that $\{\xi^{\theta}_{m,t}\}_{t\geq 0}$,  $\{\delta^{\theta}_{m,t}\}_{t\geq 0}$ and
$\{\varepsilon^{\theta}_{m,t}\}_{t\geq 0}$ satisfy the following  stochastic differential equations:
\begin{align}
    \quad\dd \xi^{\theta}_{m,t}&=B_{1}(\xi^{\theta}_{m,t})\dd t+{\bf{S}}^{\theta}_{1,m}\dd W_{1,t}, \ \ 
    \xi^{\theta}_{m,0}=c_{1},\label{xim}
    \\
    \quad\dd\delta^{\theta}_{m,t}&=B_{2}(\delta^{\theta}_{m,t})\dd t+{\bf{S}}^{\theta}_{2,m}\dd W_{2,t}, \ \ 
    \delta^{\theta}_{m,0}=c_{2},\label{deltam}
\\    
    \quad\dd\varepsilon^{\theta}_{m,t}&=B_{3}(\varepsilon^{\theta}_{m,t})\dd t+{\bf{S}}^{\theta}_{3,m}\dd W_{3,t}, \ \ 
    \varepsilon^{\theta}_{m,0}=c_{3},\label{epsilonm}
\end{align}
where ${\bf{S}}^{\theta}_{1,m}\in\mathbb{R}^{k_{1}\times r_{1}}$,
${\bf{S}}^{\theta}_{2,m}\in\mathbb{R}^{p_{1}\times r_{2}}$
and
${\bf{S}}^{\theta}_{3,m}\in\mathbb{R}^{p_{2}\times r_{3}}$. 
Furthermore, the relationship between $\eta^{\theta}_{m,t}$ and  $\xi^{\theta}_{m,t}$ is expressed as follows:
\begin{align}
    \eta^{\theta}_{m,t}={\bf{B}}_m^{\theta}\eta^{\theta}_{m,t}+{\bf{\Gamma}}_m^{\theta}\xi^{\theta}_{m,t}+\zeta^{\theta}_{m,t}\label{etam},
\end{align}
where 
${\bf{B}}^{\theta}_m\in\mathbb{R}^{k_{2}\times k_{2}}$ is a constant loading matrix, whose diagonal elements are zero, and ${\bf{\Gamma}}^{\theta}_m\in\mathbb{R}^{k_{2}\times k_{1}}$ is a constant loading matrix. Set ${\bf{\Psi}}^{\theta}_m=\mathbb{I}_{k_2}-{\bf{B}}^{\theta}_m$. It is supposed that ${\bf{\Lambda}}^{\theta}_{x_1,m}$ is a full column rank matrix and ${\bf{\Psi}}^{\theta}_{m}$ is non-singular. 
$\{\zeta^{\theta}_{m,t}\}_{t\geq 0}$ is a $k_{2}$-dimensional latent unique factor vector process
defined by the following stochastic differential equation:
\begin{align}
    \quad\dd\zeta^{\theta}_{m,t}=B_{4}(\zeta^{\theta}_{m,t})\dd t+{\bf{S}}^{\theta}_{4,m}\dd W_{4,t}, \ \ 
    \zeta^{\theta}_{m,0}=c_{4},\label{zetam}
\end{align}
where ${\bf{S}}^{\theta}_{4,m}\in\mathbb{R}^{k_{2}\times r_{4}}$. Let ${\bf{\Sigma}}^{\theta}_{\xi\xi,m}={\bf{S}}^{\theta}_{1,m}{\bf{S}}^{\theta\top}_{1,m}$, ${\bf{\Sigma}}^{\theta}_{\delta\delta,m}={\bf{S}}^{\theta}_{2,m}{\bf{S}}^{\theta\top}_{2,m}$, ${\bf{\Sigma}}^{\theta}_{\varepsilon\varepsilon,m}={\bf{S}}^{\theta}_{3,m}{\bf{S}}^{\theta\top}_{3,m}$ and ${\bf{\Sigma}}^{\theta}_{\zeta\zeta,m}={\bf{S}}^{\theta}_{4,m}{\bf{S}}^{\theta\top}_{4,m}$. It is assumed that ${\bf{\Sigma}}^{\theta}_{\delta\delta,m}$ and ${\bf{\Sigma}}^{\theta}_{\varepsilon\varepsilon,m}$ are positive definite matrices. 
Define $\mathbb{X}^{\theta}_{m,t}=(\mathbb{X}_{1,m,t}^{\theta\top},\mathbb{X}_{2,m,t}^{\theta\top})^{\top}$.
Set
\begin{align*}
    {\bf{\Sigma}}_m(\theta_m)=\begin{pmatrix}
    {\bf{\Sigma}}_m^{11}(\theta_m) & {\bf{\Sigma}}_m^{12}(\theta_m)\\
    {\bf{\Sigma}}_m^{12\top}(\theta_m) & {\bf{\Sigma}}_m^{22}(\theta_m)
    \end{pmatrix}
\end{align*}
as the variance of $\mathbb{X}^{\theta}_{m,t}$, where 
\begin{align*}                      
    \qquad\qquad{\bf{\Sigma}}^{11}_m(\theta_m)&
    ={\bf{\Lambda}}^{\theta}_{x_1,m}{\bf{\Sigma}}^{\theta}_{\xi\xi,m}{\bf{\Lambda}}_{x_1,m}^{\theta\top}
    +{\bf{\Sigma}}^{\theta}_{\delta\delta,m},\\  {\bf{\Sigma}}^{12}_m(\theta_m)&={\bf{\Lambda}}^{\theta}_{x_1,m}{\bf{\Sigma}}^{\theta}_{\xi\xi,m}{\bf{\Gamma}}_m^{\theta\top}{\bf{\Psi}}_m^{\theta-1\top}{\bf{\Lambda}}_{x_2,m}^{\theta\top},\\
    {\bf{\Sigma}}^{22}_m(\theta_m)&={\bf{\Lambda}}^{\theta}_{x_2,m}{\bf{\Psi}}_m^{\theta-1}({\bf{\Gamma}}_m^{\theta}{\bf{\Sigma}}^{\theta}_{\xi\xi,m}{\bf{\Gamma}}_m^{\theta\top}+{\bf{\Sigma}}^{\theta}_{\zeta\zeta,m}){\bf{\Psi}}_m^{\theta-1\top}{\bf{\Lambda}}_{x_2,m}^{\theta\top}+{\bf{\Sigma}}^{\theta}_{\varepsilon\varepsilon,m}.
\end{align*}
It is supposed that there exists $\theta_{m,0}\in{\rm{Int}}\Theta_m$ such that ${\bf{\Sigma}}_0={\bf{\Sigma}}_m(\theta_{m,0})$, and Model $m$ satisfies an identifiability condition.

Structural equation modeling (SEM) with latent variables is a method of analyzing the relationships between latent variables that cannot be observed; see, e.g., J{\"o}reskog \cite{Joreskog(1970)}, 
Everitt \cite{Everitt(1984)}, Mueller \cite{Mueller(1999)} and references therein. 
A researcher has often some candidate models in SEM. Note that the candidate models are usually specified to express different hypotheses. The goodness-of-fit test based on the likelihood ratio is widely used for model evaluation in SEM. 
Akaike \cite{Akaike(1987)} proposed the use of the Akaike information criterion (AIC) in a factor model. Using AIC in a factor model, we can choose the optimal number of factors in terms of prediction. AIC is also widely used in SEM to choose the optimal model as well as a factor model; see e.g., Huang \cite{Huang AIC(2017)}.

Thanks to the development of measuring devices, high-frequency data such as stock prices can be easily obtained these days, so that many researchers have studied parametric estimation of diffusion processes based on high-frequency data; see, e.g., Yoshida \cite{Yoshida(1992)}, Genon-Catalot and Jacod \cite{Genon(1993)}, Kessler \cite{kessler(1997)}, Uchida and Yoshida \cite{Uchi-Yoshi(2012)} and references therein. Recently, in the field of financial econometrics, the factor model based on high-frequency data has been extensively studied. A{\"i}t-Sahalia and Xiu \cite{Ait(2017)} proposed a continuous-time latent factor model for a high-dimensional model using principal component analysis. Kusano and Uchida \cite{Kusano(2024)} suggested classical factor analysis for diffusion processes. This method enables us to analyze the relationships between low-dimensional observed variables sampled with high frequency and latent variables. For instance, based on high-frequency stock price data, we can analyze latent variables such as a world market factor and factors related to a certain industry (Figure \ref{FA}). 
On the other hand, there have been few researchers who examine the relationships between these latent variables based on high-frequency data. Kusano and Uchida \cite{Kusano(JJSD)} proposed SEM with latent variables for diffusion processes. Using this method, one can examine the relationships between latent variables based on high-frequency data. For example, if we want to study the relationship between the world market factor and the Japanese financial factor, this method enables us to analyze the relationship (Figure \ref{SEM}). SEM with latent variables may be referred as the regression analysis between latent variables. While both explanatory and objective variables are observable in regression analysis, both of them are latent in SEM with latent variables. For the regression analysis and the market models based on high-frequency data, see, e.g., A{\"i}t-Sahalia et al. \cite{Ait(2020)}. 

The model selection problem for diffusion processes based on discrete observations has been actively studied. Uchida \cite{Uchida(2010)} proposed the contrast-based information criterion for ergodic diffusion processes, and obtained 
the asymptotic result of the difference between the contrast-based information criteria. Eguchi and Masuda \cite{Eguchi(2023)} studied the model comparison problem for semiparametric L$\acute{e}$vy driven SDE and suggested the Gaussian quasi-AIC. Since the information criterion is important in SEM as mentioned above, we propose the quasi-AIC (QAIC) of SEM with latent variables for diffusion processes and study the asymptotic properties. In this paper, we consider the non-ergodic case. For the ergodic case, see Appendix \ref{ergodic}.

The paper is organized as follows. In Section 2, we introduce the notation and assumptions. In Section 3, the QAIC of SEM with latent variables for diffusion processes is considered.  Moreover, the situation where the set of competing models includes some (not all) misspecified parametric models is studied. It is shown that the probability of choosing the misspecified models converges to zero. In Section 4, we give examples and simulation results. In Section 5, the results described in Section 3 are proved.
\begin{figure}[h]
    \includegraphics[width=0.9\columnwidth]{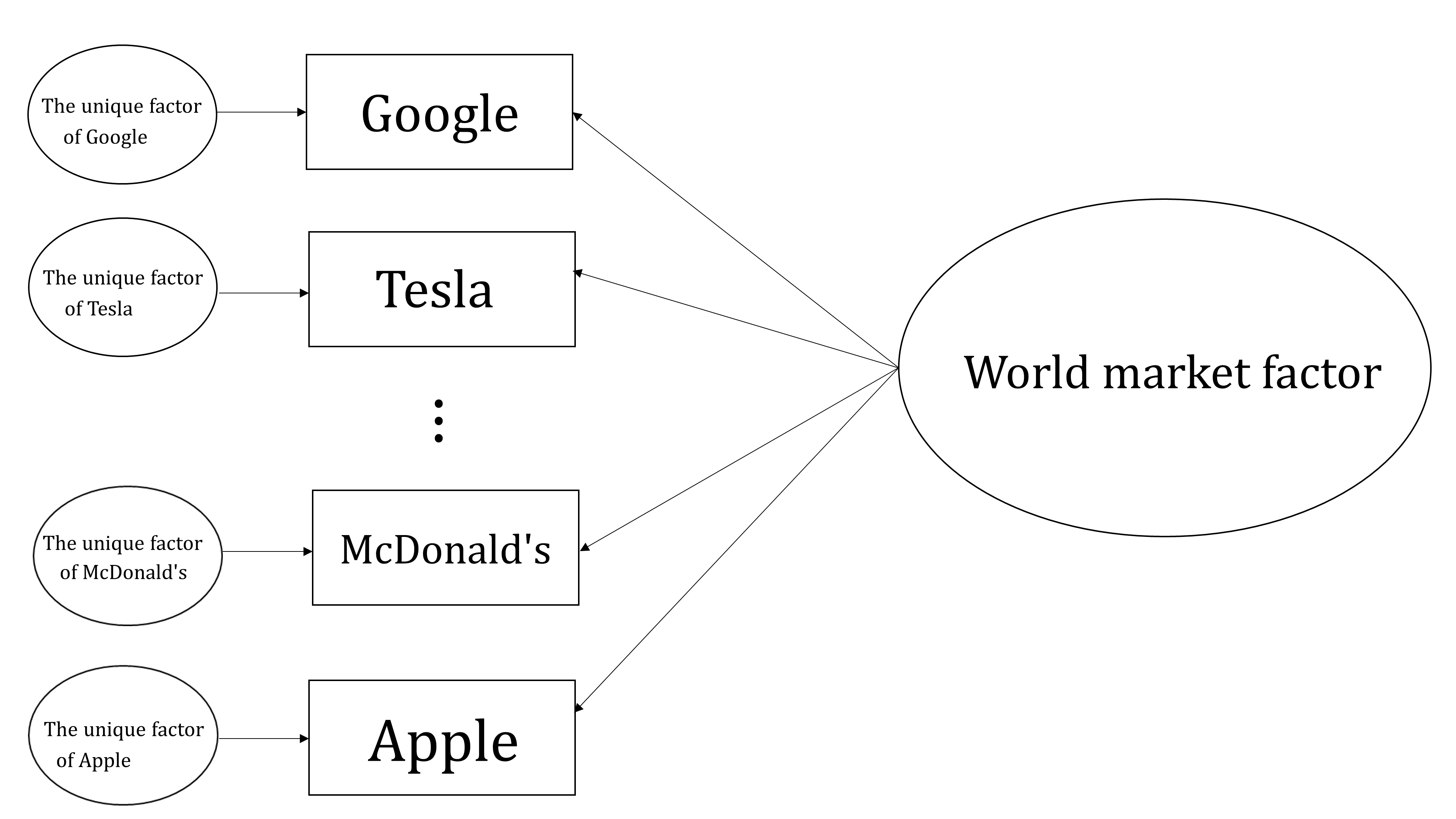}
    \includegraphics[width=0.9\columnwidth]{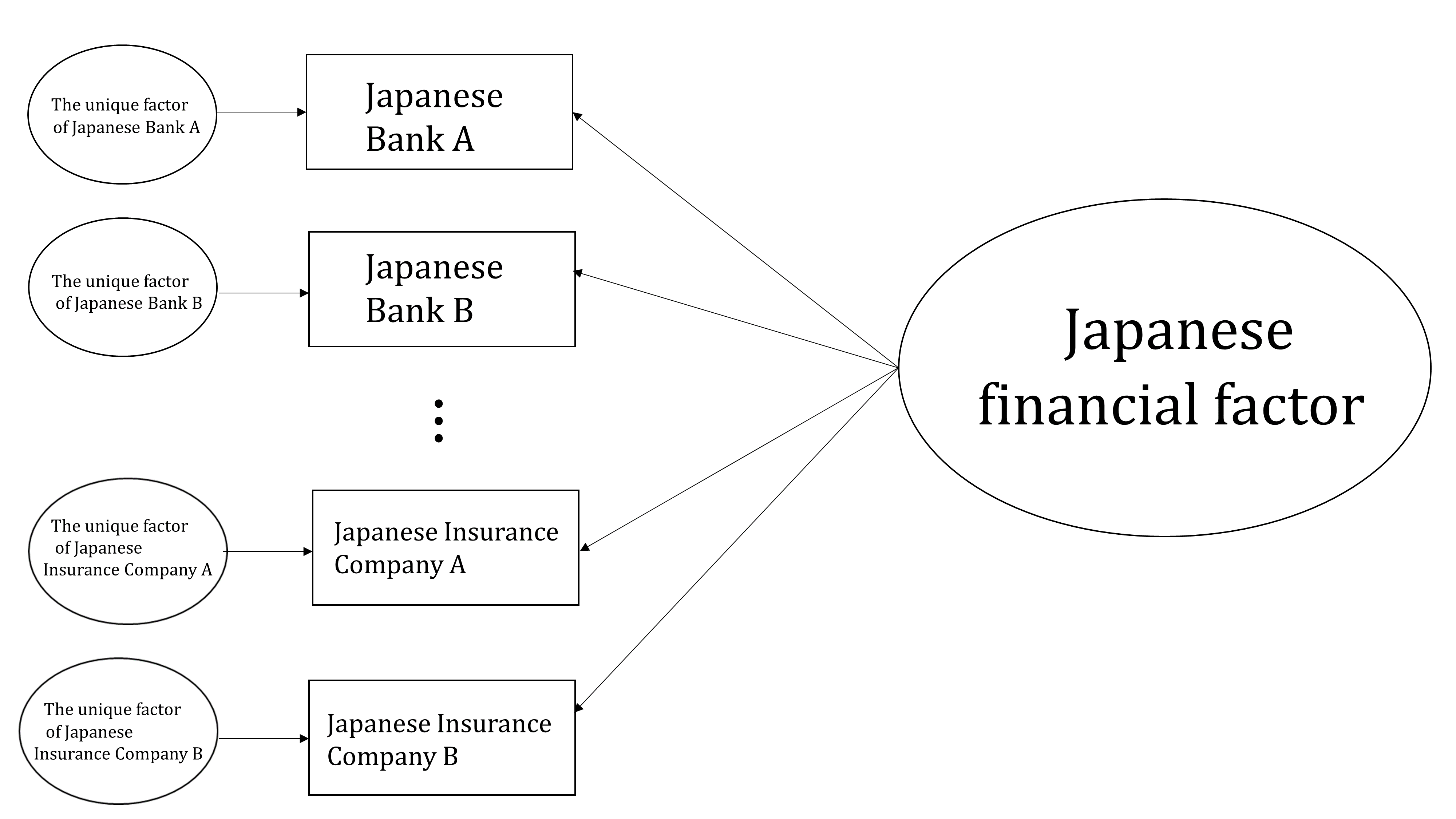}
    \caption{The path diagram for the example of factor analysis.}
    \label{FA}
\end{figure}
\begin{figure}[h]
    \includegraphics[width=0.9\columnwidth]{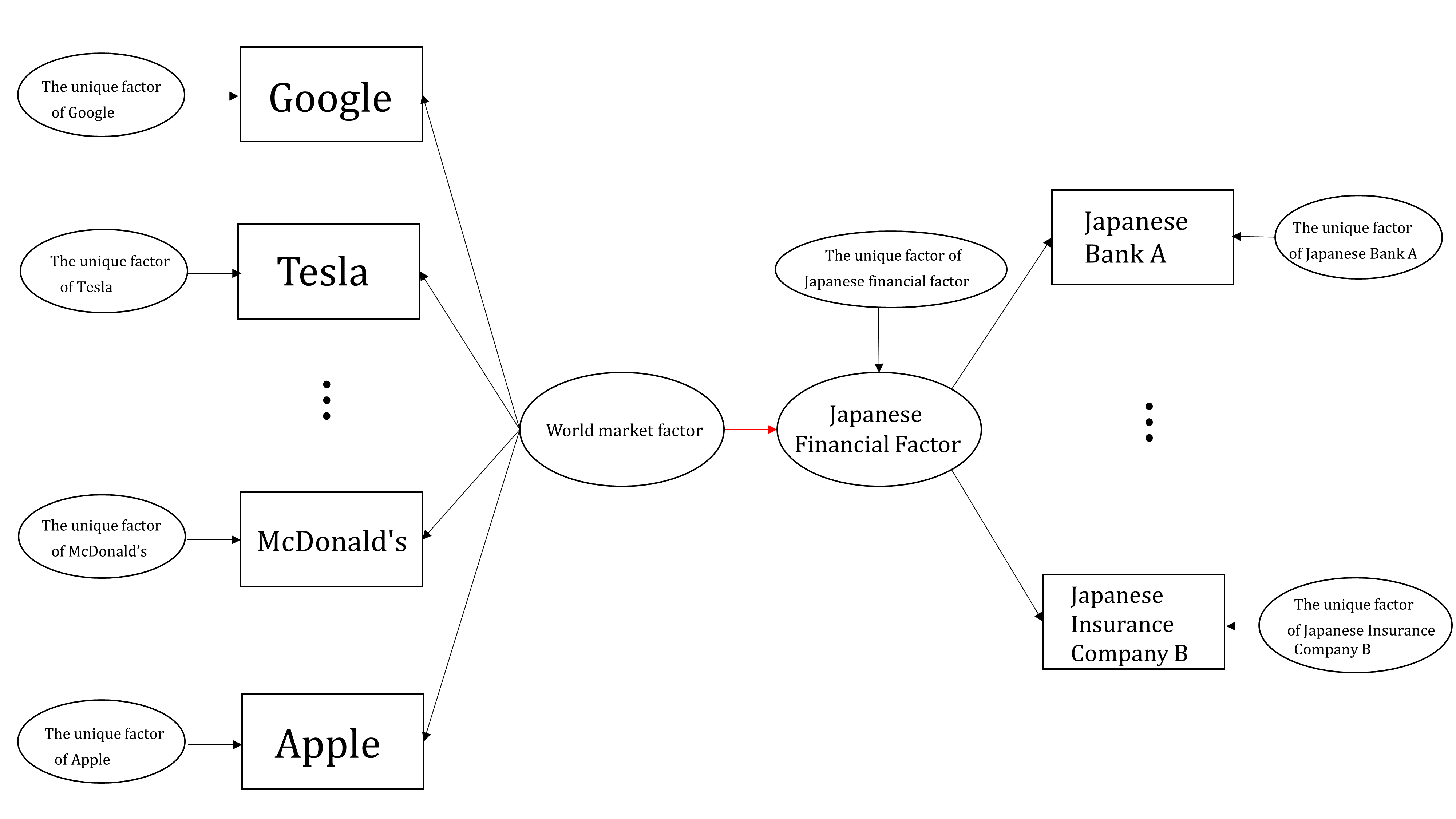}
    \caption{The path diagram for the example of SEM.}
    \label{SEM}
\end{figure}
\section{Notation and assumptions}
First, we prepare the following notations and definitions.
For any vector $v$, $|v|=\sqrt{\tr{vv^\top}}$, $v^{(i)}$ is the $i$-th element of $v$, and $\Diag v$ is the diagonal matrix, whose $i$-th diagonal element is $v^{(i)}$.
For any matrix $A$, $|A|=\sqrt{\tr{AA^\top}}$, and $A_{ij}$ is the $(i,j)$-th element of $A$. For matrices $A$ and $B$ of the same size, $A[B]=\tr(AB^{\top})$. For any matrix $A\in\mathbb{R}^{p\times p}$ and vectors $x,y\in\mathbb{R}^p$, we define $A[x,y]=x^{\top}Ay$. 
For a positive definite matrix $A$, we write $A>0$. For any symmetric matrix $A\in\mathbb{R}^{p\times p}$, $\vec A$, $\vech A$ and $\mathbb{D}_{p}$ are the vectorization of $A$, the half-vectorization of $A$ and the $p^2\times\bar{p}$ duplication matrix respectively, where $\bar{p}=p(p+1)/2$. Note that $\vec{A}=\mathbb{D}_{p}\vech{A}$; see, e.g., Harville \cite{Harville(1998)}. For any matrix $A$, $A^{+}$ stands for the Moore-Penrose inverse of $A$. Set $\mathcal{M}_p^{++}$ as the sets of all $p\times p$ real-valued positive definite matrices.
For any positive sequence $u_{n}$, 
${\rm{R}}:[0,\infty)\times \mathbb{R}^d\rightarrow \mathbb{R}$ denotes the short notation for functions which satisfy $|{\rm{R}}({u_{n}},x)|\leq u_{n}C(1+|x|)^C$ for some $C>0$. Let $ C^{k}_{\uparrow}(\mathbb R^{d})$ be the space of all functions $f$ satisfying the following conditions:
\begin{itemize}
    \item[(i)] $f$ is continuously differentiable with respect to $x\in \mathbb{R}^d$ up to order $k$. 
    \item[(ii)] $f$ and all its derivatives are of polynomial growth in $x\in \mathbb{R}^d$, i.e., 
    $g$ is
of polynomial growth in $x\in \mathbb{R}^d$ if $\displaystyle g(x)=R(1,x)$. 
\end{itemize}

The symbols $\stackrel{p}{\longrightarrow}$ and $\stackrel{d}{\longrightarrow}$ denote convergence in probability and convergence in distribution, respectively. For any process $Y_t$, $\Delta_i Y=Y_{t_{i}^n}-Y_{t_{i-1}^n}$. Set
\begin{align*}
    \mathbb{Q}_{\mathbb{XX}}=\frac{1}{T}\sum_{i=1}^n(\mathbb{X}_{t_i^n}-\mathbb{X}_{t_{i-1}^n})(\mathbb{X}_{t_i^n}-\mathbb{X}_{t_{i-1}^n})^{\top}.
\end{align*}
${\bf{E}}$ denotes the expectation under ${\bf{P}}$. Next, we make the following assumptions.
\begin{enumerate}
    \vspace{2mm}
    \item[\bf{[A]}]
    \begin{enumerate}
    \item
    \begin{enumerate}
    \item[(i)] There exists a constant $C>0$ such that
    \begin{align*}
    |B_{1}(x)-B_{1}(y)|\leq C|x-y|
    \end{align*}
    for any $x,y\in\mathbb R^{k_{1}}$.
    \vspace{2mm}
    \item[(ii)] For all $\ell>0$,
    $\displaystyle\sup_t{\bf{E}}\Bigl[\bigl|\xi_{0,t}\bigr|^{\ell}\Bigr]<\infty$.
    \vspace{2mm}
    \item[(iii)] $B_{1}\in C^{4}_{\uparrow}(\mathbb R^{k_{1}})$.
    \end{enumerate} 
    \vspace{2mm}
    \item
    \begin{enumerate}
    \item[(i)] There exists a constant $C>0$ such that
    \begin{align*}
    |B_{2}(x)-B_{2}(y)|\leq C|x-y|
    \end{align*}
    for any $x,y\in\mathbb R^{p_1}$.
    \vspace{2mm}
    \item[(ii)] For all $\ell\geq 0$, 
    $\displaystyle\sup_t{\bf{E}}\Bigl[\bigl|\delta_{0,t}\bigr|^{\ell}\Bigr]<\infty$.
    \vspace{2mm}
    \item[(iii)] $B_{2}\in C^{4}_{\uparrow}(\mathbb R^{p_1})$.
    \end{enumerate}
    \vspace{2mm}
    \item
    \begin{enumerate}
    \item[(i)] There exists a constant $C>0$ such that 
    \begin{align*}
    |B_{3}(x)-B_{3}(y)|\leq C|x-y|
    \end{align*}
    for any $x,y\in\mathbb R^{p_2}$.
    \vspace{2mm}
    \item[(ii)] For all $\ell\geq 0$, 
    $\displaystyle\sup_t{\bf{E}}\Bigl[\bigl|\varepsilon_{0,t}\bigr|^{\ell}\Bigr]<\infty$.
    \vspace{2mm}
    \item[(iii)] $B_{3}\in C^{4}_{\uparrow}(\mathbb R^{p_2})$.
    \end{enumerate}
    \vspace{2mm}
    \item
    \begin{enumerate}
    \item[(i)] There exists a constant $C>0$ such that
    \begin{align*}
    |B_{4}(x)-B_{4}(y)|\leq C|x-y|
    \end{align*}
    for any $x,y\in\mathbb R^{k_{2}}$.
    \vspace{2mm}
    \item[(ii)] For all $\ell\geq 0$, $\displaystyle\sup_t{\bf{E}}\Bigl[\bigl|\zeta_{0,t}\bigr|^{\ell}\Bigr]<\infty$.
    \vspace{2mm}
    \item[(iii)] $B_{4}\in C^{4}_{\uparrow}(\mathbb R^{k_{2}})$.
    \end{enumerate}
    \end{enumerate}
\end{enumerate}
\begin{remark}
For diffusion processes, $[{\bf{A}}]$ is the standard assumption; see, e.g., Kessler \cite{kessler(1997)}. 
\end{remark}
\section{Qaic of sem for diffusion processes}
Using a locally Gaussian approximation, we obtain the following quasi-likelihood of Model $m$ from (\ref{Xm})-(\ref{zetam}):
\begin{align*}
    \prod_{i=1}^n\frac{1}{(2\pi h_n)^{\frac{p}{2}}(\det {\bf{\Sigma}}_m(\theta_m))^{\frac{1}{2}}}\exp\Biggl(-
    \frac{1}{2h_n}{\bf{\Sigma}}_m(\theta_m)^{-1}\Bigl[\bigl(\Delta_i\mathbb{X}^{\theta}_{m}\bigr)^{\otimes 2}\Bigr]\Biggr).
\end{align*}
See Appendix 8.1 in Kusano and Uchida \cite{Kusano(2023)} for details of the quasi-likelihood. Define the quasi-likelihood ${\rm{L}}_{m,n}$ based on the discrete observations $\mathbb{X}_n$ as follows:
\begin{align*}
    {\rm{L}}_{m,n}\bigl(\mathbb{X}_n,\theta_m\bigr)=\prod_{i=1}^n\frac{1}{(2\pi h_n)^{\frac{p}{2}}(\det {\bf{\Sigma}}_m(\theta_m))^{\frac{1}{2}}}\exp\Biggl(-
    \frac{1}{2h_n}{\bf{\Sigma}}_m(\theta_m)^{-1}\Bigl[\bigl(\Delta_i\mathbb{X}\bigr)^{\otimes 2}\Bigr]\Biggr).
\end{align*}
The quasi-maximum likelihood estimator $\hat{\theta}_{m,n}$ is defined by
\begin{align*}
    {\rm{L}}_{m,n}\bigl(\mathbb{X}_n,\hat{\theta}_{m,n}(\mathbb{X}_n)\bigr)=\sup_{\theta_m\in\Theta_m}{\rm{L}}_{m,n}\bigl(\mathbb{X}_n,\theta_m\bigr).
\end{align*}
Set $\mathbb{Z}_n$ as an i.i.d. copy of $\mathbb{X}_n$. Let us consider the following Kullback-Leibler divergence between the transition density $q_n(\mathbb{Z}_n)$ of the true model (\ref{X0})-(\ref{zeta0}) and the quasi-likelihood ${\rm{L}}_{m,n}$:
\begin{align*}
\begin{split}
    {\rm{K_L}}(\mathbb{X}_n,m)&={\bf{E}}_{\mathbb{Z}_n}\Biggl[\log \frac{q_n({\mathbb{Z}}_n)}{{\rm{L}}_{m,n}\bigl(\mathbb{Z}_n,\hat{\theta}_{m,n}({\mathbb{X}_{n}})\bigr)}
    \Biggr]\\
    &={\bf{E}}_{\mathbb{Z}_n}\Bigl[\log q_n(\mathbb{Z}_n)\Bigr]-{\bf{E}}_{\mathbb{Z}_n}\Bigl[\log {\rm{L}}_{m,n}\bigl(\mathbb{Z}_n,\hat{\theta}_{m,n}({\mathbb{X}_{n}})\bigr)\Bigr],
\end{split}
\end{align*}
where ${\bf{E}}_{\mathbb{Z}_n}$ is the expectation under the law of $\mathbb{Z}_n$. Our purpose is to know the model which minimizes ${\rm{K_L}}(\mathbb{X}_n,m)$. Since ${\bf{E}}_{\mathbb{Z}_n}\Bigl[\log q_n(\mathbb{Z}_n)\Bigr]$ does not depend on the model, it is sufficient to consider the model which maximizes
\begin{align}
    {\bf{E}}_{\mathbb{Z}_n}\Bigl[\log {\rm{L}}_{m,n}\bigl(\mathbb{Z}_n,\hat{\theta}_{m,n}({\mathbb{X}_{n}})\bigr)\Bigr],\label{EZL}
\end{align}
so that we need to estimate (\ref{EZL}). Set
\begin{align*}
    \Delta_{m,0}=\left.\frac{\partial}{\partial\theta^{\top}}\vech {\bf{\Sigma}}_m(\theta_m)\right|_{\theta_m=\theta_{m,0}}
\end{align*}
and
\begin{align*}
    {\rm{Y}}_m(\theta_m)=-\frac{1}{2}\Bigl({\bf{\Sigma}}_m(\theta_m)^{-1}-{\bf{\Sigma}}_m(\theta_{m,0})^{-1}\Bigr)\Bigl[{\bf{\Sigma}}_m(\theta_{m,0})\Bigr]-\frac{1}{2}
    \log\frac{\det{\bf{\Sigma}}_m(\theta_m)}{\det{\bf{\Sigma}}_m(\theta_{m,0})}.
\end{align*}
Moreover, the following assumptions are made.
\begin{enumerate}
    \item[\bf{[B1]}]
    \begin{enumerate}
        \item There exists a constant $\chi>0$ such that
        \begin{align*}
         {\rm{Y}}_m(\theta_m)\leq -\chi\bigl|\theta_m-\theta_{m,0}\bigr|^2
        \end{align*}
        for all $\theta_m\in\Theta_m$.
        \vspace{2mm}
        \item $\rank \Delta_{m,0}=q_m$.
    \end{enumerate}
\end{enumerate}
\begin{remark}
$[{\bf{B1}}]\ ({\rm{a}})$ is the identifiability condition. $[{\bf{B1}}]\ ({\rm{b}})$ implies that the asymptotic variance of $\hat{\theta}_{m,n}$ is non-singular; see Lemma 35 in Kusano and Uchida \cite{Kusano(2023)} and Lemma \ref{thetaprob2}.
\end{remark}
By the following theorem, we obtain the asymptotically unbiased estimator of (\ref{EZL}).
\begin{theorem}\label{theorem1}
Let $m\in\{1,\cdots,M\}$. Under {\bf{[A]}} and {\bf{[B1]}}, as $n\longrightarrow\infty$,
\begin{align*}
    {\bf{E}}_{\mathbb{X}_n}\biggl[\log {\rm{L}}_{m,n}\bigl(\mathbb{X}_{n},\hat{\theta}_{m,n}({\mathbb{X}_{n}})\bigr)-{\bf{E}}_{\mathbb{Z}_n}\Bigl[\log {\rm{L}}_{m,n}\bigl(\mathbb{Z}_n,\hat{\theta}_{m,n}({\mathbb{X}_{n}})\bigr)\Bigr]\biggr]=q_m+o_p(1).
\end{align*}
\end{theorem}
We define the quasi-Akaike information criterion as 
\begin{align}
    {\rm{QAIC}}(\mathbb{X}_n,m)=-2\log{\rm{L}}_{m,n}\bigl(\mathbb{X}_n,\hat{\theta}_{m,n}({\mathbb{X}_{n}})\bigr)+2q_m. \label{QAIC}
\end{align}
Since it holds from Theorem \ref{theorem1} that ${\rm{QAIC}}(\mathbb{X}_n,m)$ is the asymptotically unbiased estimator of 
\begin{align*}
    -2{\bf{E}}_{\mathbb{Z}_n}\Bigl[\log {\rm{L}}_{m,n}\bigl(\mathbb{Z}_n,\hat{\theta}_{m,n}({\mathbb{X}_{n}})\bigr)\Bigr],
\end{align*}
we select the optimal model $\hat{m}_n$ among competing models by
\begin{align}
    {\rm{QAIC}}(\mathbb{X}_n,\hat{m}_n)=\min_{m\in\{1,\cdots,M\}}{\rm{QAIC}}(\mathbb{X}_n,m). \label{QAICm}
\end{align}
\begin{remark}
Since ${\rm{L}}_{m,n}$ is not the exact likelihood but the quasi-likelihood, 
all the competing models are misspecified. 
Note that we consider a model selection problem among the quasi-likelihood models; see, e.g., Eguchi and Masuda \cite{Eguchi(2023)}.
\end{remark}
\begin{remark}
In SEM, instead of (\ref{QAIC}), 
\begin{align}
    n{\rm{F}}_{m,n}\bigl(\mathbb{X}_n,\hat{\theta}_{m,n}(\mathbb{X}_n)\bigr)+2q_m \label{SEMQAIC}
\end{align}
is often used for a model selection as $\mathbb{Q}_{\mathbb{XX}}>0$, where 
\begin{align*}
    {\rm{F}}_{m,n}\bigl(\mathbb{X}_n,\theta_m\bigr)=
    \log\det{\bf{\Sigma}}_m(\theta_m)-\log\det\mathbb{Q}_{\mathbb{XX}}+\tr\Bigl({\bf{\Sigma}}_m(\theta_m)^{-1}\mathbb{Q}_{\mathbb{XX}}\Bigr)-p.
\end{align*}
For details of (\ref{SEMQAIC}), see, e.g., Huang \cite{Huang AIC(2017)}. Since
\begin{align*}
    -2\log{\rm{L}}_{m,n}\bigl(\mathbb{X}_n,\hat{\theta}_{m,n}({\mathbb{X}_{n}})\bigr)&=np\log(2\pi h_n)+n\log\det{\bf{\Sigma}}_m(\hat{\theta}_{m,n})+n\tr\Bigl({\bf{\Sigma}}_m(\hat{\theta}_{m,n})^{-1}\mathbb{Q}_{\mathbb{XX}}\Bigr)\\
    &=n{\rm{F}}_{m,n}\bigl(\mathbb{X}_n,\hat{\theta}_{m,n}(\mathbb{X}_n)\bigr)+n\Bigl\{p\log(2\pi h_n)+\log\det\mathbb{Q}_{\mathbb{XX}}+p\Bigr\}
\end{align*}
as $\mathbb{Q}_{\mathbb{XX}}>0$, it is shown that
\begin{align*}
    n{\rm{F}}_{m,n}\bigl(\mathbb{X}_n,\hat{\theta}_{m,n}(\mathbb{X}_n)\bigr)+2q_m={\rm{QAIC}}(\mathbb{X}_n,m)-n\Bigl\{p\log(2\pi h_n)+\log\det\mathbb{Q}_{\mathbb{XX}}+p\Bigr\}
\end{align*}
as $\mathbb{Q}_{\mathbb{XX}}>0$. Note that 
\begin{align*}
    n\Bigl\{p\log(2\pi h_n)+\log\det\mathbb{Q}_{\mathbb{XX}}+p\Bigr\}
\end{align*}
does not depend on the model. 
Even if we use (\ref{SEMQAIC}) instead of (\ref{QAIC}), the model selection results are not different.
\end{remark}
Next, we consider the situation where the set of competing models includes some (not all) misspecified parametric models; that is, there exists $m\in\{1,\cdots,M\}$ such that 
\begin{align*}
    {\bf{\Sigma}}_0\neq {\bf{\Sigma}}_{m}(\theta_{m})
\end{align*}
for all $\theta_{m}\in\Theta_{m}$. Set
\begin{align*}
    \mathcal{M}=\biggl\{m\in\{1,\cdots,M\}\ \Big|\ \mbox{There exists}\ \theta_{m,0}\in\Theta_{m}\ \mbox{such that}\ {\bf{\Sigma}}_0={\bf{\Sigma}}_{m}(\theta_{m,0}).\biggr\}
\end{align*}
and $\mathcal{M}^{c}=\{1,\cdots,M\}	\backslash \mathcal{M}$. The optimal parameter $\bar{\theta}_m$ is defined as  
\begin{align*}
    {\rm{H}}_{m}(\bar{\theta}_m)=\sup_{\theta_m\in\Theta_m}{\rm{H}}_{m}(\theta_m),
\end{align*}
where
\begin{align*}
    {\rm{H}}_{m}(\theta_m)=-\frac{1}{2}\tr\Bigl({\bf{\Sigma}}_{m}
    (\theta_{m})^{-1}{\bf{\Sigma}}_0\Bigr)-\frac{1}{2}\log\det {\bf{\Sigma}}_{m}(\theta_{m}).
\end{align*}
Note that $\bar{\theta}_m=\theta_{m,0}$ for $m\in\mathcal{M}$. Furthermore, we make the following assumption.
\vspace{2mm}
\begin{enumerate}
    \item[\bf{[B2]}] ${\rm{H}}_{m}(\theta_m)={\rm{H}}_{m}(\bar{\theta}_m)\Longrightarrow\theta_m=\bar{\theta}_m$.
    \vspace{2mm}
\end{enumerate}
$\bf{[B2]}$ implies that $\hat{\theta}_{m,n}\stackrel{p}{\longrightarrow}\bar{\theta}_{m}$; see, Lemma 36 in Kusano and Uchida \cite{Kusano(2023)}. The following asymptotic result of $\hat{m}_n$ defined in (\ref{QAICm}) holds.
\begin{theorem}\label{misstheorem}
Under {\bf{[A]}} and {\bf{[B2]}}, as $n\longrightarrow\infty$,
\begin{align*}
    {\bf{P}}\Bigl(\hat{m}_{n}\in\mathcal{M}^{c}\Bigr)\longrightarrow 0. 
\end{align*}    
\end{theorem}
Theorem \ref{misstheorem} shows that the probability of choosing the misspecified models converges to zero as $n\longrightarrow\infty$.
\section{Simulation results}
\subsection{True model}
The stochastic process $\mathbb{X}_{1,0,t}$ is defined by the following factor model
:
\begin{align*}
    \mathbb{X}_{1,0,t}=\begin{pmatrix}
    1 & 5 & 2 & 0 & 0 & 0\\
    0 & 0 & 0 & 1 & 4 & 7
    \end{pmatrix}^{\top}
    \xi_{0,t}+\delta_{0,t},
\end{align*}
where $\{\mathbb{X}_{1,0,t}\}_{t\geq 0}$ is a six-dimensional observable vector process, $\{\xi_{0,t}\}_{t\geq 0}$ is a two-dimensional latent common factor vector process, and $\{\delta_{0,t}\}_{t\geq 0}$ is a six-dimensional latent unique factor vector process. The stochastic process $\mathbb{X}_{2,0,t}$ is 
defined by the factor model as follows:
\begin{align*}
    \mathbb{X}_{2,0,t}=\begin{pmatrix}
    1\\
    2
    \end{pmatrix}\eta_{0,t}+\varepsilon_{0,t},
\end{align*}
where $\{\mathbb{X}_{2,0,t}\}_{t\geq 0}$ is a two-dimensional observable vector process, $\{\eta_{0,t}\}_{t\geq 0}$ is a one-dimensional latent common factor vector process, and $\{\varepsilon_{0,t}\}_{t\geq 0}$ is a two-dimensional latent unique factor vector process. Furthermore, the relationship between $\eta_{0,t}$ and  $\xi_{0,t}$ is expressed as follows:
\begin{align*}
    \eta_{0,t}=\begin{pmatrix}
    3 & 2
    \end{pmatrix}\xi_{0,t}+\zeta_{0,t},
\end{align*}
where $\{\zeta_{0,t}\}_{t\geq 0}$ is a one-dimensional latent unique factor vector process. 
It is supposed that $\{\xi_{0,t}\}_{t\geq 0}$ is the two-dimensional OU process as follows:
\begin{align*}
    \quad\dd \xi_{0,t}=-\left\{\begin{pmatrix}
    1 & 0.7\\
    0.7 & 0.5
    \end{pmatrix}\xi_{0,t}-\begin{pmatrix}
    1\\
    2
    \end{pmatrix}\right\}\dd t+\begin{pmatrix}
    1 & 0.3\\
    0.4 & 1
    \end{pmatrix}\dd W_{1,t}\ \ (t\in [0,T]),\ \
    \xi_{0,0}=\begin{pmatrix}
    2\\
    1
    \end{pmatrix},
\end{align*}
where $W_{1,t}$ is a two-dimensional standard Wiener process. $\{\delta_{0,t}\}_{t\geq 0}$ is defined by the six-dimensional OU process as follows:
\begin{align*}
    \quad\dd\delta_{0,t}=-\bigl(B_0\delta_{0,t}-\mu_0\bigr)
    dt+{\bf{S}}_{2,0}\dd W_{2,t}\ \ (t\in [0,T]),\ \ 
    \delta_{0,0}=c_0,
\end{align*}
where $B_0={\rm{Diag}}(3,2,4,1,2,1)^{\top}$, $\mu_0=(3,2,1,2,6,4)^{\top}$, ${\bf{S}}_{2,0}={\rm{Diag}}(3,2,1,2,1,3)^{\top}$, $c_0=(1,3,2,1,4,3)^{\top}$ and $W_{2,t}$ is a six-dimensional standard Wiener process. $\{\varepsilon_{0,t}\}_{t\geq 0}$ 
satisfies the following two-dimensional OU process:
\begin{align*}
    \quad\dd\varepsilon_{0,t}=-\left\{\begin{pmatrix}
    1 & 0 \\
    0 & 3
    \end{pmatrix}\varepsilon_{0,t}-\begin{pmatrix}
    2\\
    3
    \end{pmatrix}\right\}\dd t+\begin{pmatrix}
    1 & 0\\
    0 & 2
    \end{pmatrix}\dd W_{3,t}\ \ (t\in [0,T]),\ \ \varepsilon_{0,0}=\begin{pmatrix}
    1\\
    5
    \end{pmatrix},
\end{align*}
where $W_{3,t}$ is a two-dimensional standard Wiener process. $\{\zeta_{0,t}\}_{t\geq 0}$ is defined by the following one-dimensional OU process:
\begin{align*}
    \quad\dd\zeta_{0,t}=-\zeta_{0,t}\dd t+2\dd W_{4,t}\ \ (t\in [0,T]),\ \ 
    \zeta_{0,0}=0,
\end{align*}
where $W_{4,t}$ is a one-dimensional standard Wiener process. Figure \ref{truemodel} shows the path diagram of the true model at time $t$.
\begin{figure}[h]
    \includegraphics[width=0.9\columnwidth]{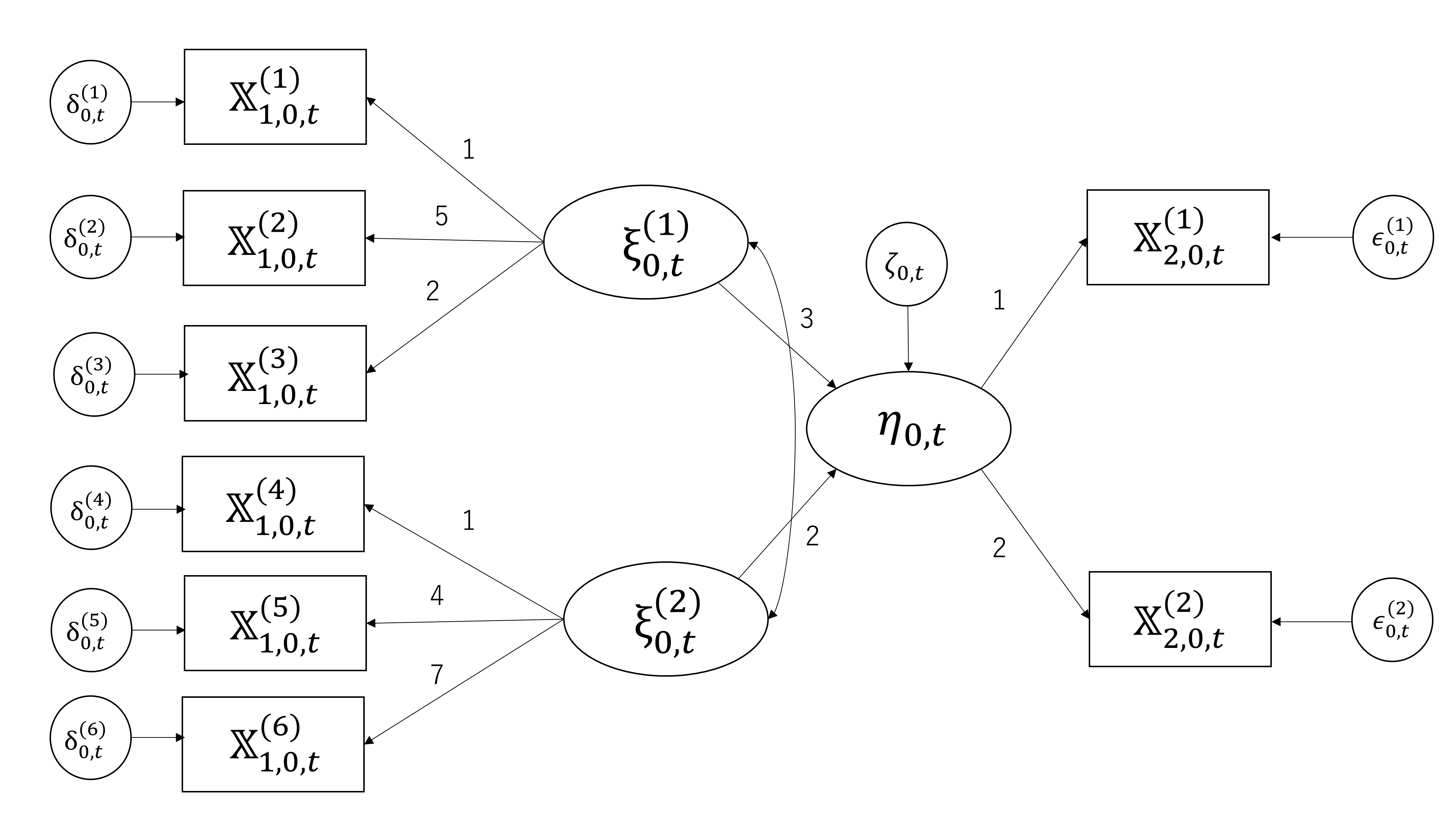}
    \caption{The path diagram of the true model at time $t$.}\label{truemodel}
\end{figure}
\subsection{Competing models}
\subsubsection{Model 1}
Set the parameter as $\theta_1\in\Theta_1\subset\mathbb{R}^{19}$. Let $p_1=6$, $p_2=2$, $k_{1}=2$ and $k_{2}=1$. Assume
\begin{align*}
    {\bf{\Lambda}}^{\theta}_{1,x_1}&=\begin{pmatrix}
    1 & \theta^{(1)}_1 & 
    \theta^{(2)}_1 & 0 & 0 & 0\\
    0 & 0 & 0 & 1 & \theta^{(3)}_1
    & \theta^{(4)}_1
    \end{pmatrix}^{\top}
\end{align*}
and
\begin{align*}
    {\bf{\Lambda}}^{\theta}_{1,x_2}=\begin{pmatrix}
    1 &
    \theta_1^{(5)}
    \end{pmatrix}^{\top},\quad
    {\bf{\Gamma}}^{\theta}_{1}=\begin{pmatrix}
    \theta_1^{(6)} & \theta_1^{(7)}
    \end{pmatrix},
\end{align*}
where $\theta_1^{(1)}$, $\theta_1^{(2)}$, $\theta_1^{(3)}$, $\theta_1^{(4)}$, $\theta_1^{(5)}$, $\theta_1^{(6)}$ and $\theta_1^{(7)}$ are not zero. 
It is supposed that ${\bf{S}}^{\theta}_{1,1}$ and ${\bf{S}}^{\theta}_{2,1}$ satisfy
\begin{align*}
    &\qquad\qquad\qquad{\bf{\Sigma}}^{\theta}_{1,\xi\xi}={\bf{S}}^{\theta}_{1,1}{\bf{S}}^{\theta\top}_{1,1}=\begin{pmatrix}
    \theta_1^{(8)} & \theta_1^{(9)}\\
    \theta_1^{(9)} & \theta_1^{(10)}
    \end{pmatrix}\in\mathcal{M}_2^{++}
\end{align*}
and
\begin{align*}
    &{\bf{\Sigma}}^{\theta}_{1,\delta\delta}={\bf{S}}^{\theta}_{2,1}{\bf{S}}^{\theta\top}_{2,1}={\rm{Diag}}\Bigl(\theta_1^{(11)},\theta_1^{(12)},
    \theta_1^{(13)},\theta_1^{(14)},\theta_1^{(15)},\theta_1^{(16)}
    \Bigr)\in\mathcal{M}_6^{++},
\end{align*}
where $\theta_1^{(9)}$ is not zero. Moreover, ${\bf{S}}^{\theta}_{3,1}$ and ${\bf{S}}^{\theta}_{4,1}$ are assumed to satisfy
\begin{align*}
    &\qquad\qquad\qquad{\bf{\Sigma}}^{\theta}_{1,\varepsilon\varepsilon}={\bf{S}}^{\theta}_{3,1}{\bf{S}}^{\theta\top}_{3,1}=\begin{pmatrix}
    \theta_1^{(17)} & 0\\
    0 & \theta_1^{(18)}
    \end{pmatrix}\in\mathcal{M}_2^{++}
\end{align*}
and ${\bf{\Sigma}}^{\theta}_{1,\zeta\zeta}=({\bf{S}}^{\theta }_{4,1})^2=\theta_1^{(19)}>0$. Set 
\begin{align*}
    \theta_{1,0}=\Bigl(5,2,4,7,2,3,2,1.09,0.70,1.16,9,4,1,4,1,9,1,4,4\Bigr).
\end{align*}
It holds that ${\bf{\Sigma}}_0={\bf{\Sigma}}_1(\theta_{1,0})$, so that Model $1$ is a correctly specified model. There exists a constant $\chi>0$ such that
\begin{align}
    {\rm{Y}}_1(\theta_1)\leq -\chi|\theta_1-\theta_{1,0}|^2 \label{Y1iden}
\end{align}
for all $\theta_1\in\Theta_1$. For the proof of (\ref{Y1iden}), see Appendix \ref{iden}.
Figure \ref{Model1} shows the path diagram of Model $1$ at time $t$.
\begin{figure}[h]
    \includegraphics[width=0.9\columnwidth]{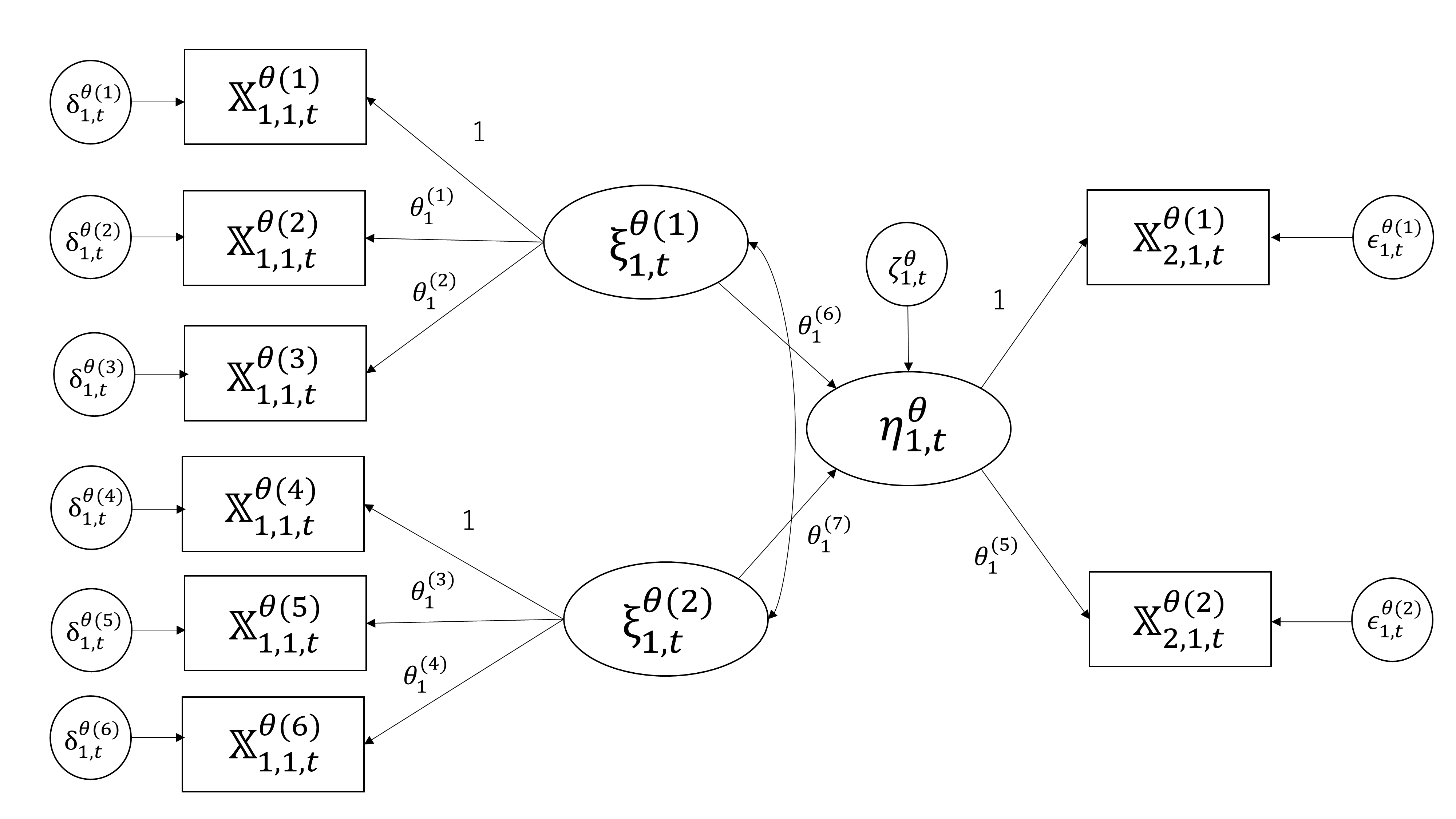}
    \caption{The path diagram of Model 1.} \label{Model1}
\end{figure}
\subsubsection{Model 2}
The parameter is defined as $\theta_2\in\Theta_2\subset\mathbb{R}^{20}$. Set $p_1=6$, $p_2=2$, $k_{1}=2$ and $k_{2}=1$. Suppose
\begin{align*}
    {\bf{\Lambda}}^{\theta}_{2,x_1}&=\begin{pmatrix}
    1 &
    \theta^{(1)}_2 &
    \theta^{(2)}_2 &
    0 &
    0 & 0\\
    0 & 0 & \theta^{(3)}_2 & 1 
    & \theta^{(4)}_2
    & \theta^{(5)}_2
    \end{pmatrix}^{\top}
\end{align*}
and
\begin{align*}
    {\bf{\Lambda}}^{\theta}_{2,x_2}=\begin{pmatrix}
    1 & \theta_2^{(6)}
    \end{pmatrix}^{\top},\quad {\bf{\Gamma}}^{\theta}_{2}=\begin{pmatrix}
    \theta_2^{(7)} & \theta_2^{(8)}
    \end{pmatrix},
\end{align*}
where $\theta_2^{(1)}$, $\theta_2^{(2)}$, $\theta_2^{(3)}$, $\theta_2^{(4)}$, $\theta_2^{(5)}$, $\theta_2^{(6)}$, $\theta_2^{(7)}$ and $\theta_2^{(8)}$ are not zero. ${\bf{S}}^{\theta}_{1,2}$ and ${\bf{S}}^{\theta}_{2,2}$ are assumed to satisfy
\begin{align*}
    {\bf{\Sigma}}^{\theta}_{2,\xi\xi}&={\bf{S}}^{\theta}_{1,2}{\bf{S}}^{\theta\top}_{1,2}=\begin{pmatrix}
    \theta_2^{(9)} & \theta_2^{(10)}\\
    \theta_2^{(10)} & \theta_2^{(11)}
    \end{pmatrix}\in\mathcal{M}_2^{++}
\end{align*}
and
\begin{align*}
    {\bf{\Sigma}}^{\theta}_{2,\delta\delta}&={\bf{S}}^{\theta}_{2,2}{\bf{S}}^{\theta\top}_{2,2}={\rm{Diag}}\Bigl(\theta_2^{(12)},\theta_2^{(13)},
    \theta_2^{(14)},\theta_2^{(15)},\theta_2^{(16)},\theta_2^{(17)}
    \Bigr)\in\mathcal{M}_6^{++},
\end{align*}
where $\theta_2^{(10)}$ is not zero. Furthermore, we suppose that ${\bf{S}}^{\theta}_{3,2}$ and ${\bf{S}}^{\theta}_{4,2}$ satisfy
\begin{align*}
    {\bf{\Sigma}}^{\theta}_{2,\varepsilon\varepsilon}&={\bf{S}}^{\theta}_{3,2}{\bf{S}}^{\theta\top}_{3,2}=\begin{pmatrix}
    \theta_2^{(18)} & 0\\
    0 & \theta_2^{(19)}
    \end{pmatrix}\in\mathcal{M}_2^{++}
\end{align*}
and ${\bf{\Sigma}}^{\theta}_{2,\zeta\zeta}=({\bf{S}}^{\theta}_{4,2})^2=\theta_2^{(20)}>0$. Let 
\begin{align*}
    \theta_{2,0}=\Bigl(5,2,0,4,7,2,3,2,1.09,0.70,1.16,9,4,1,4,1,9,1,4,4\Bigr).
\end{align*}
Since ${\bf{\Sigma}}_0={\bf{\Sigma}}_2(\theta_{2,0})$, Model $2$ is a correctly specified model. In a similar way to the proof of (\ref{Y1iden}), we can prove that there exists a constant $\chi>0$ such that
\begin{align*}
    {\rm{Y}}_2(\theta_2)\leq -\chi|\theta_2-\theta_{2,0}|^2
\end{align*}
for all $\theta_2\in\Theta_2$. Figure \ref{Model2} shows the path diagram of Model $2$ at time $t$.
\begin{figure}[h]
    \includegraphics[width=0.9\columnwidth]{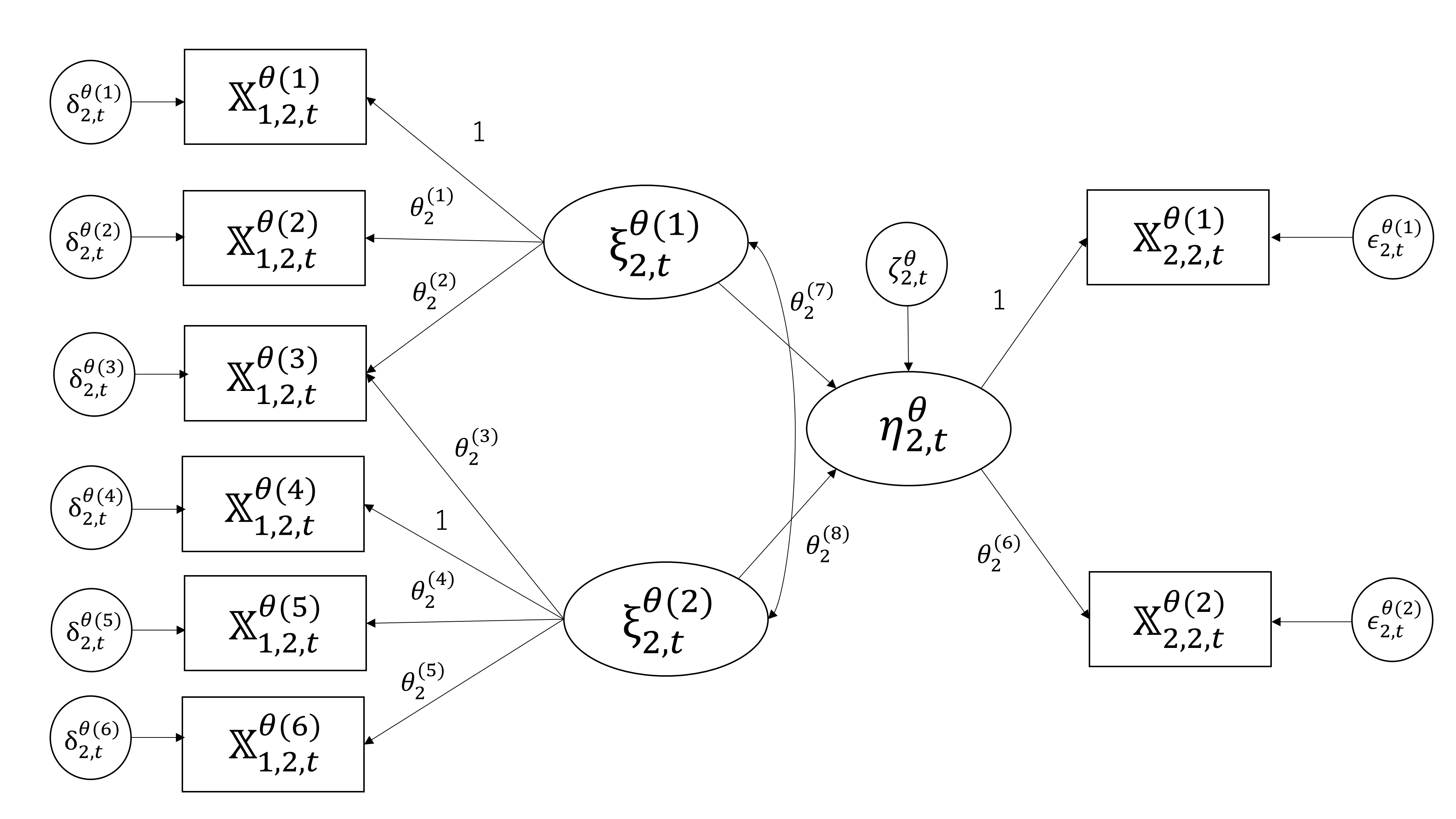}
    \caption{The path diagram of Model 2.}\label{Model2}
\end{figure}
\subsubsection{Model 3} 
Set the parameter as $\theta_3\in\Theta_3\subset\mathbb{R}^{17}$. Let $p_1=6$, $p_2=2$, $k_{1}=1$ and $k_{2}=1$. Assume 
\begin{align*}
    {\bf{\Lambda}}^{\theta}_{3,x_1}&=\begin{pmatrix}
    1 &
    \theta^{(1)}_3 &
    \theta^{(2)}_3 &
    \theta^{(3)}_3 &
    \theta^{(4)}_3 &
    \theta^{(5)}_3 
    \end{pmatrix}^{\top}
\end{align*}
and
\begin{align*}
    {\bf{\Lambda}}^{\theta}_{3,x_2}=\begin{pmatrix}
    1 &
    \theta_3^{(6)}
    \end{pmatrix}^{\top},\quad
    {\bf{\Gamma}}^{\theta}_{3}=
    \theta_3^{(7)},
\end{align*}
where $\theta^{(1)}_3$, $\theta^{(2)}_3$, $\theta^{(3)}_3$, $\theta^{(4)}_3$, $\theta^{(5)}_3$, $\theta^{(6)}_3$ and $\theta^{(7)}_3$ are not zero. We assume that ${\bf{S}}^{\theta}_{1,3}$ and ${\bf{S}}^{\theta}_{2,3}$ satisfy ${\bf{\Sigma}}^{\theta}_{3,\xi\xi}=({\bf{S}}^{\theta}_{1,3})^2=\theta_3^{(8)}>0$ and
\begin{align*}
    {\bf{\Sigma}}^{\theta}_{3,\delta\delta}&={\bf{S}}^{\theta}_{2,3}{\bf{S}}^{\theta\top}_{2,3}
    ={\rm{Diag}}\Bigl(\theta_3^{(9)},\theta_3^{(10)},
    \theta_3^{(11)},\theta_3^{(12)},\theta_3^{(13)},\theta_3^{(14)}
    \Bigr)\in\mathcal{M}_6^{++}.
\end{align*}
Moreover, it is supposed that ${\bf{S}}^{\theta}_{3,3}$ and ${\bf{S}}^{\theta}_{4,3}$ satisfy
\begin{align*}
    {\bf{\Sigma}}^{\theta}_{3,\varepsilon\varepsilon}&={\bf{S}}^{\theta}_{3,3}{\bf{S}}^{\theta\top}_{3,3}=\begin{pmatrix}
    \theta_3^{(15)} & 0\\
    0 & \theta_3^{(16)}
    \end{pmatrix}\in\mathcal{M}_2^{++}
\end{align*}
and ${\bf{\Sigma}}^{\theta}_{3,\zeta\zeta}=({\bf{S}}^{\theta}_{4,3})^2=\theta_3^{(17)}>0$. For any $\theta_3\in\Theta_3$,
one has ${\bf{\Sigma}}_0\neq{\bf{\Sigma}}_3(\theta_{3})$, so that Model $3$ is a misspecified model. Figure \ref{Model3} shows the path diagram of Model $3$ at time $t$.
\begin{figure}[h]
    \includegraphics[width=0.9\columnwidth]{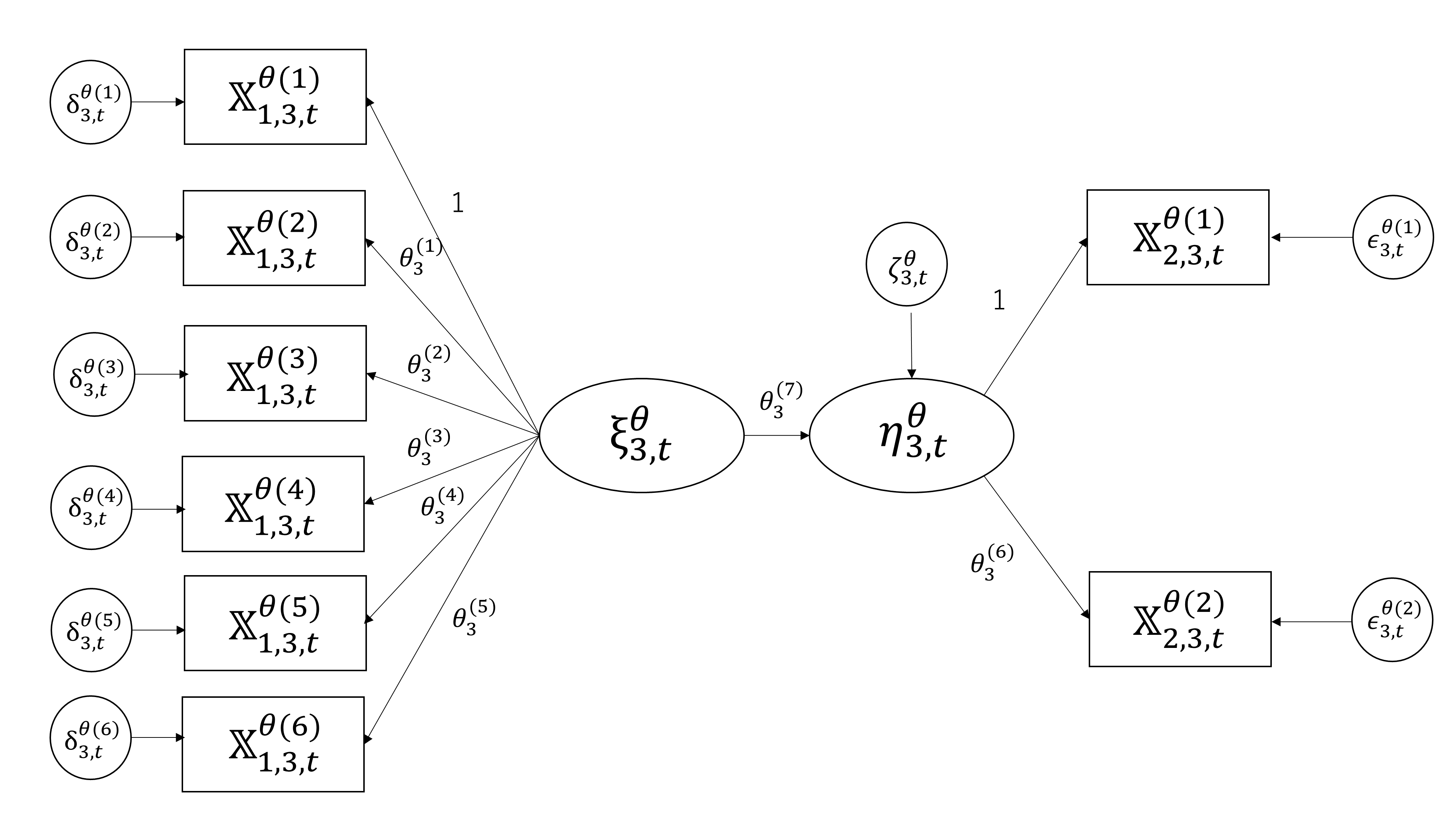}
    \caption{The path diagram of Model 3.}\label{Model3}
\end{figure}
\subsection{Simulation results}
In the simulation, we use optim() with the BFGS method in R language. The initial parameter is chosen as $\theta=\theta_0$. The number of iterations is 10000. Set $T=1$ and consider the case where $n=10^2, 10^3, 10^4, 10^5$. Table \ref{table} shows the number of models selected by QAIC. Since Model 3 is not selected, this simulation result implies that Theorem \ref{misstheorem} seems to be correct in this example. Furthermore, we see from this result that QAIC does not have consistency. In other words, the over-fitted model (Model 2) is selected with significant probability. This result is natural since QAIC chooses the best model in terms of prediction.
}
\begin{table}[h]
    \ \\ \ \\ \ \\ \ \\ \ \\ \ \\
    \centering
    \begin{tabular}{ccccc}
    & $n=10^2$ & $n=10^3$ & $n=10^4$ & $n=10^5$ \\\hline
    Model 1 & 8394 & 8417 & 8461 & 8410 \\
    Model 2 & 1606 & 1583 & 1539 & 1590 \\
    Model 3 & 0 & 0 & 0 & 0
\end{tabular}
    \caption{The number of models selected by QAIC.}\label{table}
\end{table}
\newpage
\clearpage
\section{proof}
In this section, we may omit the model index ``$m$", 
and we use $\hat{\theta}_{n}$ instead of $\hat{\theta}_{n}(\mathbb{X}_n)$. 
Moreover, we simply write $\mathbb{X}_{1,0,t}$, $\mathbb{X}_{2,0,t}$, $\xi_{0,t}$, $\delta_{0,t}$, $\varepsilon_{0,t}$ and $\zeta_{0,t}$ as $\mathbb{X}_{1,t}$, $\mathbb{X}_{2,t}$, $\xi_{t}$, $\delta_{t}$, $\varepsilon_{t}$ and $\zeta_{t}$, respectively. For any process $Y_t$ and $\ell\geq 0$, 
we set ${\rm{R}}_{i}(h_n^{\ell},Y)={\rm{R}}(h_n^{\ell},Y_{t_{i}^n})$. Without loss of generality, we suppose that $T=1$. Set
\begin{align*}
    \mathscr{F}^{n}_{i}
    =\sigma\bigl(W_{1,s},W_{2,s},W_{3,s},W_{4,s},s\leq t_{i}^n\bigr)
\end{align*}
for $i=0,\cdots, n$. Let
\begin{align*}
    {\rm{H}}_n(\mathbb{X}_n,\theta)&=\log(2\pi h_n)^{\frac{np}{2}}{\rm{L}}_n(\mathbb{X}_n,\theta).
\end{align*}
Set ${\bf{I}}(\theta_0)=\Delta^{\top}_{0}{\bf{W}}(\theta_0)^{-1}\Delta_{0}$, where 
\begin{align*}
    {\bf{W}}(\theta_0)=2\mathbb{D}^{+}_{p}\bigl({\bf{\Sigma}}(\theta_0)\otimes{\bf{\Sigma}}(\theta_0)\bigr)\mathbb{D}^{+\top}_{p}.
\end{align*}
Define the random field ${\rm{Y}}_n:$
\begin{align*}
    {\rm{Y}}_n(\mathbb{X}_n,\theta;\theta_0)=\frac{1}{n}\biggl\{{\rm{H}}_n(\mathbb{X}_n,\theta)-{\rm{H}}_n(\mathbb{X}_n,\theta_0)\biggr\}.
\end{align*}
Let ${\rm{Z}}_n$ be the random field as follows:
\begin{align*}
    {\rm{Z}}_n(\mathbb{X}_n,u;\theta_0)=\exp\Biggl\{{\rm{H}}_n\biggl(\mathbb{X}_n,\theta_0+\frac{1}{\sqrt{n}}u\biggr)-{\rm{H}}_n(\mathbb{X}_n,\theta_0)\Biggr\}
\end{align*}
for $u\in\mathbb{U}_n$, where
\begin{align*}
    {\mathbb{U}}_{n}=\left\{u\in\mathbb{R}^{q}:\  \theta_{0}+\frac{1}{\sqrt{n}}u\in\Theta\right\}.
\end{align*}
Set ${\rm{V}}_{n}(r)=\bigl\{u\in{\mathbb{U}}_{n}: r\leq |u|\bigr\}$ and $\hat{u}_n=\sqrt{n}(\hat{\theta}_n-\theta_0)$. ${\bf{V}}_{\mathbb{X}_n}$ denotes the variance under the law of $\mathbb{X}_n$. Write $\partial_{\theta}=\partial/\partial\theta$ and $\partial^2_{\theta}=\partial_{\theta}\partial^{\top}_{\theta}$. Define $\zeta$ as a $q$-dimensional standard normal random variable.
\begin{lemma}\label{thetaprob1}
Under {\bf{[A]}}, as $n\longrightarrow\infty$,
\begin{align*}
    \frac{1}{\sqrt{n}}\partial_{\theta}{\rm{H}}_n(\mathbb{X}_n,\theta_0) 
    &\stackrel{d}{\longrightarrow}{\bf{I}}(\theta_{0})^{\frac{1}{2}}\zeta
\end{align*}
and
\begin{align*}
    \quad\frac{1}{n}\partial^2_{\theta}{\rm{H}}_n(\mathbb{X}_n,\theta_0)&\stackrel{p}{\longrightarrow}-{\bf{I}}(\theta_0). 
\end{align*}
\end{lemma}
\begin{lemma}\label{thetaprob2}
Under {\bf{[A]}} and {\bf{[B1]}}, as $n\longrightarrow\infty$,
\begin{align*}
    \hat{\theta}_{n}\stackrel{p}{\longrightarrow}\theta_0 
\end{align*}
and
\begin{align*}
    \sqrt{n}(\hat{\theta}_{n}-\theta_{0})&\stackrel{d}{\longrightarrow}{\bf{I}}(\theta_{0})^{-\frac{1}{2}}\zeta. 
\end{align*} 
\end{lemma}
\begin{proof}[Proofs of Lemmas \ref{thetaprob1}-\ref{thetaprob2}]
    In the same way as the proof of Theorem 2 in Kusano and Uchida \cite{Kusano(JJSD)}, we can prove the results. See also Appendix \ref{appendix lemma}.
\end{proof}
In the proofs of Lemmas \ref{EXVX}-\ref{H3lemma}, 
we simply write ${\bf{E}}_{\mathbb{X}_n}$, ${\bf{V}}_{\mathbb{X}_n}$, $\mathbb{H}_n(\mathbb{X}_n,\theta)$, ${\rm{Y}}_n(\mathbb{X}_n,\theta;\theta_0)$ and ${\rm{Z}}_n(\mathbb{X}_n,u;\theta_0)$ as ${\bf{E}}$, ${\bf{V}}$, $\mathbb{H}_n(\theta)$, ${\rm{Y}}_n(\theta;\theta_0)$ and ${\rm{Z}}_n(u;\theta_0)$, respectively. 
\begin{lemma}\label{EXVX}
Under {[\bf{A}]}, for all $L>1$,
\begin{align}
    &\qquad\qquad{\bf{E}}_{\mathbb{X}_n}\Biggl[\Bigl|\mathbb{X}_{t_i^n}-\mathbb{X}_{t_{i-1}^n}\Bigr|^{L}\Biggr]\leq C_L h_n^{\frac{L}{2}}, \label{Xine}\\
    &{\bf{E}}_{\mathbb{X}_n}\Biggl[\Bigl|\mathbb{X}_{t_i^n}-{\bf{E}}_{\mathbb{X}_n}\Bigl[\mathbb{X}_{t_{i}^n}\big|\mathscr{F}^{n}_{i-1}\Bigr]\Bigr|^{L}\Biggr]\leq C_L h_n^{\frac{L}{2}}, \label{EXine}\\
    &{\bf{E}}_{\mathbb{X}_n}\Biggl[\Bigl|{\bf{E}}_{\mathbb{X}_n}\Bigl[\mathbb{X}_{t_{i}^n}\big|\mathscr{F}^{n}_{i-1}\Bigr]-\mathbb{X}_{t_{i-1}^n}\Bigr|^{L}\Biggr]\leq C_L h_n^L \label{EXine2}
\end{align}
and
\begin{align}
    {\bf{E}}_{\mathbb{X}_n}\Biggl[\Bigl|{\bf{V}}_{\mathbb{X}_n}
    \Bigl[\mathbb{X}_{t_{i}^n}\big|\mathscr{F}^{n}_{i-1}\Bigr]\Bigr|^L\Biggr]
    \leq C_Lh_n^L. \label{Vine}
\end{align}
\end{lemma}
\begin{proof}
First, we will prove (\ref{Xine}). Lemmas 14-15 in Kusano and Uchida \cite{Kusano(2023)} implies
\begin{align}
\begin{split}
    {\bf{E}}
    \biggl[\Delta_i\mathbb{X}^{(j)}_{1}\big|\mathscr{F}^{n}_{i-1}\biggr]
    &={\bf{E}}
    \biggl[A^{(j)}_{i,n}\big|\mathscr{F}^{n}_{i-1}\biggr]+{\bf{E}}
    \biggl[B^{(j)}_{i,n}\big|\mathscr{F}^{n}_{i-1}\biggr]\\
    &={\rm{R}}_{i-1}(h_n,\xi)+{\rm{R}}_{i-1}(h_n,\delta)
\end{split}\label{X1}
\end{align}
for $j=1,\cdots,p_1$, where 
\begin{align*}
    A_{i,n}={\bf{\Lambda}}_{x_1,0}\Delta_i\xi,\quad B_{i,n}=\Delta_i\delta.
\end{align*}
Since  
\begin{align*}
    \mathbb{X}_{2,t}&={\bf{\Lambda}}_{x_2,0}{\bf{\Psi}}_0^{-1}{\bf{\Gamma}}_0\xi_{t}+{\bf{\Lambda}}_{x_2,0}{\bf{\Psi}}_0^{-1}\zeta_{t}+\varepsilon_{t},
\end{align*}
it follows from Lemmas 16-18 in Kusano and Uchida \cite{Kusano(2023)} that
\begin{align}
\begin{split}
    {\bf{E}}
    \biggl[\Delta_i\mathbb{X}^{(k)}_{2}\big|\mathscr{F}^{n}_{i-1}\biggr]
    &={\bf{E}}
    \biggl[C^{(k)}_{i,n}\big|\mathscr{F}^{n}_{i-1}\biggr]+{\bf{E}}
    \biggl[D^{(k)}_{i,n}\big|\mathscr{F}^{n}_{i-1}\biggr]+{\bf{E}}
    \biggl[E^{(k)}_{i,n}\big|\mathscr{F}^{n}_{i-1}\biggr]\\
    &={\rm{R}}_{i-1}(h_n,\xi)+{\rm{R}}_{i-1}(h_n,\varepsilon)+{\rm{R}}_{i-1}(h_n,\zeta)
\end{split}\label{X2}
\end{align}
for $k=1,\cdots,p_2$, where
\begin{align*}
    C_{i,n}={\bf{\Lambda}}_{x_2,0}{\bf{\Psi}}_0^{-1}{\bf{\Gamma}}_0\Delta_i\xi,\quad
    D_{i,n}={\bf{\Lambda}}_{x_2,0}{\bf{\Psi}}_0^{-1}\Delta_i\zeta,\quad
    E_{i,n}=\Delta_i\varepsilon.
\end{align*}
Lemma 20 in Kusano and Uchida \cite{Kusano(2023)} shows
\begin{align*}  
    {\bf{E}}\biggl[\bigl|\Delta_i\mathbb{X}^{(j)}_{1}\bigr|^{L}\big|\mathscr{F}^{n}_{i-1}\biggr]
    &\leq C_L{\bf{E}}\biggl[\bigl|A^{(j)}_{i,n}\bigr|^L\big|\mathscr{F}^{n}_{i-1}\biggr]+C_L{\bf{E}}\biggl[\bigl|B^{(j)}_{i,n}\bigr|^L\big|\mathscr{F}^{n}_{i-1}\biggr]\\
    &\leq C_L{\rm{R}}_{i-1}(h_n^{\frac{L}{2}},\xi)+C_L{\rm{R}}_{i-1}(h_n^{\frac{L}{2}},\delta)
\end{align*}
for all $L>1$, so that
\begin{align}
\begin{split}
    {\bf{E}}\biggl[\bigl|\Delta_i\mathbb{X}^{(j)}_{1}\bigr|^{L}\biggr]&={\bf{E}}\Biggl[
    {\bf{E}}\biggl[\bigl|\Delta_i\mathbb{X}^{(j)}_{1}\bigr|^{L}\big|\mathscr{F}^{n}_{i-1}\biggr]
    \Biggr]\\
    &\leq C_L{\bf{E}}\biggl[{\rm{R}}_{i-1}(h_n^{\frac{L}{2}},\xi)\biggr]+C_L{\bf{E}}\biggl[{\rm{R}}_{i-1}(h_n^{\frac{L}{2}},\delta)\biggr]\leq C_Lh_n^{\frac{L}{2}}.\label{EX1L}
\end{split}
\end{align}
Similarly, we see from Lemma 20 in Kusano and Uchida \cite{Kusano(2023)} that
\begin{align}
\begin{split}
    {\bf{E}}\biggl[\bigl|\Delta_i\mathbb{X}^{(k)}_{2}\bigr|^{L}\biggr]\leq C_Lh_n^{\frac{L}{2}} \label{EX2L}
\end{split}
\end{align}
for any $L>1$. Thus, it holds from (\ref{EX1L}) and (\ref{EX2L}) that for all $L>1$,
\begin{align*}
   {\bf{E}}\biggl[\bigl|\Delta_i\mathbb{X}\bigr|^{L}\biggr]
   &\leq C_L\sum_{\ell=1}^p
   {\bf{E}}\biggl[\bigl|\Delta_i\mathbb{X}^{(\ell)}\bigr|^{L}\biggr]\\
   &=C_L\sum_{j=1}^{p_1}{\bf{E}}
   \biggl[\bigl|\Delta_i\mathbb{X}^{(j)}_{1}\bigr|^{L}\biggr]
   +C_L\sum_{k=1}^{p_2}{\bf{E}}
   \biggl[\bigl|\Delta_i\mathbb{X}^{(k)}_{2}\bigr|^{L}\biggr]\leq C_Lh_n^{\frac{L}{2}},
\end{align*}
which yields (\ref{Xine}). Using (\ref{X1}) and (\ref{EX1L}), one gets
\begin{align}
\begin{split}
    &\quad\ {\bf{E}}\Biggl[\Bigl|\mathbb{X}^{(j)}_{1,t_i^n}-{\bf{E}}\Bigl[\mathbb{X}^{(j)}_{1,t_{i}^n}\big|\mathscr{F}^{n}_{i-1}\Bigr]\Bigr|^{L}\Biggr]\\
    &={\bf{E}}\Biggl[\Bigl|\Delta_i\mathbb{X}^{(j)}_{1}-{\rm{R}}_{i-1}(h_n,\xi)-{\rm{R}}_{i-1}(h_n,\delta)
    \Bigr|^{L}\Biggr]\\
    &\leq C_L{\bf{E}}\biggl[\bigl|\Delta_i\mathbb{X}^{(j)}_{1}\bigr|^{L}\biggr]
    +C_L{\bf{E}}\biggl[\bigl|{\rm{R}}_{i-1}(h_n,\xi)\bigr|^{L}\biggr]+C_L{\bf{E}}\biggl[\bigl|{\rm{R}}_{i-1}(h_n,\delta)\bigr|^{L}\biggr]\\
    &\leq C_L\Bigl(h_n^{\frac{L}{2}}+h_n^L+h_n^L\Bigr)\\
    &\leq C_Lh_n^{\frac{L}{2}}
\end{split}\label{X1d}
\end{align}
for all $L>1$. In an analogous manner, (\ref{X2}) and (\ref{EX2L}) deduce
\begin{align}
\begin{split}
    {\bf{E}}\Biggl[\Bigl|\mathbb{X}^{(k)}_{2,t_i^n}-{\bf{E}}\Bigl[\mathbb{X}^{(k)}_{2,t_{i}^n}\big|\mathscr{F}^{n}_{i-1}\Bigr]\Bigr|^{L}\Biggr]
    \leq C_Lh_n^{\frac{L}{2}}
\end{split}\label{X2d}
\end{align}
for any $L>1$. Consequently, we see from (\ref{X1d}) and (\ref{X2d}) that
\begin{align*}
    {\bf{E}}\Biggl[\Bigl|\mathbb{X}_{t_i^n}-{\bf{E}}\Bigl[\mathbb{X}_{t_{i}^n}\big|\mathscr{F}^{n}_{i-1}\Bigr]\Bigr|^{L}\Biggr]
    &\leq C_L\sum_{\ell=1}^p{\bf{E}}\Biggl[\Bigl|\mathbb{X}^{(\ell)}_{t_i^n}-{\bf{E}}\Bigl[\mathbb{X}^{(\ell)}_{t_{i}^n}\big|\mathscr{F}^{n}_{i-1}\Bigr]\Bigr|^{L}\Biggr]
    \\
    &=C_L\sum_{j=1}^{p_1}{\bf{E}}\Biggl[\Bigl|\mathbb{X}^{(j)}_{1,t_i^n}-{\bf{E}}\Bigl[\mathbb{X}^{(j)}_{1,t_{i}^n}\big|\mathscr{F}^{n}_{i-1}\Bigr]\Bigr|^{L}\Biggr]
    \\
    &\quad+C_L\sum_{k=1}^{p_2}{\bf{E}}\Biggl[\Bigl|\mathbb{X}^{(k)}_{2,t_i^n}-{\bf{E}}\Bigl[\mathbb{X}^{(k)}_{2,t_{i}^n}\big|\mathscr{F}^{n}_{i-1}\Bigr]\Bigr|^{L}\Biggr]\\
    &\leq C_Lh_n^\frac{L}{2}
\end{align*}
for all $L>1$, which yields (\ref{EXine}). It follows from (\ref{X1}) that
\begin{align}
\begin{split}
    {\bf{E}}\Biggl[\Bigl|{\bf{E}}\Bigl[\mathbb{X}^{(j)}_{1,t_{i}^n}\big|\mathscr{F}^{n}_{i-1}\Bigr]-\mathbb{X}^{(j)}_{1,t_{i-1}^n}\Bigr|^{L}\Biggr]&\leq{\bf{E}}\Biggl[\Bigl|
    {\rm{R}}_{i-1}(h_n,\xi)+{\rm{R}}_{i-1}(h_n,\delta)
    \Bigr|^{L}\Biggr]\\
    &\leq C_L{\bf{E}}\biggl[\bigl|{\rm{R}}_{i-1}(h_n,\xi)\bigr|^{L}\biggr]+C_L{\bf{E}}\biggl[\bigl|{\rm{R}}_{i-1}(h_n,\delta)\bigr|^{L}\biggr]\\
    &\leq C_L h_n^L
\end{split}\label{X1d1}
\end{align}
for any $L>1$. In a similar way, (\ref{X2}) implies
\begin{align}
    {\bf{E}}\Biggl[\Bigl|{\bf{E}}\Bigl[\mathbb{X}^{(k)}_{2,t_{i}^n}\big|\mathscr{F}^{n}_{i-1}\Bigr]-\mathbb{X}^{(k)}_{2,t_{i-1}^n}\Bigr|^{L}\Biggr]&\leq C_L h_n^L \label{X2d2}
\end{align}
for all $L>1$. Hence, it holds from (\ref{X1d1}) and (\ref{X2d2}) that
\begin{align*}
    {\bf{E}}\Biggl[\Bigl|{\bf{E}}\Bigl[\mathbb{X}_{t_{i}^n}\big|\mathscr{F}^{n}_{i-1}\Bigr]-\mathbb{X}_{t_{i-1}^n}\Bigr|^{L}\Biggr]&\leq C_L\sum_{\ell=1}^p{\bf{E}}\Biggl[\Bigl|{\bf{E}}\Bigl[\mathbb{X}^{(\ell)}_{t_{i}^n}\big|\mathscr{F}^{n}_{i-1}\Bigr]-\mathbb{X}^{(\ell)}_{t_{i-1}^n}\Bigr|^{L}\Biggr]\\
    &= C_L\sum_{j=1}^{p_1}{\bf{E}}\Biggl[\Bigl|{\bf{E}}\Bigl[\mathbb{X}^{(j)}_{1,t_{i}^n}\big|\mathscr{F}^{n}_{i-1}\Bigr]-\mathbb{X}^{(j)}_{1,t_{i-1}^n}\Bigr|^{L}\Biggr]\\
    &\quad+C_L\sum_{k=1}^{p_2}{\bf{E}}\Biggl[\Bigl|{\bf{E}}\Bigl[\mathbb{X}^{(k)}_{2,t_{i}^n}\big|\mathscr{F}^{n}_{i-1}\Bigr]-\mathbb{X}^{(k)}_{2,t_{i-1}^n}\Bigr|^{L}\Biggr]\\
    &\leq C_Lh_n^L
\end{align*}
for all $L>1$, so that (\ref{EXine2}) holds. Next, we consider (\ref{Vine}).
Since it follows from (\ref{X1}) and Lemma 21 in Kusano and Uchida \cite{Kusano(2023)} that
\begin{align*}
    &\quad\ {\bf{E}}
    \Biggl[\biggl(\mathbb{X}^{(j_1)}_{1,t_{i}^n}-{\bf{E}}
    \Bigl[\mathbb{X}^{(j_1)}_{1,t_{i}^n}\big|\mathscr{F}^{n}_{i-1}\Bigr]\biggr)\biggl(\mathbb{X}^{(j_2)}_{1,t_{i}^n}-{\bf{E}}
    \Bigl[\mathbb{X}^{(j_2)}_{1,t_{i}^n}\big|\mathscr{F}^{n}_{i-1}\Bigr]\biggr)\Big|\mathscr{F}^{n}_{i-1}\Biggr]\\
    &={\bf{E}}
    \Biggl[\biggl(\mathbb{X}^{(j_1)}_{1,t_{i}^n}-\mathbb{X}^{(j_1)}_{1,t_{i}^n}-{\rm{R}}_{i-1}(h_n,\xi)-{\rm{R}}_{i-1}(h_n,\delta)\biggr)\\
    &\qquad\qquad\qquad\qquad\times\biggl(\mathbb{X}^{(j_2)}_{1,t_{i}^n}-\mathbb{X}^{(j_2)}_{1,t_{i}^n}-{\rm{R}}_{i-1}(h_n,\xi)-{\rm{R}}_{i-1}(h_n,\delta)\biggr)\Big|\mathscr{F}^{n}_{i-1}\Biggr]\\
    &={\bf{E}}
    \Biggl[\Bigl(\Delta_i\mathbb{X}^{(j_1)}_{1}\Bigr)
    \Bigl(\Delta_i\mathbb{X}^{(j_2)}_{1}
    \Bigr)\Big|\mathscr{F}^{n}_{i-1}\Biggr]\\
    &\quad+{\rm{R}}_{i-1}(h_n,\xi){\bf{E}}\biggl[\Delta_i\mathbb{X}^{(j_1)}_{1}\big|\mathscr{F}^{n}_{i-1}\biggr]
    +{\rm{R}}_{i-1}(h_n,\delta){\bf{E}}\biggl[\Delta_i\mathbb{X}^{(j_1)}_{1}\big|\mathscr{F}^{n}_{i-1}\biggr]\\
    &\quad+{\rm{R}}_{i-1}(h_n,\xi){\bf{E}}\biggl[\Delta_i\mathbb{X}^{(j_2)}_{1}\big|\mathscr{F}^{n}_{i-1}\biggr]+{\rm{R}}_{i-1}(h_n,\delta){\bf{E}}\biggl[\Delta_i\mathbb{X}^{(j_2)}_{1}\big|\mathscr{F}^{n}_{i-1}\biggr]\\
    &\quad+{\rm{R}}_{i-1}(h_n^2,\xi)+{\rm{R}}_{i-1}(h_n^2,\delta)
    +{\rm{R}}_{i-1}(h_n,\xi){\rm{R}}_{i-1}(h_n,\delta)\\
    &=h_n {\bf{\Sigma}}^{11}(\theta_0)_{j_1j_2}+{\rm{R}}_{i-1}(h_n^2,\xi)+{\rm{R}}_{i-1}(h_n^2,\delta)+{\rm{R}}_{i-1}(h_n,\xi){\rm{R}}_{i-1}(h_n,\delta)
\end{align*}
for $j_1,j_2=1,\cdots,p_1$, we see
\begin{align*}
    &\quad\ {\bf{E}}\left[\left|{\bf{E}}
    \Biggl[\biggl(\mathbb{X}^{(j_1)}_{1,t_{i}^n}-{\bf{E}}
    \Bigl[\mathbb{X}^{(j_1)}_{1,t_{i}^n}\big|\mathscr{F}^{n}_{i-1}\Bigr]\biggr)\biggl(\mathbb{X}^{(j_2)}_{1,t_{i}^n}-{\bf{E}}
    \Bigl[\mathbb{X}^{(j_2)}_{1,t_{i}^n}\big|\mathscr{F}^{n}_{i-1}\Bigr]\biggr)\Big|\mathscr{F}^{n}_{i-1}\Biggr]\right|^L\right]\\
    &\leq C_Lh_n^L\Bigl|{\bf{\Sigma}}^{11}(\theta_0)_{j_1j_2}\Bigr|^L+C_L{\bf{E}}\biggl[\bigl|{\rm{R}}_{i-1}(h_n^2,\xi)\bigr|^L\biggr]+C_L{\bf{E}}\biggl[\bigl|{\rm{R}}_{i-1}(h_n^2,\delta)\bigr|^L\biggr]\\
    &\qquad\qquad\qquad\qquad\qquad\qquad\qquad\qquad+C_L{\bf{E}}\biggl[\bigl|{\rm{R}}_{i-1}(h_n,\xi){\rm{R}}_{i-1}(h_n,\delta)\bigr|^L\biggr]\leq C_Lh_n^L
\end{align*}
for any $L>1$. In an analogous manner, one has
\begin{align*}
    &\quad\ {\bf{E}}
    \Biggl[\biggl(\mathbb{X}^{(j)}_{1,t_{i}^n}-{\bf{E}}
    \Bigl[\mathbb{X}^{(j)}_{1,t_{i}^n}\big|\mathscr{F}^{n}_{i-1}\Bigr]\biggr)\biggl(\mathbb{X}^{(k)}_{2,t_{i}^n}-{\bf{E}}
    \Bigl[\mathbb{X}^{(k)}_{2,t_{i}^n}\big|\mathscr{F}^{n}_{i-1}\Bigr]\biggr)\Big|\mathscr{F}^{n}_{i-1}\Biggr]\\
    &=h_n {\bf{\Sigma}}^{12}(\theta_0)_{jk}+{\rm{R}}_{i-1}(h_n^2,\xi)+{\rm{R}}_{i-1}(h_n,\xi){\rm{R}}_{i-1}(h_n,\delta)+{\rm{R}}_{i-1}(h_n,\xi){\rm{R}}_{i-1}(h_n,\varepsilon)\\
    &\quad+{\rm{R}}_{i-1}(h_n,\xi){\rm{R}}_{i-1}(h_n,\zeta)+{\rm{R}}_{i-1}(h_n,\delta){\rm{R}}_{i-1}(h_n,\varepsilon)+{\rm{R}}_{i-1}(h_n,\delta){\rm{R}}_{i-1}(h_n,\zeta)
\end{align*}
for $j=1,\cdots,p_1$ and $k=1,\cdots,p_2$, and
\begin{align*}
    &\quad\ {\bf{E}}
    \Biggl[\biggl(\mathbb{X}^{(k_1)}_{2,t_{i}^n}-{\bf{E}}
    \Bigl[\mathbb{X}^{(k_1)}_{2,t_{i}^n}\big|\mathscr{F}^{n}_{i-1}\Bigr]\biggr)\biggl(\mathbb{X}^{(k_2)}_{2,t_{i}^n}-{\bf{E}}
    \Bigl[\mathbb{X}^{(k_2)}_{2,t_{i}^n}\big|\mathscr{F}^{n}_{i-1}\Bigr]\biggr)\Big|\mathscr{F}^{n}_{i-1}\Biggr]\\
    &=h_n {\bf{\Sigma}}^{22}(\theta_0)_{k_1k_2}+{\rm{R}}_{i-1}(h_n^2,\xi)+{\rm{R}}_{i-1}(h_n^2,\varepsilon)+{\rm{R}}_{i-1}(h_n^2,\zeta)\\
    &\quad+{\rm{R}}_{i-1}(h_n,\xi){\rm{R}}_{i-1}(h_n,\varepsilon)+{\rm{R}}_{i-1}(h_n,\xi){\rm{R}}_{i-1}(h_n,\zeta)+{\rm{R}}_{i-1}(h_n,\varepsilon){\rm{R}}_{i-1}(h_n,\zeta)
\end{align*}
for $k_1,k_2=1,\cdots,p_2$, so that we get
\begin{align*}
    {\bf{E}}\left[\left|{\bf{E}}
    \Biggl[\biggl(\mathbb{X}^{(j)}_{1,t_{i}^n}-{\bf{E}}
    \Bigl[\mathbb{X}^{(j)}_{1,t_{i}^n}\big|\mathscr{F}^{n}_{i-1}\Bigr]\biggr)\biggl(\mathbb{X}^{(k)}_{2,t_{i}^n}-{\bf{E}}
    \Bigl[\mathbb{X}^{(k)}_{2,t_{i}^n}\big|\mathscr{F}^{n}_{i-1}\Bigr]\biggr)\Big|\mathscr{F}^{n}_{i-1}\Biggr]\right|^L\right]\leq C_Lh_n^L
\end{align*}
and
\begin{align*}
    {\bf{E}}\left[\left|{\bf{E}}
    \Biggl[\biggl(\mathbb{X}^{(k_1)}_{2,t_{i}^n}-{\bf{E}}
    \Bigl[\mathbb{X}^{(k_1)}_{2,t_{i}^n}\big|\mathscr{F}^{n}_{i-1}\Bigr]\biggr)\biggl(\mathbb{X}^{(k_2)}_{2,t_{i}^n}-{\bf{E}}
    \Bigl[\mathbb{X}^{(k_2)}_{2,t_{i}^n}\big|\mathscr{F}^{n}_{i-1}\Bigr]\biggr)\Big|\mathscr{F}^{n}_{i-1}\Biggr]\right|^L\right]\leq C_Lh_n^L
\end{align*}
for all $L>1$. Therefore, it is shown that
\begin{align*}
    &\quad\ {\bf{E}}\Biggl[\Bigl|{\bf{V}}
    \Bigl[\mathbb{X}_{t_{i}^n}\big|\mathscr{F}^{n}_{i-1}\Bigr]\Bigr|^L\Biggr]\\
    &\leq C_L\sum_{j_1=1}^{p_1}\sum_{j_2=1}^{p_1}{\bf{E}}\left[\left|{\bf{E}}
    \Biggl[\biggl(\mathbb{X}^{(j_1)}_{1,t_{i}^n}-{\bf{E}}
    \Bigl[\mathbb{X}^{(j_1)}_{1,t_{i}^n}\big|\mathscr{F}^{n}_{i-1}\Bigr]\biggr)\biggl(\mathbb{X}^{(j_2)}_{1,t_{i}^n}-{\bf{E}}
    \Bigl[\mathbb{X}^{(j_2)}_{1,t_{i}^n}\big|\mathscr{F}^{n}_{i-1}\Bigr]\biggr)\Big|\mathscr{F}^{n}_{i-1}\Biggr]\right|^L\right]\\
    &\quad+C_L\sum_{j=1}^{p_1}\sum_{k=1}^{p_2}{\bf{E}}\left[\left|{\bf{E}}
    \Biggl[\biggl(\mathbb{X}^{(j)}_{1,t_{i}^n}-{\bf{E}}
    \Bigl[\mathbb{X}^{(j)}_{1,t_{i}^n}\big|\mathscr{F}^{n}_{i-1}\Bigr]\biggr)\biggl(\mathbb{X}^{(k)}_{2,t_{i}^n}-{\bf{E}}
    \Bigl[\mathbb{X}^{(k)}_{2,t_{i}^n}\big|\mathscr{F}^{n}_{i-1}\Bigr]\biggr)\Big|\mathscr{F}^{n}_{i-1}\Biggr]\right|^L\right]\\
    &\quad+C_L\sum_{k_1=1}^{p_2}\sum_{k_2=1}^{p_2}{\bf{E}}\left[\left|{\bf{E}}
    \Biggl[\biggl(\mathbb{X}^{(k_1)}_{2,t_{i}^n}-{\bf{E}}
    \Bigl[\mathbb{X}^{(k_1)}_{2,t_{i}^n}\big|\mathscr{F}^{n}_{i-1}\Bigr]\biggr)\biggl(\mathbb{X}^{(k_2)}_{2,t_{i}^n}-{\bf{E}}
    \Bigl[\mathbb{X}^{(k_2)}_{2,t_{i}^n}\big|\mathscr{F}^{n}_{i-1}\Bigr]\biggr)\Big|\mathscr{F}^{n}_{i-1}\Biggr]\right|^L\right]\\
    &\leq C_Lh_n^L
\end{align*}
for any $L>1$, which yields (\ref{Vine}).
\end{proof}
\begin{lemma}\label{deltalemma}
Under {\bf{[A]}}, for all $L>0$,
\begin{align*}
    \sup_{n\in\mathbb{N}}{\bf{E}}_{\mathbb{X}_n}\Biggl[\biggl|\frac{1}{\sqrt{n}}\partial_{\theta^{(j)}}{\rm{H}}_n(\mathbb{X}_n,\theta_0)\biggr|^L\Biggr]<\infty
\end{align*}
for $j=1,\cdots,q$.
\end{lemma}
\begin{proof}[\textbf{Proof of Lemma \ref{deltalemma}}]
Note that
\begin{align*}
    \partial_{\theta^{(j)}}{\rm{H}}_n(\theta)&=-\frac{1}{2h_n}\sum_{i=1}^n\Bigl(\partial_{\theta^{(j)}}{\bf{\Sigma}}(\theta)^{-1}\Bigr)\Bigl[\bigl(\Delta_i \mathbb{X}\bigr)^{\otimes 2}\Bigr]-\frac{n}{2}\partial_{\theta^{(j)}}\log\det{\bf{\Sigma}}(\theta_0)\\
    &=\frac{1}{2h_n}\sum_{i=1}^n\Bigl({\bf{\Sigma}}(\theta)^{-1}\Bigr)\Bigl(\partial_{\theta^{(j)}}{\bf{\Sigma}}(\theta)\Bigr)\Bigl({\bf{\Sigma}}(\theta)^{-1}\Bigr)\Bigl[\bigl(\Delta_i \mathbb{X}\bigr)^{\otimes 2}\Bigr]-\frac{n}{2}\Bigl({\bf{\Sigma}}(\theta)^{-1}\Bigr)\Bigl[\partial_{\theta^{(j)}}{\bf{\Sigma}}(\theta)\Bigr]
\end{align*}
for $j=1,\cdots,q$. Since
\begin{align*}
    \Bigl(\Delta_i \mathbb{X}\Bigr)^{\otimes 2}
    &=\biggl(\mathbb{X}_{t_i^n}-{\bf{E}}\Bigl[\mathbb{X}_{t_{i}^n}\big|\mathscr{F}^{n}_{i-1}\Bigr]\biggr)^{\otimes 2}+\biggl({\bf{E}}\Bigl[\mathbb{X}_{t_{i}^n}\big|\mathscr{F}^{n}_{i-1}\Bigr]-\mathbb{X}_{t_{i-1}^n}\biggr)^{\otimes 2}\\
    &\qquad\qquad+\biggl(\mathbb{X}_{t_i^n}-{\bf{E}}\Bigl[\mathbb{X}_{t_{i}^n}\big|\mathscr{F}^{n}_{i-1}\Bigr]\biggr)\biggl({\bf{E}}\Bigl[\mathbb{X}_{t_{i}^n}\big|\mathscr{F}^{n}_{i-1}\Bigr]-\mathbb{X}_{t_{i-1}^n}\biggr)^{\top}\\
    &\qquad\qquad+\biggl({\bf{E}}\Bigl[\mathbb{X}_{t_{i}^n}\big|\mathscr{F}^{n}_{i-1}\Bigr]-\mathbb{X}_{t_{i-1}^n}\biggr)\biggl(\mathbb{X}_{t_i^n}-{\bf{E}}\Bigl[\mathbb{X}_{t_{i}^n}\big|\mathscr{F}^{n}_{i-1}\Bigr]\biggr)^{\top},
\end{align*}
we have
\begin{align*}
    &\quad\ \frac{1}{\sqrt{n}}\partial_{\theta^{(j)}}{\rm{H}}_n(\theta_0)\\
    &=\frac{1}{2n^{\frac{1}{2}}h_n}\sum_{i=1}^n\Biggl\{\Bigl({\bf{\Sigma}}(\theta_0)^{-1}\Bigr)\Bigl(\partial_{\theta^{(j)}}{\bf{\Sigma}}(\theta_0)\Bigr)\Bigl({\bf{\Sigma}}(\theta_0)^{-1}\Bigr)\Bigl[\bigl(\Delta_i \mathbb{X}\bigr)^{\otimes 2}\Bigr]-h_n\Bigl({\bf{\Sigma}}(\theta_0)^{-1}\Bigr)\Bigl[\partial_{\theta^{(j)}}{\bf{\Sigma}}(\theta_0)\Bigr]\Biggr\}\\
    &=\frac{1}{2n^{\frac{1}{2}}h_n}\sum_{i=1}^n\Bigl({\bf{\Sigma}}(\theta_0)^{-1}\Bigr)\Bigl(\partial_{\theta^{(j)}}{\bf{\Sigma}}(\theta_0)\Bigr)\Bigl({\bf{\Sigma}}(\theta_0)^{-1}\Bigr)\Bigl[\bigl(\Delta_i \mathbb{X}\bigr)^{\otimes 2}-h_n{\bf{\Sigma}}(\theta_0)\Bigr]\\
    &={\bf{M}}_n^{(j)}+{\bf{R}}_n^{(j)}
\end{align*}
for $j=1,\cdots,q$, where
\begin{align*}
    {\bf{M}}_n^{(j)}&=\frac{1}{2n^{\frac{1}{2}}h_n}\sum_{i=1}^n\Bigl({\bf{\Sigma}}(\theta_0)^{-1}\Bigr)\Bigl(\partial_{\theta^{(j)}}{\bf{\Sigma}}(\theta_0)\Bigr)\Bigl({\bf{\Sigma}}(\theta_0)^{-1}\Bigr)\\
    &\qquad\qquad\qquad\qquad\qquad\qquad\Biggl[\biggl(\mathbb{X}_{t_i^n}-{\bf{E}}\Bigl[\mathbb{X}_{t_{i}^n}\big|\mathscr{F}^{n}_{i-1}\Bigr]\biggr)^{\otimes 2}-{\bf{V}}\Bigl[\mathbb{X}_{t_{i}^n}\big|\mathscr{F}^{n}_{i-1}\Bigr]\Biggr]\\
    &\quad+\frac{1}{n^{\frac{1}{2}}h_n}\sum_{i=1}^n\Bigl({\bf{\Sigma}}(\theta_0)^{-1}\Bigr)\Bigl(\partial_{\theta^{(j)}}{\bf{\Sigma}}(\theta_0)\Bigr)\Bigl({\bf{\Sigma}}(\theta_0)^{-1}\Bigr)\\
    &\qquad\qquad\qquad\qquad\qquad\qquad\Biggl[\mathbb{X}_{t_i^n}-{\bf{E}}\Bigl[\mathbb{X}_{t_{i}^n}\big|\mathscr{F}^{n}_{i-1}\Bigr],{\bf{E}}\Bigl[\mathbb{X}_{t_{i}^n}\big|\mathscr{F}^{n}_{i-1}\Bigr]-\mathbb{X}_{t_{i-1}^n}\Biggr]
\end{align*}
and
\begin{align*}
    {\bf{R}}_n^{(j)}&=\frac{1}{2n^{\frac{1}{2}}h_n}\sum_{i=1}^n\Bigl({\bf{\Sigma}}(\theta_0)^{-1}\Bigr)\Bigl(\partial_{\theta^{(j)}}{\bf{\Sigma}}(\theta_0)\Bigr)\Bigl({\bf{\Sigma}}(\theta_0)^{-1}\Bigr)\Biggl[\biggl({\bf{E}}\Bigl[\mathbb{X}_{t_{i}^n}\big|\mathscr{F}^{n}_{i-1}\Bigr]-\mathbb{X}_{t_{i-1}^n}\biggr)^{\otimes 2}\Biggr]\\
    &\qquad\quad+\frac{1}{2n^{\frac{1}{2}}h_n}\sum_{i=1}^n\Bigl({\bf{\Sigma}}(\theta_0)^{-1}\Bigr)\Bigl(\partial_{\theta^{(j)}}{\bf{\Sigma}}(\theta_0)\Bigr)\Bigl({\bf{\Sigma}}(\theta_0)^{-1}\Bigr)\Biggl[{\bf{V}}\Bigl[\mathbb{X}_{t_{i}^n}\big|\mathscr{F}^{n}_{i-1}\Bigr]-h_n{\bf{\Sigma}}(\theta_0)\Biggr].
\end{align*}
First, we will prove 
\begin{align}
    \sup_{n\in\mathbb{N}}{\bf{E}}\biggl[\bigl|{\bf{M}}^{(j)}_n\bigr|^L\biggr]<\infty
    \label{supM}
\end{align}
for any $L>1$. Set
\begin{align*}
    {\bf{N}}^{(j)}_k=\frac{1}{2h_n}\sum_{\ell=1}^{k}{\bf{L}}^{(j)}_\ell
\end{align*}
for $k=0,\cdots,n$, where
\begin{align*}
    {\bf{L}}^{(j)}_\ell&=\Bigl({\bf{\Sigma}}(\theta_0)^{-1}\Bigr)\Bigl(\partial_{\theta^{(j)}}{\bf{\Sigma}}(\theta_0)\Bigr)\Bigl({\bf{\Sigma}}(\theta_0)^{-1}\Bigr)\Biggl[\biggl(\mathbb{X}_{t_i^n}-{\bf{E}}\Bigl[\mathbb{X}_{t_{i}^n}\big|\mathscr{F}^{n}_{i-1}\Bigr]\biggr)^{\otimes 2}-{\bf{V}}\Bigl[\mathbb{X}_{t_{i}^n}\big|\mathscr{F}^{n}_{i-1}\Bigr]\Biggr]\\
    &\quad+2\Bigl({\bf{\Sigma}}(\theta_0)^{-1}\Bigr)\Bigl(\partial_{\theta^{(j)}}{\bf{\Sigma}}(\theta_0)\Bigr)\Bigl({\bf{\Sigma}}(\theta_0)^{-1}\Bigr)\Biggl[\mathbb{X}_{t_i^n}-{\bf{E}}\Bigl[\mathbb{X}_{t_{i}^n}\big|\mathscr{F}^{n}_{i-1}\Bigr],{\bf{E}}\Bigl[\mathbb{X}_{t_{i}^n}\big|\mathscr{F}^{n}_{i-1}\Bigr]-\mathbb{X}_{t_{i-1}^n}\Biggr]
\end{align*}
for $\ell=1,\cdots,k$. Since 
\begin{align*}
    {\bf{E}}\biggl[{\bf{L}}^{(j)}_{\ell}\big|\mathscr{F}^{n}_{\ell-1}\biggr]=0,
\end{align*}
one has
\begin{align*}
    {\bf{E}}\biggl[{\bf{N}}^{(j)}_k\big|\mathscr{F}^{n}_{k-1}\biggr]&=\frac{1}{2h_n}\sum_{\ell=1}^{k-1}{\bf{L}}^{(j)}_{\ell}+\frac{1}{2h_n}{\bf{E}}\biggl[{\bf{L}}^{(j)}_k\big|\mathscr{F}^{n}_{k-1}\biggr]
    ={\bf{N}}^{(j)}_{k-1},
\end{align*}
so that $\{{\bf{N}}_k^{(j)}\}_{k=0}^{n}$ is a discrete-time martingale with respect to $\{\mathscr{F}^n_{i}\}_{i=0}^n$. Note that $\sqrt{n}{\bf{M}}^{(j)}_n$ is the terminal
value of $\{{\bf{N}}_k^{(j)}\}_{k=0}^{n}$:
\begin{align*}
    \sqrt{n}{\bf{M}}_n^{(j)}={\bf{N}}^{(j)}_n.
\end{align*}
Using the Burkholder inequality and
\begin{align*}
    \bigl\langle {\bf{N}}^{(j)}\bigr\rangle_n=\sum_{k=1}^n\bigl({\bf{N}}_k^{(j)}-{\bf{N}}_{k-1}^{(j)}\bigr)^2=\frac{1}{4h^2_n}\sum_{k=1}^n{\bf{L}}_{k}^{(j)2},
\end{align*}
we have
\begin{align*}
    {\bf{E}}\biggl[\bigl|{\bf{N}}^{(j)}_n\bigr|^{L}\biggr]
    &\leq C_{L}{\bf{E}}\biggl[\bigl\langle {\bf{N}}^{(j)}\bigr\rangle_n^{\frac{L}{2}}\biggr]\\
    &\leq \frac{C_L}{h_n^L}{\bf{E}}\left[\Biggl(\sum_{k=1}^n {\bf{L}}^{(j)2}_{k}\Biggr)
    ^{\frac{L}{2}}\right]
    \leq \frac{C_L}{h_n^L}\times n^{\frac{L}{2}-1}\sum_{k=1}^n{\bf{E}}\biggl[\bigl|{\bf{L}}^{(j)}_{k}\bigr|^{L}\biggr]
\end{align*}
for all $L>1$, which yields
\begin{align}
    {\bf{E}}\biggl[\bigl|{\bf{M}}^{(j)}_n\bigr|^{L}\biggr]=\frac{1}{n^{\frac{L}{2}}}{\bf{E}}\biggl[\bigl|{\bf{N}}^{(j)}_n\bigr|^{L}\biggr]
    \leq \frac{C_L}{nh_n^L}\sum_{k=1}^n{\bf{E}}\biggl[\big|{\bf{L}}^{(j)}_k\big|^{L}\biggr].
    \label{Mine}
\end{align}
Moreover, it follows from Lemma \ref{EXVX} and the Cauchy-Schwartz inequality that
\begin{align*}
    {\bf{E}}\biggl[\bigl|{\bf{L}}^{(j)}_k\bigr|^{L}\biggr]&\leq C_L{\bf{E}}\Biggl[\Bigl|\mathbb{X}_{t_i^n}-{\bf{E}}\Bigl[\mathbb{X}_{t_{i}^n}\big|\mathscr{F}^{n}_{i-1}\Bigr]\Bigr|^{2L}
    \Biggr]+C_L{\bf{E}}\Biggl[\Bigl|{\bf{V}}\Bigl[\mathbb{X}_{t_{i}^n}\big|\mathscr{F}^{n}_{i-1}\Bigr]\Bigr|^{L}
    \Biggr]\\
    &\quad+C_L{\bf{E}}\Biggl[\Bigl|\mathbb{X}_{t_i^n}-{\bf{E}}\Bigl[\mathbb{X}_{t_{i}^n}\big|\mathscr{F}^{n}_{i-1}\Bigr]\Bigr|^{L}\Bigl|{\bf{E}}\Bigl[\mathbb{X}_{t_{i}^n}\big|\mathscr{F}^{n}_{i-1}\Bigr]-\mathbb{X}_{t_{i-1}^n}\Bigr|^{L}
    \Biggr]\\
    &\leq C_L{\bf{E}}\Biggl[\Bigl|\mathbb{X}_{t_i^n}-{\bf{E}}\Bigl[\mathbb{X}_{t_{i}^n}\big|\mathscr{F}^{n}_{i-1}\Bigr]\Bigr|^{2L}
    \Biggr]+C_L{\bf{E}}\Biggl[\Bigl|{\bf{V}}\Bigl[\mathbb{X}_{t_{i}^n}\big|\mathscr{F}^{n}_{i-1}\Bigr]\Bigr|^{L}
    \Biggr]\\
    &\quad+C_L{\bf{E}}\Biggl[\Bigl|\mathbb{X}_{t_i^n}-{\bf{E}}\Bigl[\mathbb{X}_{t_{i}^n}\big|\mathscr{F}^{n}_{i-1}\Bigr]\Bigr|^{2L}\Biggr]^{\frac{1}{2}}{\bf{E}}\Biggl[\Bigl|{\bf{E}}\Bigl[\mathbb{X}_{t_{i}^n}\big|\mathscr{F}^{n}_{i-1}\Bigr]-\mathbb{X}_{t_{i-1}^n}\Bigr|^{2L}
    \Biggr]^{\frac{1}{2}}\\
    &\leq C_L\Bigl(h_n^{L}+h_n^{L}+h_n^{\frac{3}{2}L}\Bigr)\\
    &\leq C_L h_n^L
\end{align*}
for any $L>1$, so that it holds from (\ref{Mine}) that
\begin{align*}
    {\bf{E}}\biggl[\bigl|{\bf{M}}^{(j)}_n\bigr|^{L}\biggr]\leq C_L,
\end{align*}
which implies (\ref{supM}). Next, we will prove
\begin{align}
    \sup_{n\in\mathbb{N}}{\bf{E}}\biggl[\bigl|{\bf{R}}^{(j)}_n\bigr|^L\biggr]<\infty
    \label{supR}
\end{align}
for all $L>1$. In an analogous manner to the proof of Lemma \ref{EXVX}, one has
\begin{align}
    {\bf{E}}\Biggl[\Bigl|{\bf{V}}\Bigl[\mathbb{X}_{t_{i}^n}\big|\mathscr{F}^{n}_{i-1}\Bigr]-h_n{\bf{\Sigma}}(\theta_0)\Bigr|^{L}\Biggr]\leq C_Lh_n^{2L} \label{V2}
\end{align}
for all $L>1$. Lemma \ref{EXVX} and (\ref{V2}) show
\begin{align*}
     {\bf{E}}\biggl[\bigl|{\bf{R}}^{(j)}_n\bigr|^L\biggr]&\leq\frac{C_L}{n^{\frac{L}{2}}h_n^L}\times
     n^{L-1}\sum_{i=1}^n{\bf{E}}\Biggl[\Bigl|{\bf{E}}\Bigl[\mathbb{X}_{t_{i}^n}\big|\mathscr{F}^{n}_{i-1}\Bigr]-\mathbb{X}_{t_{i-1}^n}\Bigr|^{2L}
    \Biggr]\\
    &\quad+\frac{C_L}{n^{\frac{L}{2}}h_n^L}\times
     n^{L-1}\sum_{i=1}^n{\bf{E}}\Biggl[\Bigl|{\bf{V}}\Bigl[\mathbb{X}_{t_{i}^n}\big|\mathscr{F}^{n}_{i-1}\Bigr]-h_n{\bf{\Sigma}}(\theta_0)\Bigr|^{L}\Biggr]\\
    &\leq\frac{C_L n^{\frac{L}{2}}}{h_n^L}\bigl(h_n^{2L}+h_n^{2L}\bigr)\\
    &\leq C_L(nh_n^2)^{\frac{L}{2}}
\end{align*}
for any $L>1$. Since $nh_n^2=n^{-1}\longrightarrow 0$ as $n\longrightarrow\infty$, we obtain (\ref{supR}). Consequently, for all $L>1$, it holds from (\ref{supM}) and (\ref{supR}) that
\begin{align*}
    \sup_{n\in\mathbb{N}}{\bf{E}}\Biggl[\biggl|\frac{1}{\sqrt{n}}\partial_{\theta^{(j)}}{\rm{H}}_n(\theta_0)\biggr|^L\Biggr]&\leq C_L\sup_{n\in\mathbb{N}}{\bf{E}}\biggl[\bigl|{\bf{M}}^{(j)}_n\bigr|^L\biggr]+C_L\sup_{n\in\mathbb{N}}{\bf{E}}\biggl[\bigl|{\bf{R}}^{(j)}_n\bigr|^L\biggr]\\
    &<\infty.
\end{align*}
Therefore, it is shown that for all $L>0$, 
\begin{align*}
    \sup_{n\in\mathbb{N}}{\bf{E}}\Biggl[\biggl|\frac{1}{\sqrt{n}}\partial_{\theta^{(j)}}{\rm{H}}_n(\theta_0)\biggr|^L\Biggr]<\infty
\end{align*}
for $j=1,\cdots, q$. 
\end{proof}
\begin{lemma}\label{gammalemma}
Under {\bf{[A]}}, for all $\varepsilon\in(0,\frac{1}{2})$ and $L>0$,
\begin{align*}
    \sup_{n\in\mathbb{N}}{\bf{E}}_{\mathbb{X}_n}\Biggl[\biggl(n^{\varepsilon}\biggl|\frac{1}{n}\partial_{\theta^{(j_1)}}\partial_{\theta^{(j_2)}}{\rm{H}}_n(\mathbb{X}_n,\theta_0)+{\bf{I}}(\theta_0)_{j_1j_2}\biggr|\biggr)^{L}\Biggr]<\infty
\end{align*}
for $j_1,j_2=1,\cdots, q$.
\end{lemma}
\begin{proof}[\textbf{Proof of Lemma \ref{gammalemma}}]
Note that
\begin{align*}
    \Bigl(\partial_{\theta^{(j_1)}}\partial_{\theta^{(j_2)}}{\bf{\Sigma}}(\theta)^{-1}\Bigr)\Bigl[{\bf{\Sigma}}(\theta)\Bigr]&=\tr\biggl\{\Bigl({\bf{\Sigma}}(\theta)^{-1}\Bigr)\Bigl(\partial_{\theta^{(j_1)}}{\bf{\Sigma}}(\theta)\Bigr)\Bigl({\bf{\Sigma}}(\theta)^{-1}\Bigr)\Bigl(\partial_{\theta^{(j_2)}}{\bf{\Sigma}}(\theta)\Bigr)\biggr\}\\
    &\qquad-\tr\biggl\{\Bigl({\bf{\Sigma}}(\theta)^{-1}\Bigr)\Bigl(\partial_{\theta^{(j_1)}}\partial_{\theta^{(j_2)}}{\bf{\Sigma}}(\theta)\Bigr)\biggr\}\\
    &\qquad+\tr\biggl\{\Bigl({\bf{\Sigma}}(\theta)^{-1}\Bigr)\Bigl(\partial_{\theta^{(j_2)}}{\bf{\Sigma}}(\theta)\Bigr)\Bigl({\bf{\Sigma}}(\theta)^{-1}\Bigr)\Bigl(\partial_{\theta^{(j_1)}}{\bf{\Sigma}}(\theta)\Bigr)\biggr\}
\end{align*}
and
\begin{align*}
    \partial_{\theta^{(j_1)}}\partial_{\theta^{(j_2)}}\log \det {\bf{\Sigma}}(\theta)
    &=-\tr\biggl\{\Bigl({\bf{\Sigma}}(\theta)^{-1}\Bigr)\Bigl(\partial_{\theta^{(j_1)}}{\bf{\Sigma}}(\theta)\Bigr)\Bigl({\bf{\Sigma}}(\theta)^{-1}\Bigr)\Bigl(\partial_{\theta^{(j_2)}}{\bf{\Sigma}}(\theta)\Bigr)\biggr\}\\
    &\qquad+\tr\biggl\{\Bigl({\bf{\Sigma}}(\theta)^{-1}\Bigr)\Bigl(\partial_{\theta^{(j_1)}}\partial_{\theta^{(j_2)}}{\bf{\Sigma}}(\theta)\Bigr)\biggr\}
\end{align*}
for $j_1,j_2=1,\cdots,q$. Since
\begin{align*}
    &\quad\ \frac{1}{2}\Bigl(\partial_{\theta^{(j_1)}}\partial_{\theta^{(j_2)}}{\bf{\Sigma}}(\theta_0)^{-1}\Bigr)\Bigl[{\bf{\Sigma}}(\theta_0)\Bigr]+\frac{1}{2}\partial_{\theta^{(j_1)}}\partial_{\theta^{(j_2)}}\log \det {\bf{\Sigma}}(\theta_0)\\
    &=\frac{1}{2}\tr\Biggl\{\Bigl({\bf{\Sigma}}(\theta_0)^{-1}\Bigr)\Bigl(\partial_{\theta^{(j_2)}}{\bf{\Sigma}}(\theta_0)\Bigr)\Bigl({\bf{\Sigma}}(\theta_0)^{-1}\Bigr)\Bigl(\partial_{\theta^{(j_1)}}{\bf{\Sigma}}(\theta_0)\Bigr)\Biggl\}\\
    &=\frac{1}{2}\Bigl(\vec{\partial_{\theta^{(j_1)}}{\bf{\Sigma}}(\theta_0)}\Bigr)^{\top}\Bigl({\bf{\Sigma}}(\theta_0)^{-1}\otimes{\bf{\Sigma}}(\theta_0)^{-1}\Bigr)\Bigl(\vec{\partial_{\theta^{(j_2)}}{\bf{\Sigma}}(\theta_0)}\Bigr)\\
    &=\frac{1}{2}\Bigl(\vech{\partial_{\theta^{(j_1)}}{\bf{\Sigma}}(\theta_0)}\Bigr)^{\top}\mathbb{D}_p^{\top}\Bigl({\bf{\Sigma}}(\theta_0)^{-1}\otimes{\bf{\Sigma}}(\theta_0)^{-1}\Bigr)\mathbb{D}_p\Bigl(\vech{\partial_{\theta^{(j_2)}}{\bf{\Sigma}}(\theta_0)}\Bigr)\\
    &=\Bigl(\partial_{\theta^{(j_1)}}\vech{{\bf{\Sigma}}(\theta_0)}\Bigr)^{\top}{\bf{W}}(\theta_0)^{-1}\Bigl(\partial_{\theta^{(j_2)}}\vech{{\bf{\Sigma}}(\theta_0)}\Bigr)\\
    &={\bf{I}}(\theta_0)_{j_1j_2},
\end{align*}
we have
\begin{align*}
    \frac{1}{n}\partial_{\theta^{(j_1)}}\partial_{\theta^{(j_2)}}{\rm{H}}_n(\theta_0)
    &=-\frac{1}{2nh_n}\sum_{i=1}^n\Bigl(\partial_{\theta^{(j_1)}}\partial_{\theta^{(j_2)}}{\bf{\Sigma}}(\theta_0)^{-1}\Bigr)\Bigl[\bigl(\Delta_i \mathbb{X}\bigr)^{\otimes 2}\Bigr]-\frac{1}{2}\partial_{\theta^{(j_1)}}\partial_{\theta^{(j_2)}}\log\det{\bf{\Sigma}}(\theta_0)\\
    &=-\frac{1}{2nh_n}\sum_{i=1}^n\Bigl(\partial_{\theta^{(j_1)}}\partial_{\theta^{(j_2)}}{\bf{\Sigma}}(\theta_0)^{-1}\Bigr)\Bigl[\bigl(\Delta_i \mathbb{X}\bigr)^{\otimes 2}-h_n{\bf{\Sigma}}(\theta_0)\Bigr]\\
    &\qquad\qquad\quad-\frac{1}{2}\Bigl(\partial_{\theta^{(j_1)}}\partial_{\theta^{(j_2)}}{\bf{\Sigma}}(\theta_0)^{-1}\Bigr)\Bigl[{\bf{\Sigma}}(\theta_0)\Bigr]-\frac{1}{2}\partial_{\theta^{(j_1)}}\partial_{\theta^{(j_2)}}\log\det{\bf{\Sigma}}(\theta_0)\\
    &=-\frac{1}{2nh_n}\sum_{i=1}^n\Bigl(\partial_{\theta^{(j_1)}}\partial_{\theta^{(j_2)}}{\bf{\Sigma}}(\theta_0)^{-1}\Bigr)\Bigl[\bigl(\Delta_i \mathbb{X}\bigr)^{\otimes 2}-h_n{\bf{\Sigma}}(\theta_0)\Bigr]-{\bf{I}}(\theta_0)_{j_1j_2},
\end{align*}
so that a decomposition is given by 
\begin{align*}
    \frac{1}{n}\partial_{\theta^{(j_1)}}\partial_{\theta^{(j_2)}}{\rm{H}}_n(\theta_0)+{\bf{I}}(\theta_0)_{j_1j_2}={\bf{M}}^{\dagger}_{n,j_1j_2}+{\bf{R}}^{\dagger}_{n,j_1j_2},
\end{align*}
where
\begin{align*}
    {\bf{M}}^{\dagger}_{n,j_1j_2}&=-\frac{1}{2nh_n}\sum_{i=1}^n\Bigl(\partial_{\theta^{(j_1)}}\partial_{\theta^{(j_2)}}{\bf{\Sigma}}(\theta_0)^{-1}\Bigr)\Biggl[\biggl(\mathbb{X}_{t_i^n}-{\bf{E}}\Bigl[\mathbb{X}_{t_{i}^n}\big|\mathscr{F}^{n}_{i-1}\Bigr]\biggr)^{\otimes 2}-{\bf{V}}\Bigl[\mathbb{X}_{t_{i}^n}\big|\mathscr{F}^{n}_{i-1}\Bigr]\Biggr]\\
    &\qquad-\frac{1}{nh_n}\sum_{i=1}^n\Bigl(\partial_{\theta^{(j_1)}}\partial_{\theta^{(j_2)}}{\bf{\Sigma}}(\theta_0)^{-1}\Bigr)\Biggl[\mathbb{X}_{t_i^n}-{\bf{E}}\Bigl[\mathbb{X}_{t_{i}^n}\big|\mathscr{F}^{n}_{i-1}\Bigr],{\bf{E}}\Bigl[\mathbb{X}_{t_{i}^n}\big|\mathscr{F}^{n}_{i-1}\Bigr]-\mathbb{X}_{t_{i-1}^n}\Biggr]
\end{align*}
and
\begin{align*}
    {\bf{R}}^{\dagger}_{n,j_1j_2}&=-\frac{1}{2nh_n}\sum_{i=1}^n\Bigl(\partial_{\theta^{(j_1)}}\partial_{\theta^{(j_2)}}{\bf{\Sigma}}(\theta_0)^{-1}\Bigr)\Biggl[\biggl({\bf{E}}\Bigl[\mathbb{X}_{t_{i}^n}\big|\mathscr{F}^{n}_{i-1}\Bigr]-\mathbb{X}_{t_{i-1}^n}\biggr)^{\otimes 2}\Biggr]\\
    &\qquad\qquad\qquad-\frac{1}{2nh_n}\sum_{i=1}^n\Bigl(\partial_{\theta^{(j_1)}}\partial_{\theta^{(j_2)}}{\bf{\Sigma}}(\theta_0)^{-1}\Bigr)\Biggl[{\bf{V}}\Bigl[\mathbb{X}_{t_{i}^n}\big|\mathscr{F}^{n}_{i-1}\Bigr]-h_n{\bf{\Sigma}}(\theta_0)\Biggr].
\end{align*}
Let 
\begin{align*}
    {\bf{N}}^{\dagger}_{k,j_1j_2}=\frac{1}{2h_n}\sum_{\ell=1}^{k}{\bf{L}}^{\dagger}_{\ell,j_1j_2}
\end{align*}
for $k=0,\cdots, n$, where
\begin{align*}
    {\bf{L}}^{\dagger}_{\ell,j_1j_2}&=-\Bigl(\partial_{\theta^{(j_1)}}\partial_{\theta^{(j_2)}}{\bf{\Sigma}}(\theta_0)^{-1}\Bigr)\Biggl[\biggl(\mathbb{X}_{t_i^n}-{\bf{E}}\Bigl[\mathbb{X}_{t_{i}^n}\big|\mathscr{F}^{n}_{i-1}\Bigr]\biggr)^{\otimes 2}-{\bf{V}}\Bigl[\mathbb{X}_{t_{i}^n}\big|\mathscr{F}^{n}_{i-1}\Bigr]\Biggr]\\
    &\qquad\qquad-2\Bigl(\partial_{\theta^{(j_1)}}\partial_{\theta^{(j_2)}}{\bf{\Sigma}}(\theta_0)^{-1}\Bigr)\Biggl[\mathbb{X}_{t_i^n}-{\bf{E}}\Bigl[\mathbb{X}_{t_{i}^n}\big|\mathscr{F}^{n}_{i-1}\Bigr],{\bf{E}}\Bigl[\mathbb{X}_{t_{i}^n}\big|\mathscr{F}^{n}_{i-1}\Bigr]-\mathbb{X}_{t_{i-1}^n}\Biggr]
\end{align*}
for $\ell=1,\cdots,k$. In a similar way to the proof of Lemma \ref{deltalemma}, $\{{\bf{N}}_{k,j_1j_2}^{\dagger}\}_{k=0}^{n}$ is a discrete-time martingale with respect to $\{\mathscr{F}^n_{i}\}_{i=0}^n$, and $n{\bf{M}}^{\dagger}_{n,j_1j_2}$ is the terminal value of $\{{\bf{N}}_{k,j_1j_2}^{\dagger}\}_{k=0}^{n}$:
\begin{align*}
    n{\bf{M}}^{\dagger}_{n,j_1j_2}={\bf{N}}^{\dagger}_{n,j_1j_2}.
\end{align*}
In a similar way to the proof of Lemma \ref{deltalemma}, it follows from the Burkholder inequality that
\begin{align*}
    {\bf{E}}\biggl[\bigl|{\bf{N}}^{\dagger}_{n,j_1j_2}\bigr|^{L}\biggr]
    \leq \frac{C_L}{h_n^L}\times n^{\frac{L}{2}-1}\sum_{k=1}^n{\bf{E}}\biggl[\bigl|{\bf{L}}^{\dagger}_{k,j_1j_2}\bigr|^{L}\biggr]
\end{align*}
for all $L>1$, which yields
\begin{align*}
    {\bf{E}}\biggl[\bigl|{n^{\varepsilon}\bf{M}}^{\dagger}_{n,j_1j_2}\bigr|^{L}\biggr]
    =n^{L(\varepsilon-1)}{\bf{E}}\biggl[\bigl|n{\bf{M}}^{\dagger}_{n,j_1j_2}\bigr|^{L}\biggr]\leq \frac{C_L}{h_n^{L}}\times n^{L(\varepsilon-\frac{1}{2})-1}\sum_{k=1}^n{\bf{E}}\biggl[\bigl|{\bf{L}}^{\dagger}_{k,j_1j_2}\bigr|^{L}\biggr].
\end{align*}
For any $L>1$, it is shown that
\begin{align*}
    {\bf{E}}\biggl[\bigl|{\bf{L}}^{\dagger}_{k,j_1j_2}\bigr|^{L}\biggr]\leq C_L h_n^L
\end{align*}
in an analogous manner to the proof of Lemma \ref{deltalemma}, which deduces
\begin{align*}
    {\bf{E}}\biggl[\bigl|{n^{\varepsilon}\bf{M}}^{\dagger}_{n,j_1j_2}\bigr|^{L}\biggr]\leq C_L n^{L(\varepsilon-\frac{1}{2})}.
\end{align*}
Since $\varepsilon-\frac{1}{2}<0$, we have $n^{L(\varepsilon-\frac{1}{2})}\longrightarrow 0$ as $n\longrightarrow\infty$, so that
\begin{align}
    \sup_{n\in\mathbb{N}}{\bf{E}}\biggl[\bigl|{n^{\varepsilon}\bf{M}}^{\dagger}_{n,j_1j_2}\bigr|^{L}\biggr]<\infty \label{supMd}
\end{align}
for all $L>1$. Furthermore, we see from Lemma \ref{EXVX} and (\ref{V2}) that for all $L>1$,
\begin{align*}
     {\bf{E}}\biggl[\bigl|n^{\varepsilon}{\bf{R}}^{\dagger}_{n,j_1j_2}\bigr|^L\biggr]&\leq\frac{C_L n^{L(\varepsilon-1)}}{h_n^L}\times
     n^{L-1}\sum_{i=1}^n{\bf{E}}\Biggl[\Bigl|{\bf{E}}\Bigl[\mathbb{X}_{t_{i}^n}\big|\mathscr{F}^{n}_{i-1}\Bigr]-\mathbb{X}_{t_{i-1}^n}\Bigr|^{2L}
    \Biggr]\\
    &\quad+\frac{C_L n^{L(\varepsilon-1)}}{h_n^L}\times n^{L-1}\sum_{i=1}^n{\bf{E}}\Biggl[\Bigl|{\bf{V}}\Bigl[\mathbb{X}_{t_{i}^n}\big|\mathscr{F}^{n}_{i-1}\Bigr]-h_n{\bf{\Sigma}}(\theta_0)\Bigr|^{L}\Biggr]\\
    &\leq\frac{C_L n^{L\varepsilon}}{h_n^L}\bigl(h_n^{2L}+h_n^{2L}\bigr)\\
    &\leq C_L(nh_n^2)^{\frac{L}{2}}n^{L(\varepsilon-\frac{1}{2})}
\end{align*}
and $(nh_n^2)^{\frac{L}{2}}n^{L(\varepsilon-\frac{1}{2})}\longrightarrow 0$ as $n\longrightarrow\infty$, which implies
\begin{align}
    \sup_{n\in\mathbb{N}}{\bf{E}}\biggl[\bigl|{n^{\varepsilon}\bf{R}}^{\dagger}_{n,j_1j_2}\bigr|^{L}\biggr]<\infty .\label{supRd}
\end{align}
Hence, it holds from (\ref{supMd}) and (\ref{supRd}) that
\begin{align*}
    \sup_{n\in\mathbb{N}}{\bf{E}}\Biggl[\biggl(n^{\varepsilon}\biggl|\frac{1}{n}\partial_{\theta^{(j_1)}}\partial_{\theta^{(j_2)}}{\rm{H}}_n(\theta_0)+{\bf{I}}(\theta_0)_{j_1j_2}\biggr|\biggr)^{L}\Biggr]&\leq \sup_{n\in\mathbb{N}}{\bf{E}}\biggl[\bigl|{n^{\varepsilon}\bf{M}}^{\dagger}_{n,j_1j_2}\bigr|^{L}\biggr]+\sup_{n\in\mathbb{N}}{\bf{E}}\biggl[\bigl|{n^{\varepsilon}\bf{R}}^{\dagger}_{n,j_1j_2}\bigr|^{L}\biggr]\\
    &<\infty
\end{align*}
for all $L>1$. Therefore, for all $L>0$, we obtain
\begin{align*}
    \sup_{n\in\mathbb{N}}{\bf{E}}\Biggl[\biggl(n^{\varepsilon}\biggl|\frac{1}{n}\partial_{\theta^{(j_1)}}\partial_{\theta^{(j_2)}}{\rm{H}}_n(\theta_0)+{\bf{I}}(\theta_0)_{j_1j_2}\biggr|\biggr)^{L}\Biggr]<\infty
\end{align*}
for $j_1,j_2=1,\cdots,q$.
\end{proof}
\begin{lemma}\label{Ylemma}
Under {\bf{[A]}}, for all $\varepsilon\in(0,\frac{1}{2})$ and $L>0$,
\begin{align*}
    \sup_{n\in\mathbb{N}}{\bf{E}}_{\mathbb{X}_n}
    \left[\left(\sup_{\theta\in\Theta}n^{\varepsilon}
    \Bigl|{\rm{Y}}_n(\mathbb{X}_n,\theta;\theta_0)-{\rm{Y}}(\theta)\Bigr|\right)^{L}\right]<\infty.
\end{align*}
\end{lemma}
\begin{proof}[\textbf{Proof of Lemma \ref{Ylemma}}]
Since
\begin{align*}
    {\rm{Y}}_n(\theta;\theta_0)&=\frac{1}{n}\Bigl\{{\rm{H}}_n(\theta)-{\rm{H}}_n(\theta_0)\Bigr\}\\
    &=-\frac{1}{2nh_n}\sum_{i=1}^n\Bigl({\bf{\Sigma}}(\theta)^{-1}-{\bf{\Sigma}}(\theta_0)^{-1}\Bigr)\Bigl[(\Delta_i \mathbb{X})^{\otimes 2}\Bigr]-\frac{1}{2}\log\frac{\det{\bf{\Sigma}}(\theta)}{\det{\bf{\Sigma}}(\theta_0)}
\end{align*}
and
\begin{align*}
    {\rm{Y}}(\theta)
    &=-\frac{1}{2}\Bigl({\bf{\Sigma}}(\theta)^{-1}-{\bf{\Sigma}}(\theta_0)^{-1}\Bigr)\Bigl[{\bf{\Sigma}}(\theta_0)\Bigr]-\frac{1}{2}
    \log\frac{\det{\bf{\Sigma}}(\theta)}{\det{\bf{\Sigma}}(\theta_0)},
\end{align*}
one has a decomposition
\begin{align*}
    {\rm{Y}}_n(\theta;\theta_0)-{\rm{Y}}(\theta)
    &=-\frac{1}{2nh_n}\sum_{i=1}^n\Bigl({\bf{\Sigma}}(\theta)^{-1}-{\bf{\Sigma}}(\theta_0)^{-1}\Bigr)\Bigl[(\Delta_i \mathbb{X})^{\otimes 2}-h_n{\bf{\Sigma}}(\theta_0)\Bigr]\\
    &={\bf{M}}^{\dagger\dagger}_n+{\bf{R}}^{\dagger\dagger}_n,
\end{align*}
where
\begin{align*}
    {\bf{M}}^{\dagger\dagger}_{n}&=-\frac{1}{2nh_n}\sum_{i=1}^n\Bigl({\bf{\Sigma}}(\theta)^{-1}-{\bf{\Sigma}}(\theta_0)^{-1}\Bigr)\Biggl[\biggl(\mathbb{X}_{t_i^n}-{\bf{E}}\Bigl[\mathbb{X}_{t_{i}^n}\big|\mathscr{F}^{n}_{i-1}\Bigr]\biggr)^{\otimes 2}-{\bf{V}}\Bigl[\mathbb{X}_{t_{i}^n}\big|\mathscr{F}^{n}_{i-1}\Bigr]\Biggr]\\
    &\quad-\frac{1}{nh_n}\sum_{i=1}^n\Bigl({\bf{\Sigma}}(\theta)^{-1}-{\bf{\Sigma}}(\theta_0)^{-1}\Bigr)\Biggl[\mathbb{X}_{t_i^n}-{\bf{E}}\Bigl[\mathbb{X}_{t_{i}^n}\big|\mathscr{F}^{n}_{i-1}\Bigr],{\bf{E}}\Bigl[\mathbb{X}_{t_{i}^n}\big|\mathscr{F}^{n}_{i-1}\Bigr]-\mathbb{X}_{t_{i-1}^n}\Biggr]
\end{align*}
and
\begin{align*}
    {\bf{R}}^{\dagger\dagger}_{n}&=-\frac{1}{2nh_n}\sum_{i=1}^n\Bigl({\bf{\Sigma}}(\theta)^{-1}-{\bf{\Sigma}}(\theta_0)^{-1}\Bigr)\Biggl[\biggl({\bf{E}}\Bigl[\mathbb{X}_{t_{i}^n}\big|\mathscr{F}^{n}_{i-1}\Bigr]-\mathbb{X}_{t_{i-1}^n}\biggr)^{\otimes 2}\Biggr]\\
    &\qquad\qquad\qquad\qquad\qquad-\frac{1}{2nh_n}\sum_{i=1}^n\Bigl({\bf{\Sigma}}(\theta)^{-1}-{\bf{\Sigma}}(\theta_0)^{-1}\Bigr)\Biggl[{\bf{V}}\Bigl[\mathbb{X}_{t_{i}^n}\big|\mathscr{F}^{n}_{i-1}\Bigr]-h_n{\bf{\Sigma}}(\theta_0)\Biggr].
\end{align*}
In an analogous manner to Lemma \ref{gammalemma}, one has
\begin{align*}
    \sup_{n\in\mathbb{N}}\sup_{\theta\in\Theta}{\bf{E}}\biggl[\bigl|{n^{\varepsilon}\bf{M}}^{\dagger\dagger}_{n}\bigr|^{L}\biggr]<\infty
\end{align*}
and
\begin{align*}
    \sup_{n\in\mathbb{N}}\sup_{\theta\in\Theta}{\bf{E}}\biggl[\bigl|{n^{\varepsilon}\partial_{\theta}\bf{M}}^{\dagger\dagger}_{n}\bigr|^{L}\biggr]<\infty
\end{align*}
for all $L>1$. Consequently, it holds from the Sobolev inequality that
\begin{align*}
    {\bf{E}}\Biggl[\sup_{\theta\in\Theta}\bigl|{n^{\varepsilon}\bf{M}}^{\dagger\dagger}_{n}\bigr|^{L}\Biggr]&\leq {\bf{E}}\Biggl[\int_{\Theta}\bigl|{n^{\varepsilon}\bf{M}}^{\dagger\dagger}_{n}\bigr|^{L}+\bigl|{n^{\varepsilon}\partial_{\theta}\bf{M}}^{\dagger\dagger}_{n}\bigr|^{L}d\theta\Biggr]\\
    &=\int_{\Theta}{\bf{E}}\biggl[\bigl|{n^{\varepsilon}\bf{M}}^{\dagger\dagger}_{n}\bigr|^{L}\biggr]d\theta+\int_{\Theta}{\bf{E}}\biggl[
    \bigl|{n^{\varepsilon}\partial_{\theta}\bf{M}}^{\dagger\dagger}_{n}\bigr|^{L}\biggr]d\theta\\
    &\leq\int_{\Theta}\sup_{\theta\in\Theta}{\bf{E}}\biggl[\bigl|{n^{\varepsilon}\bf{M}}^{\dagger\dagger}_{n}\bigr|^{L}\biggr]d\theta+\int_{\Theta}\sup_{\theta\in\Theta}{\bf{E}}\biggl[
    \bigl|{n^{\varepsilon}\partial_{\theta}\bf{M}}^{\dagger\dagger}_{n}\bigr|^{L}\biggr]d\theta\\
    &\leq C_{\Theta}\sup_{\theta\in\Theta}{\bf{E}}\biggl[\bigl|{n^{\varepsilon}\bf{M}}^{\dagger\dagger}_{n}\bigr|^{L}\biggr]+C_{\Theta}\sup_{\theta\in\Theta}{\bf{E}}\biggl[
    \bigl|{n^{\varepsilon}\partial_{\theta}\bf{M}}^{\dagger\dagger}_{n}\bigr|^{L}\biggr]
\end{align*}
for any $L>q$, which yields
\begin{align}
    \sup_{n\in\mathbb{N}}{\bf{E}}\Biggl[\sup_{\theta\in\Theta}\bigl|{n^{\varepsilon}\bf{M}}^{\dagger\dagger}_{n}\bigr|^{L}\Biggr]<\infty.\label{supthetaMdd}
\end{align}
In a similar way, it is shown that
\begin{align}
    \sup_{n\in\mathbb{N}}{\bf{E}}\Biggl[\sup_{\theta\in\Theta}\bigl|{n^{\varepsilon}\bf{R}}^{\dagger\dagger}_{n}\bigr|^{L}\Biggr]<\infty \label{supthetaRdd}
\end{align}
for all $L>q$. Thus, we see from (\ref{supthetaMdd}) and (\ref{supthetaRdd}) that
\begin{align*}
    &\quad\ \sup_{n\in\mathbb{N}}{\bf{E}}\Biggl[\biggl(n^{\varepsilon}\sup_{\theta\in\Theta}\Bigl|{\rm{Y}}_n(\theta;\theta_0)-{\rm{Y}}(\theta)\Bigr|\biggr)^{L}\Biggr]\\
    &\leq C_L \sup_{n\in\mathbb{N}}{\bf{E}}\Biggl[\sup_{\theta\in\Theta}\bigl|{n^{\varepsilon}\bf{M}}^{\dagger\dagger}_{n}\bigr|^{L}\Biggr]
    +C_L\sup_{n\in\mathbb{N}}{\bf{E}}\Biggl[\sup_{\theta\in\Theta}\bigl|{n^{\varepsilon}\bf{R}}^{\dagger\dagger}_{n}\bigr|^{L}\Biggr]<\infty
\end{align*}
for any $L>q$. Therefore, one gets
\begin{align*}
    &\quad\ \sup_{n\in\mathbb{N}}{\bf{E}}\Biggl[\biggl(n^{\varepsilon}\sup_{\theta\in\Theta}\Bigl|{\rm{Y}}_n(\theta;\theta_0)-{\rm{Y}}(\theta)\Bigr|\biggr)^{L}\Biggr]<\infty
\end{align*}
for all $\varepsilon\in(0,\frac{1}{2})$ and $L>0$.
\end{proof}
\begin{lemma}\label{H3lemma}
Under {\bf{[A]}}, for all $L>0$,
\begin{align*}
    \sup_{n\in\mathbb{N}}{\bf{E}}_{\mathbb{X}_n}
    \left[\left(\frac{1}{n}\sup_{\theta\in\Theta}
    \Bigl|\partial_{\theta^{(j_1)}}\partial_{\theta^{(j_2)}}\partial_{\theta^{(j_3)}}{\rm{H}}_n(\mathbb{X}_n,\theta)
    \Bigr|\right)^{L}\right]
    <\infty
\end{align*}
for $j_1,j_2,j_3=1,\cdots, q$.
\end{lemma}
\begin{proof}[\textbf{Proof of Lemma \ref{H3lemma}}]
Since
\begin{align*}
    \frac{1}{n}\partial_{\theta^{(j_1)}}\partial_{\theta^{(j_2)}}\partial_{\theta^{(j_3)}}
    {\rm{H}}_n(\theta)
    &=-\frac{1}{2nh_n}\sum_{i=1}^n
    \Bigl(\partial_{\theta^{(j_1)}}\partial_{\theta^{(j_2)}}\partial_{\theta^{(j_3)}}
    {\bf{\Sigma}}(\theta)^{-1}\Bigr)
    \Bigl[\bigl(\Delta_i \mathbb{X}\bigr)^{\otimes 2}\Bigr]\\
    &\qquad\qquad\qquad\qquad\quad
    -\frac{1}{2}
    \partial_{\theta^{(j_1)}}\partial_{\theta^{(j_2)}}\partial_{\theta^{(j_3)}}
    \log\det{\bf{\Sigma}}(\theta)
\end{align*}
for $j_1,j_2,j_3=1,\cdots,q$, it holds from Lemma \ref{EXVX} that
\begin{align*}
    &\quad\ {\bf{E}}
    \Biggl[\biggl(\frac{1}{n}
    \Bigl|\partial_{\theta^{(j_1)}}\partial_{\theta^{(j_2)}}\partial_{\theta^{(j_3)}}{\rm{H}}_n(\theta)
    \Bigr|\biggr)^{L}\Biggr]\\
    &\leq\frac{C_{L}}{n^{L}h_n^{L}}\times n^{L-1}\sum_{i=1}^n{\bf{E}}
    \Biggl[
    \biggl|\Bigl(
    \partial_{\theta^{(j_1)}}\partial_{\theta^{(j_2)}}
    \partial_{\theta^{(j_3)}}
    {\bf{\Sigma}}(\theta)^{-1}\Bigr)\Bigl[\bigl(\Delta_i \mathbb{X}\bigr)^{\otimes 2}\Bigr]
    \biggr|^{L}\Biggr]\\
    &\qquad\qquad\qquad\qquad\qquad\qquad\qquad\qquad
    +C_L\biggl|\partial_{\theta^{(j_1)}}\partial_{\theta^{(j_2)}}\partial_{\theta^{(j_3)}}
    \log\det{\bf{\Sigma}}(\theta)\biggr|^L\\
    &\leq\frac{C_{L}}{nh_n^{L}}\biggl|\partial_{\theta^{(j_1)}}\partial_{\theta^{(j_2)}}\partial_{\theta^{(j_3)}}
    {\bf{\Sigma}}(\theta)^{-1}\biggr|^L
    \sum_{i=1}^n{\bf{E}}
    \biggl[
    \bigl|\Delta_i \mathbb{X}
    \bigr|^{2L}\biggr]\\
    &\qquad\qquad\qquad\qquad\qquad\qquad\qquad
    +C_L\biggl|\partial_{\theta^{(j_1)}}\partial_{\theta^{(j_2)}}
    \partial_{\theta^{(j_3)}}
    \log\det{\bf{\Sigma}}(\theta)\biggr|^L\\
    &\leq C_L\sup_{\theta\in\Theta}\biggl|\partial_{\theta^{(j_1)}}\partial_{\theta^{(j_2)}}\partial_{\theta^{(j_3)}}
    {\bf{\Sigma}}(\theta)^{-1}\biggr|^L+C_L\sup_{\theta\in\Theta}
    \biggl|\partial_{\theta^{(j_1)}}\partial_{\theta^{(j_2)}}
    \partial_{\theta^{(j_3)}}
    \log\det{\bf{\Sigma}}(\theta)\biggr|^L\\
    &\leq C_L
\end{align*}
for all $L>1$, which yields
\begin{align}
    \sup_{n\in\mathbb{N}}\sup_{\theta\in\Theta}{\bf{E}}
    \Biggl[\biggl(\frac{1}{n}
    \Bigl|\partial_{\theta^{(j_1)}}\partial_{\theta^{(j_2)}}
    \partial_{\theta^{(j_3)}}{\rm{H}}_n(\theta)
    \Bigr|\biggr)^{L}\Biggr]<\infty. \label{H3}
\end{align}
Similarly, one has
\begin{align}
    \sup_{n\in\mathbb{N}}\sup_{\theta\in\Theta}{\bf{E}}
    \Biggl[\biggl(\frac{1}{n}
    \Bigl|\partial_{\theta}\partial_{\theta^{(j_1)}}\partial_{\theta^{(j_2)}}
    \partial_{\theta^{(j_3)}}{\rm{H}}_n(\theta)
    \Bigr|\biggr)^{L}\Biggr]<\infty \label{H4}
\end{align}
for any $L>1$. 
By using the Sobolev inequality, it is shown that
\begin{align*}
    &\quad\ {\bf{E}}
    \Biggl[\biggl(\frac{1}{n}\sup_{\theta\in\Theta}
    \Bigl|\partial_{\theta^{(j_1)}}\partial_{\theta^{(j_2)}}
    \partial_{\theta^{(j_3)}}
    {\rm{H}}_n(\theta)
    \Bigr|\biggr)^{L}\Biggr]\\
    &\leq{\bf{E}}\Biggl[\int_{\Theta}
    \biggl|\frac{1}{n}\partial_{\theta^{(j_1)}}\partial_{\theta^{(j_2)}}\partial_{\theta^{(j_3)}}
    {\rm{H}}_n(\theta)
    \biggr|^{L}+\biggl|\frac{1}{n}\partial_{\theta}\partial_{\theta^{(j_1)}}\partial_{\theta^{(j_2)}}\partial_{\theta^{(j_3)}}
    {\rm{H}}_n(\theta)
    \biggr|^{L}d\theta\Biggr]\\
    &=\int_{\Theta}{\bf{E}}\Biggl[
    \biggl|\frac{1}{n}\partial_{\theta^{(j_1)}}\partial_{\theta^{(j_2)}}\partial_{\theta^{(j_3)}}
    {\rm{H}}_n(\theta)
    \biggr|^{L}\Biggr]d\theta+\int_{\Theta}{\bf{E}}\Biggl[\biggl|\frac{1}{n}\partial_{\theta}\partial_{\theta^{(j_1)}}\partial_{\theta^{(j_2)}}\partial_{\theta^{(j_3)}}
    {\rm{H}}_n(\theta)
    \biggr|^{L}\Biggr]d\theta\\
    &\leq\int_{\Theta}\sup_{\theta\in\Theta}{\bf{E}}\Biggl[
    \biggl|\frac{1}{n}\partial_{\theta^{(j_1)}}\partial_{\theta^{(j_2)}}\partial_{\theta^{(j_3)}}
    {\rm{H}}_n(\theta)
    \biggr|^{L}\Biggr]d\theta\\
    &\qquad\qquad\qquad\qquad\qquad+\int_{\Theta}\sup_{\theta\in\Theta}{\bf{E}}\Biggl[
    \biggl|\frac{1}{n}\partial_{\theta}\partial_{\theta^{(j_1)}}\partial_{\theta^{(j_2)}}\partial_{\theta^{(j_3)}}
    {\rm{H}}_n(\theta)
    \biggr|^{L}\Biggr]d\theta\\
    &\leq C_{\Theta}\sup_{\theta\in\Theta}{\bf{E}}\Biggl[
    \biggl|\frac{1}{n}\partial_{\theta^{(j_1)}}\partial_{\theta^{(j_2)}}\partial_{\theta^{(j_3)}}
    {\rm{H}}_n(\theta)
    \biggr|^{L}\Biggr]\\
    &\qquad\qquad\qquad\qquad\qquad+C_{\Theta}\sup_{\theta\in\Theta}{\bf{E}}\Biggl[
    \biggl|\frac{1}{n}\partial_{\theta}\partial_{\theta^{(j_1)}}\partial_{\theta^{(j_2)}}\partial_{\theta^{(j_3)}}
    {\rm{H}}_n(\theta)
    \biggr|^{L}\Biggr]
\end{align*}
for all $L>q$, so that we obtain from (\ref{H3}) and (\ref{H4}) that for all $\varepsilon\in(0,\frac{1}{2})$ and $L>0$,
\begin{align*}
    \sup_{n\in\mathbb{N}}{\bf{E}}
    \Biggl[\biggl(\frac{1}{n}\sup_{\theta\in\Theta}
    \Bigl|\partial_{\theta^{(j_1)}}\partial_{\theta^{(j_2)}}
    \partial_{\theta^{(j_3)}}{\rm{H}}_n(\theta)
    \Bigr|\biggr)^{L}\Biggr]<\infty
\end{align*}
for $j_1,j_2,j_3=1,\cdots,q$.
\end{proof}
\begin{lemma}\label{Zine}
Under {\bf{[A]}} and {\bf{[B1]}}, for all $L>0$, there exists $C_L>0$ such that
\begin{align*}
    {\bf{P}}\left(\sup_{u\in {\rm{V}}_n(r)}{\rm{Z}}_n(\mathbb{X}_n,u;\theta_0)\geq e^{-r}\right)\leq\frac{C_L}{r^L}
\end{align*}
for all $r>0$ and $n\in\mathbb{N}$.
\end{lemma}
\begin{proof}
It is enough to check the regularity conditions [A1$^{\prime\prime}$], [A4$^{\prime}$], [A6], [B1] and [B2] of Theorem 3 (c) in Yoshida \cite{Yoshida(2011)}. It is supposed that $\alpha$, $\rho_1$, $\rho_2$, $\beta_1$ and $\beta_2$ satisfy [A4$^{\prime}$]:
\begin{align*}
    0<\beta_1<\frac{1}{2},\ \ 0<\rho_1<\min\Bigl\{1,\beta,\frac{2\beta_1}{1-\alpha}\Bigr\},\ \ 2\alpha<\rho_2,\ \ \beta_2\geq 0,\ \
    1-2\beta_2-\rho_2>0, 
\end{align*}
where $\beta=\alpha(1-\alpha)^{-1}$.
For example, we can take $\alpha=\frac{1}{10}$, $\rho_1=\frac{1}{10}$, $\rho_2=\frac{1}{4}$, $\beta_1=\frac{1}{4}$ and $\beta_2=\frac{1}{3}$. For any $L>0$, it follows from Lemmas \ref{deltalemma} and \ref{Ylemma} that 
\begin{align*}
    \sup_{n\in\mathbb{N}}{\bf{E}}\Biggl[\biggl|\frac{1}{\sqrt{n}}\partial_{\theta}{\rm{H}}_n(\theta_0)\biggr|^{M_1}\Biggr]<\infty
\end{align*}
and
\begin{align*}
    \sup_{n\in\mathbb{N}}{\bf{E}}
    \left[\left(\sup_{\theta\in\Theta}n^{\frac{1}{2}-\beta_1}
    \Bigl|{\rm{Y}}_n(\theta;\theta_0)-{\rm{Y}}(\theta)\Bigr|\right)^{M_2}\right]<\infty,
\end{align*}
where $M_1=L(1-\rho_1)^{-1}>0$ and $M_2=L(1-2\beta_2-\rho_2)^{-1}>0$, which satisfies [A6]. Furthermore, we see from Lemmas \ref{gammalemma} and \ref{H3lemma} that
\begin{align*}
    \sup_{n\in\mathbb{N}}{\bf{E}}
    \left[\left(\sup_{\theta\in\Theta}\frac{1}{n}
    \Bigl|\partial^{3}_{\theta}{\rm{H}}_n(\theta)
    \Bigr|\right)^{M_3}\right]
    <\infty
\end{align*}
and 
\begin{align*}
    \sup_{n\in\mathbb{N}}{\bf{E}}\Biggl[\biggl|n^{\beta_1}\left\{\frac{1}{n}\partial^2_{\theta}{\rm{H}}_{n}(\theta_0)+{\bf{I}}(\theta_0)\right\}\biggr|^{M_4}\Biggr]
    <\infty
\end{align*}
for all $L>0$, where $M_3=L(\beta-\rho_1)^{-1}>0$ and $M_4=L\bigl(\frac{2\beta_1}{1-\alpha}-\rho_1\bigr)^{-1}>0$. Hence, [A1$^{\prime\prime}$] is satisfied. It follows from Lemma 35 in Kusano and Uchida \cite{Kusano(2023)} and {\bf{[B1]}} (b) that ${\bf{I}}(\theta_0)$ is a positive definite matrix, which satisfies [B1]. Moreover, {\bf{[B1]}} (a) yields [B2]. 
\end{proof}
$\mathbb{E}$ denotes the expectation under the probability measure on the probability space on which $\zeta$ is realized.
\begin{lemma}\label{moment}
Under {\bf{[A]}} and {\bf{[B1]}}, for all $L>0$,
\begin{align*}
    \sup_{n\in\mathbb{N}}{\bf{E}}_{\mathbb{X}_n}\Biggl[\Bigl|\sqrt{n}(\hat{\theta}_{n}-\theta_{0})\Bigr|^L\Biggr]<\infty
\end{align*}
and for $f\in C_{\uparrow}(\mathbb{R}^{q})$,
\begin{align*}
   {\bf{E}}_{\mathbb{X}_n}\biggl[f\Bigl(\sqrt{n}(\hat{\theta}_{n}-\theta_{0})\Bigr)\biggr]\longrightarrow\mathbb{E}\biggl[f\Bigl({\bf{I}}(\theta_{0})^{-\frac{1}{2}}\zeta\Bigr)\biggr]
\end{align*}
as $n\longrightarrow\infty$.
\end{lemma}
\begin{proof}
Note that $\hat{u}_n\in\mathbb{U}_n$ since $\theta_0+\frac{1}{\sqrt{n}}\hat{u}_n=\hat{\theta}_n\in\Theta$. For all $r>0$, we have
    \begin{align*}
    0&\leq {\rm{H}}_n(\hat{\theta}_n)-{\rm{H}}_n(\theta_0)\\
    &={\rm{H}}_n\Bigl(\theta_0+\frac{1}{\sqrt{n}}\hat{u}_n\Bigr)-{\rm{H}}_n(\theta_0)\\
    &=\log {\rm{Z}}_n(\hat{u}_n;\theta_0)\leq \log \sup_{u\in{\rm{V}}_n(r)}{\rm{Z}}_n(u;\theta_0)
    \end{align*}
on $\{|\hat{u}_n|\geq r\}$, which yields
\begin{align*}
    1\leq \sup_{u\in{\rm{V}}_n(r)}{\rm{Z}}_n(u;\theta_0)
\end{align*}
on $\{|\hat{u}_n|\geq r\}$. For any $L>0$, it holds from Lemma \ref{Zine} that there exists $C_L>0$ such that
\begin{align*}
    {\bf{P}}\Bigl(\bigl|\sqrt{n}(\hat{\theta}_n-\theta_0)\bigr|\geq r\Bigr)&\leq {\bf{P}}\left(\sup_{u\in{\rm{V}}_n(r)}{\rm{Z}}_n(u;\theta_0)\geq 1\right)\\
    &\leq{\bf{P}}\left(\sup_{u\in{\rm{V}}_n(r)}{\rm{Z}}_n(u;\theta_0)\geq e^{-r}\right)
    \leq \frac{C_L}{r^L}
\end{align*}
for all $r>0$ and $n\in\mathbb{N}$, which implies
\begin{align}
    \sup_{n\in\mathbb{N}}{\bf{E}}\Biggl[\Bigl|\sqrt{n}(\hat{\theta}_{n}-\theta_{0})\Bigr|^L\Biggr]<\infty. \label{thetamoment}
\end{align}
Furthermore, we see from (\ref{thetamoment}) and Lemma \ref{thetaprob2} that for all $f\in C_{\uparrow}(\mathbb{R}^{q})$,
\begin{align*}
   {\bf{E}}\Biggl[f\Bigl(\sqrt{n}(\hat{\theta}_{n}-\theta_{0})\Bigr)\Biggr]\longrightarrow\mathbb{E}\Biggl[f\Bigl({\bf{I}}(\theta_{0})^{-\frac{1}{2}}\zeta\Bigr)\Biggr]
\end{align*}
as $n\longrightarrow\infty$.
\end{proof}
\begin{proof}[\textbf{Proof of Theorem 1}]
Let us consider the following decomposition:
\begin{align*}
    {\bf{E}}_{\mathbb{X}_n}\biggl[\log {\rm{L}}_{n}(\mathbb{X}_{n},\hat{\theta}_{n})-{\bf{E}}_{\mathbb{Z}_n}\Bigl[\log {\rm{L}}_{n}(\mathbb{Z}_n,\hat{\theta}_{n})\Bigr]\biggr]
    &={\bf{E}}_{\mathbb{X}_n}\biggl[{\rm{H}}_{n}(\mathbb{X}_n,\hat{\theta}_{n})-{\bf{E}}_{\mathbb{Z}_n}\Bigl[{\rm{H}}_{n}(\mathbb{Z}_n,\hat{\theta}_{n})\Bigr]\biggr]\\
    &={\bf{E}}_{\mathbb{X}_n}\biggl[{\rm{H}}_{n}(\mathbb{X}_n,\hat{\theta}_{n})-{\rm{H}}_{n}(\mathbb{X}_n,\theta_0)\biggr]\\
    &\quad+{\bf{E}}_{\mathbb{X}_n}\biggr[{\rm{H}}_{n}(\mathbb{X}_n,\theta_0)\biggr]
    -{\bf{E}}_{\mathbb{Z}_n}\biggl[{\rm{H}}_{n}(\mathbb{Z}_n,\theta_0)\biggr]\\
    &\quad+{\bf{E}}_{\mathbb{Z}_n}\biggl[{\rm{H}}_{n}(\mathbb{Z}_n,\theta_0)\biggr]-
    {\bf{E}}_{\mathbb{X}_n}\biggl[{\bf{E}}_{\mathbb{Z}_n}\Bigl[{\rm{H}}_{n}(\mathbb{Z}_n,\hat{\theta}_{n})\Bigr]\biggr]\\
    &={\rm{D}}_{1,n}+{\rm{D}}_{2,n}+{\rm{D}}_{3,n},
\end{align*}
where 
\begin{align*}
    {\rm{D}}_{1,n}&={\bf{E}}_{\mathbb{X}_n}\biggl[{\rm{H}}_{n}(\mathbb{X}_n,\hat{\theta}_{n})-{\rm{H}}_{n}(\mathbb{X}_n,\theta_0)\biggr],\\
    {\rm{D}}_{2,n}&={\bf{E}}_{\mathbb{X}_n}\biggr[{\rm{H}}_{n}(\mathbb{X}_n,\theta_0)\biggr]
    -{\bf{E}}_{\mathbb{Z}_n}\biggl[{\rm{H}}_{n}(\mathbb{Z}_n,\theta_0)\biggr],\\
    {\rm{D}}_{3,n}&={\bf{E}}_{\mathbb{Z}_n}\biggl[{\rm{H}}_{n}(\mathbb{Z}_n,\theta_0)\biggr]-
    {\bf{E}}_{\mathbb{X}_n}\biggl[{\bf{E}}_{\mathbb{Z}_n}\Bigl[{\rm{H}}_{n}(\mathbb{Z}_n,\hat{\theta}_{n})\Bigr]\biggr].
\end{align*}
First of all, we will prove
\begin{align}
    {\rm{D}}_{1,n}&\longrightarrow \frac{q}{2}\label{D1}
\end{align}
as $n\longrightarrow\infty$. Using the Taylor expansion, one has
\begin{align*}
    {\rm{H}}_{n}(\mathbb{X}_n,\hat{\theta}_{n})-{\rm{H}}_{n}(\mathbb{X}_n,\theta_0)
    &=\sum_{i=1}^{q}\partial_{\theta^{(i)}}{\rm{H}}_{n}(\mathbb{X}_n,\theta_{0})(\hat{\theta}^{(i)}_{n}-\theta^{(i)}_{0})\\
    &\quad+\frac{1}{2}\sum_{i=1}^q\sum_{j=1}^q\partial_{\theta^{(i)}}\partial_{\theta^{(j)}}
    {\rm{H}}_{n}(\mathbb{X}_n,\theta_{0})(\hat{\theta}^{(i)}_{n}-\theta^{(i)}_{0})(\hat{\theta}^{(j)}_{n}-\theta^{(j)}_{0})\\
    &\quad+\frac{1}{2}\sum_{i=1}^q\sum_{j=1}^q\sum_{k=1}^q
    \left(\int_0^1(1-\lambda)^2\partial_{\theta^{(i)}}\partial_{\theta^{(j)}}\partial_{\theta^{(k)}}
    {\rm{H}}_{n}(\mathbb{X}_n,\tilde{\theta}_{n,\lambda})d\lambda\right)\\
    &\qquad\qquad\qquad\qquad\qquad\qquad\times(\hat{\theta}^{(i)}_{n}-\theta^{(i)}_{0})(\hat{\theta}^{(j)}_{n}-\theta^{(j)}_{0})(\hat{\theta}^{(k)}_{n}-\theta^{(k)}_{0})
    \\
    &={\rm{E}}_{1,n}+{\rm{E}}_{2,n}+{\rm{E}}_{3,n},
\end{align*}
where $\tilde{\theta}_{n,\lambda}=\theta_0+\lambda(\hat{\theta}_n-\theta_0)$ and
\begin{align*}
    {\rm{E}}_{1,n}&=\sum_{i=1}^{q}\partial_{\theta^{(i)}}{\rm{H}}_{n}(\mathbb{X}_n,\theta_{0})(\hat{\theta}^{(i)}_{n}-\theta^{(i)}_{0}),\\
    {\rm{E}}_{2,n}&=\frac{1}{2}\sum_{i=1}^q\sum_{j=1}^q\partial_{\theta^{(i)}}\partial_{\theta^{(j)}}
    {\rm{H}}_{n}(\mathbb{X}_n,\theta_{0})(\hat{\theta}^{(i)}_{n}-\theta^{(i)}_{0})(\hat{\theta}^{(j)}_{n}-\theta^{(j)}_{0}),\\
    {\rm{E}}_{3,n}&=\frac{1}{2}\sum_{i=1}^q\sum_{j=1}^q\sum_{k=1}^q
    \left(\int_0^1(1-\lambda)^2\partial_{\theta^{(i)}}\partial_{\theta^{(j)}}\partial_{\theta^{(k)}}
    {\rm{H}}_{n}(\mathbb{X}_n,\tilde{\theta}_{n,\lambda})d\lambda\right)\\
    &\qquad\qquad\qquad\qquad\qquad\qquad\qquad\quad\times
    (\hat{\theta}^{(i)}_{n}-\theta^{(i)}_{0})(\hat{\theta}^{(j)}_{n}-\theta^{(j)}_{0})(\hat{\theta}^{(k)}_{n}-\theta^{(k)}_{0}).
\end{align*}
First, we consider the expectation of $\rm{E}_{1,n}$. Set 
\begin{align*}
    A_n=\Bigl\{\hat{\theta}_n\in{\rm{int}}\Theta\Bigr\}.
\end{align*}
Note that ${\bf{P}}(A_n)\longrightarrow 1$ as $n\longrightarrow\infty$.
By using the Taylor expansion, one gets
\begin{align*}
    0&=\frac{1}{\sqrt{n}}\partial_{\theta^{(i)}}{\rm{H}}_{n}(\mathbb{X}_n,\hat{\theta}_{n})\\
    &=
    \frac{1}{\sqrt{n}}\partial_{\theta^{(i)}}{\rm{H}}_{n}(\mathbb{X}_n,\theta_0)+\frac{1}{n}\sum_{j=1}^q\partial_{\theta^{(i)}}\partial_{\theta^{(j)}}{\rm{H}}_{n}(\mathbb{X}_n,\theta_0)\hat{u}_n^{(j)}\\
    &\quad+\frac{1}{n\sqrt{n}}\sum_{j=1}^q\sum_{k=1}^q\left(\int_{0}^{1}(1-\lambda)\partial_{\theta^{(i)}}\partial_{\theta^{(j)}}\partial_{\theta^{(k)}}{\rm{H}}_{n}(\mathbb{X}_n,\tilde{\theta}_{n,\lambda})d\lambda\right)\hat{u}_n^{(j)}\hat{u}_n^{(k)}
\end{align*}
for $i=1,\cdots,q$ on $A_n$, so that we have
\begin{align}
    \begin{split}
    {\rm{E}}_{1,n}&=\frac{1}{\sqrt{n}}\sum_{i=1}^q\partial_{\theta^{(i)}}{\rm{H}}_{n}(\mathbb{X}_n,\theta_{0})\hat{u}^{(i)}_{n}\\ 
    &=\sum_{i=1}^q\sum_{j=1}^q {\bf{I}}(\theta_0)_{ij}\hat{u}^{(i)}_n\hat{u}^{(j)}_n-\sum_{i=1}^q{\rm{R}}_{1,n}^{(i)}\hat{u}^{(i)}_{n},
     \end{split}\label{E1taylor}
\end{align}
where
\begin{align*}
    {\rm{R}}_{1,n}^{(i)}=\left\{
    \begin{alignedat}{2}   
    \begin{split}
    &\sum_{j=1}^q\left\{\frac{1}{n}\partial_{\theta^{(i)}}\partial_{\theta^{(j)}}{\rm{H}}_{n}(\mathbb{X}_n,\theta_0)+{\bf{I}}(\theta_0)_{ij}\right\}\hat{u}^{(j)}_n \\
    &\quad+\frac{1}{n\sqrt{n}}\sum_{j=1}^q\sum_{k=1}^q\left(\int_{0}^{1}(1-\lambda)\partial_{\theta^{(i)}}\partial_{\theta^{(j)}}\partial_{\theta^{(k)}}{\rm{H}}_{n}(\mathbb{X}_n,\tilde{\theta}_{n,\lambda})d\lambda\right)\hat{u}_n^{(j)}\hat{u}_n^{(k)},
    \end{split} &\quad\bigl(\mbox{on}\ A_{n}\bigr) \\
    &\sum_{j=1}^q {\bf{I}}(\theta_0)_{ij}\hat{u}^{(j)}_n-\frac{1}{\sqrt{n}}\partial_{\theta^{(i)}}{\rm{H}}_{n}(\mathbb{X}_n,\theta_{0}).
    &\quad\bigl(\mbox{on}\ A_{n}^c\bigr)
  \end{alignedat} 
  \right.
\end{align*}
Let
\begin{align*}
    \bar{{\rm{R}}}_{1,n}^{(i)}&=\sum_{j=1}^q\left\{\frac{1}{n}\partial_{\theta^{(i)}}\partial_{\theta^{(j)}}{\rm{H}}_{n}(\mathbb{X}_n,\theta_0)+{\bf{I}}(\theta_0)_{ij}\right\}\hat{u}^{(j)}_n \\
    &\qquad+\frac{1}{n\sqrt{n}}\sum_{j=1}^q\sum_{k=1}^q\left(\int_{0}^{1}(1-\lambda)\partial_{\theta^{(i)}}\partial_{\theta^{(j)}}\partial_{\theta^{(k)}}{\rm{H}}_{n}(\mathbb{X}_n,\tilde{\theta}_{n,\lambda})d\lambda\right)\hat{u}_n^{(j)}\hat{u}_n^{(k)}.
\end{align*}
Since 
\begin{align*}
    &\quad\ {\bf{E}}_{\mathbb{X}_n}\biggl[\bigl|n^{\frac{1}{4}}\bar{{\rm{R}}}_{1,n}^{(i)}\bigr|^{2}\biggr]\\
    &\leq C\sum_{j=1}^q{\bf{E}}_{\mathbb{X}_n}\left[\Biggl|n^{\frac{1}{4}}\left\{\frac{1}{n}\partial_{\theta^{(i)}}\partial_{\theta^{(j)}}{\rm{H}}_{n}(\mathbb{X}_n,\theta_0)+{\bf{I}}(\theta_0)_{ij}\right\}\hat{u}^{(j)}_n\Biggr|^{2}\right]\\
    &\qquad+\frac{C}{\sqrt{n}}\sum_{j=1}^q\sum_{k=1}^q{\bf{E}}_{\mathbb{X}_n}\left[\Biggl|\frac{1}{n}\left(\int_{0}^{1}(1-\lambda)\partial_{\theta^{(i)}}\partial_{\theta^{(j)}}\partial_{\theta^{(k)}}{\rm{H}}_{n}(\mathbb{X}_n,\tilde{\theta}_{n,\lambda})d\lambda\right)\hat{u}_n^{(j)}\hat{u}_n^{(k)}\Biggr|^{2}\right]\\
    &\leq C\sum_{j=1}^q{\bf{E}}_{\mathbb{X}_n}\left[\biggl(n^{\frac{1}{4}}\left|\frac{1}{n}\partial_{\theta^{(i)}}\partial_{\theta^{(j)}}{\rm{H}}_{n}(\mathbb{X}_n,\theta_0)+{\bf{I}}(\theta_0)_{ij}\right|\biggr)^{4}\right]^{\frac{1}{2}}{\bf{E}}_{\mathbb{X}_n}\biggl[\bigl|\hat{u}^{(j)}_n\bigr|^{4}\biggr]^{\frac{1}{2}}\\
    &\qquad+\frac{C}{\sqrt{n}}\sum_{j=1}^q\sum_{k=1}^q{\bf{E}}_{\mathbb{X}_n}\left[\biggl(\frac{1}{n}\sup_{\theta\in\Theta}\Bigl|\partial_{\theta^{(i)}}\partial_{\theta^{(j)}}\partial_{\theta^{(k)}}{\rm{H}}_{n}(\mathbb{X}_n,\theta)\Bigr|\biggr)^{4}\right]^{\frac{1}{2}}\\
    &\qquad\qquad\qquad\qquad\qquad\qquad\qquad\qquad\qquad\qquad\times
    {\bf{E}}_{\mathbb{X}_n}\biggl[\bigl|\hat{u}^{(j)}_n\bigr|^{8}\biggr]^{\frac{1}{4}}
    {\bf{E}}_{\mathbb{X}_n}\biggl[\bigl|\hat{u}^{(k)}_n\bigr|^{8}\biggr]^{\frac{1}{4}},
\end{align*}
it holds from Lemmas \ref{gammalemma}, \ref{H3lemma} and \ref{moment} that
\begin{align}
    \sup_{n\in\mathbb{N}}{\bf{E}}_{\mathbb{X}_n}\biggl[\bigl|n^{\frac{1}{4}}\bar{{\rm{R}}}_{1,n}^{(i)}\bigr|^{2}\biggr]<\infty. \label{R}
\end{align}
Consequently, we see from (\ref{R}) and Lemma \ref{moment} that
\begin{align*}
    \left|{\bf{E}}_{\mathbb{X}_n}
    \Biggl[\sum_{i=1}^q{\rm{R}}_{1,n}^{(i)}\hat{u}^{(i)}_{n}1_{A_n}\Biggr]\right|
    &=\left|{\bf{E}}_{\mathbb{X}_n}
    \Biggl[\sum_{i=1}^q\bar{{\rm{R}}}_{1,n}^{(i)}\hat{u}^{(i)}_{n}1_{A_n}\Biggr]\right|\\
    &\leq \sum_{i=1}^q{\bf{E}}_{\mathbb{X}_n}\biggl[\bigl|\bar{{\rm{R}}}_{1,n}^{(i)}\bigr|\bigl|\hat{u}^{(i)}_{n}\bigr|\biggr]\\
    &\leq \sum_{i=1}^q{\bf{E}}_{\mathbb{X}_n}\biggl[\bigl|\bar{{\rm{R}}}_{1,n}^{(i)}\bigr|^{2}\biggr]^{\frac{1}{2}}{\bf{E}}_{\mathbb{X}_n}\biggl[\bigl|\hat{u}_n^{(i)}\bigr|^{2}\biggr]^{\frac{1}{2}}\\
    &\leq \frac{1}{n^{\frac{1}{4}}}\sum_{i=1}^q\sup_{n\in\mathbb{N}}{\bf{E}}_{\mathbb{X}_n}\biggl[\bigl|n^{\frac{1}{4}}\bar{{\rm{R}}}_{1,n}^{(i)}\bigr|^{2}\biggr]^{\frac{1}{2}}
    \sup_{n\in\mathbb{N}}{\bf{E}}_{\mathbb{X}_n}
    \biggl[\bigl|\hat{u}_n^{(i)}\bigr|^{2}\biggr]^{\frac{1}{2}}
    \longrightarrow 0
\end{align*}
as $n\longrightarrow\infty$, which yields
\begin{align}
    {\bf{E}}_{\mathbb{X}_n}
    \left[\sum_{i=1}^q{\rm{R}}_{1,n}^{(i)}\hat{u}^{(i)}_{n}1_{A_n}\right]
    \longrightarrow 0 \label{RA}
\end{align}
as $n\longrightarrow\infty$. Set
\begin{align*}
    \underline{{\rm{R}}}_{1,n}^{(i)}=\sum_{j=1}^q {\bf{I}}(\theta_0)_{ij}\hat{u}^{(j)}_n-\frac{1}{\sqrt{n}}\partial_{\theta^{(i)}}{\rm{H}}_{n}(\mathbb{X}_n,\theta_{0}).
\end{align*}
Using Lemmas \ref{deltalemma} and \ref{moment}, we obtain
\begin{align*}
    {\bf{E}}_{\mathbb{X}_n}\biggl[\bigl|\underline{{\rm{R}}}
    _{1,n}^{(i)}\bigr|^2\biggr]&\leq C\sum_{j=1}^q{\bf{E}}_{\mathbb{X}_n}\Biggl[
    \Bigl|{\bf{I}}(\theta_0)_{ij}\hat{u}^{(j)}_n\Bigr|^2\Biggr]+C
    {\bf{E}}_{\mathbb{X}_n}\Biggl[
    \biggl|\frac{1}{\sqrt{n}}\partial_{\theta^{(i)}}{\rm{H}}_{n}(\mathbb{X}_n,\theta_{0})\biggr|^2\Biggr]\\
    &\leq C\sum_{j=1}^q\sup_{n\in\mathbb{N}}{\bf{E}}_{\mathbb{X}_n}\biggl[
    \bigl|\hat{u}^{(j)}_n\bigr|^2\biggr]+C\sup_{n\in\mathbb{N}}{\bf{E}}_{\mathbb{X}_n}\Biggl[
    \biggl|\frac{1}{\sqrt{n}}\partial_{\theta^{(i)}}{\rm{H}}_{n}(\mathbb{X}_n,\theta_{0})\biggr|^2\Biggr]<\infty,
\end{align*}
which implies
\begin{align}
    \sup_{n\in\mathbb{N}}{\bf{E}}_{\mathbb{X}_n}\biggl[\bigl|\underline{{\rm{R}}}
    _{1,n}^{(i)}\bigr|^2\biggr]<\infty. \label{supRbar}
\end{align}
It follows from (\ref{supRbar}) and Lemma \ref{moment} that 
\begin{align*}
    \left|{\bf{E}}_{\mathbb{X}_n}
    \left[\sum_{i=1}^q{\rm{R}}_{1,n}^{(i)}\hat{u}^{(i)}_{n}1_{A_n^c}\right]\right|
    &=\left|{\bf{E}}_{\mathbb{X}_n}\left[\sum_{i=1}^q{\underline{\rm{R}}}_{1,n}^{(i)}\hat{u}^{(i)}_{n}1_{A_n^c}\right]\right|\\
    &\leq \sum_{i=1}^q{\bf{E}}_{\mathbb{X}_n}\biggl[\bigl|\underline{{\rm{R}}}
    _{1,n}^{(i)}\bigr|\bigl|\hat{u}^{(i)}_{n}\bigr|1_{A_n^c}\biggr]\\
    &\leq \sum_{i=1}^q{\bf{E}}_{\mathbb{X}_n}\biggl[\bigl|\underline{{\rm{R}}}
    _{1,n}^{(i)}\bigr|^2\biggr]^{\frac{1}{2}}{\bf{E}}_{\mathbb{X}_n}\biggl[
    \bigl|\hat{u}^{(i)}_{n}\bigr|^4\biggr]^{\frac{1}{4}}{\bf{E}}_{\mathbb{X}_n}\Bigl[1_{A_n^c}\Bigr]^{\frac{1}{4}}\\
    &\leq {\bf{P}}\bigl(A_n^c\bigr)^{\frac{1}{4}}\sum_{i=1}^q\sup_{n\in\mathbb{N}}{\bf{E}}_{\mathbb{X}_n}\biggl[\bigl|\underline{{\rm{R}}}
    _{1,n}^{(i)}\bigr|^2\biggr]^{\frac{1}{2}}\sup_{n\in\mathbb{N}}{\bf{E}}_{\mathbb{X}_n}\biggl[
    \bigl|\hat{u}^{(i)}_{n}\bigr|^4\biggr]^{\frac{1}{4}}\\
    &\longrightarrow 0
\end{align*}
as $n\longrightarrow\infty$, so that one gets
\begin{align}
    {\bf{E}}_{\mathbb{X}_n}
    \left[\sum_{i=1}^q{\rm{R}}_{1,n}^{(i)}\hat{u}^{(i)}_{n}1_{A_n^c}\right]
    \longrightarrow 0 \label{RAc}
\end{align}
as $n\longrightarrow\infty$. Hence, it holds from (\ref{RA}) and (\ref{RAc}) that 
\begin{align}
    \begin{split}
    {\bf{E}}_{\mathbb{X}_n}
    \left[\sum_{i=1}^q{\rm{R}}_{1,n}^{(i)}\hat{u}^{(i)}_{n}\right]&=
    {\bf{E}}_{\mathbb{X}_n}
    \left[\sum_{i=1}^q{\rm{R}}_{1,n}^{(i)}\hat{u}^{(i)}_{n}1_{A_n}\right]+{\bf{E}}_{\mathbb{X}_n}
    \left[\sum_{i=1}^q{\rm{R}}_{1,n}^{(i)}\hat{u}^{(i)}_{n}1_{A_n^c}\right]
    \longrightarrow 0
    \end{split}\label{Ru}
\end{align}
as $n\longrightarrow\infty$. Let 
\begin{align*}
    f_1(x)=x^{\top}{\bf{I}}(\theta_0)x
\end{align*}
for $x\in\mathbb{R}^q$. Since $f_1\in C_{\uparrow}(\mathbb{R}^{q})$, we see from Lemma \ref{moment} that
\begin{align}
    \begin{split}
     {\bf{E}}_{\mathbb{X}_n}\biggl[\hat{u}_n^{\top}
    {\bf{I}}(\theta_0)\hat{u}_n\biggr]
    &={\bf{E}}_{\mathbb{X}_n}\Bigl[f_1\bigl(\hat{u}_n\bigr)\Bigr]\\
    &\longrightarrow \mathbb{E}\biggl[f_1\Bigl({\bf{I}}(\theta_0)^{-\frac{1}{2}}\zeta\Bigr)\biggr]=\mathbb{E}\Bigl[\zeta^{\top}\zeta\Bigr]=q
    \end{split}\label{uIu}
\end{align}
as $n\longrightarrow\infty$. Therefore, (\ref{E1taylor}), (\ref{Ru}) and (\ref{uIu}) show
\begin{align}
\begin{split}
    {\bf{E}}_{\mathbb{X}_n}\Bigl[{\rm{E}}_{1,n}\Bigr]&=
    {\bf{E}}_{\mathbb{X}_n}\left[\sum_{i=1}^q\sum_{j=1}^q {\bf{I}}(\theta_0)_{ij}\hat{u}^{(i)}_n\hat{u}^{(j)}_n\right]
    -{\bf{E}}_{\mathbb{X}_n}
    \left[\sum_{i=1}^q{\rm{R}}_{1,n}^{(i)}\hat{u}^{(i)}_{n}\right]\longrightarrow q
    \label{E1}
\end{split}
\end{align}
as $n\longrightarrow\infty$. Next, the expectation of $\rm{E}_{2,n}$ is considered. Note that
\begin{align*}
    {\rm{E}}_{2,n}
    &=-\frac{1}{2}\sum_{i=1}^q\sum_{j=1}^q {\bf{I}}(\theta_0)_{ij}\hat{u}^{(i)}_n\hat{u}^{(j)}_n+\frac{1}{2}\sum_{i=1}^q\sum_{j=1}^q {\rm{R}}_{2,n,ij}\hat{u}^{(i)}_n\hat{u}^{(j)}_n,
\end{align*}
where
\begin{align*}
    {\rm{R}}_{2,n,ij}=\frac{1}{n}\partial_{\theta^{(i)}}\partial_{\theta^{(j)}}{\rm{H}}_{n}(\mathbb{X}_n,\theta_{0})+{\bf{I}}(\theta_0)_{ij}
\end{align*}
for $i,j=1,\cdots, q$. 
By using Lemmas \ref{gammalemma} and \ref{moment}, it is shown that
\begin{align*}
    &\quad\ \left|{\bf{E}}_{\mathbb{X}_n}
    \Biggl[\sum_{i=1}^q\sum_{j=1}^q {\rm{R}}_{2,n,ij}\hat{u}^{(i)}_n\hat{u}^{(j)}_n\Biggr]\right|\\
    &\leq\sum_{i=1}^q\sum_{j=1}^q{\bf{E}}_{\mathbb{X}_n}\biggl[\bigl|{\rm{R}}_{2,n,ij}\bigr|\bigl|\hat{u}^{(i)}_n\bigr|\bigl|\hat{u}^{(j)}_n\bigr|\biggr]\\
    &\leq\sum_{i=1}^q\sum_{j=1}^q{\bf{E}}_{\mathbb{X}_n}\biggl[\bigl|{\rm{R}}_{2,n,ij}\bigr|^2\biggr]^{\frac{1}{2}}{\bf{E}}_{\mathbb{X}_n}\biggl[\bigl|\hat{u}^{(i)}_n\bigr|^4\biggr]^{\frac{1}{4}}{\bf{E}}_{\mathbb{X}_n}\biggl[\bigl|\hat{u}^{(j)}_n\bigr|^4\biggr]^{\frac{1}{4}}\\
    &\leq\frac{1}{n^{\frac{1}{4}}}\sum_{i=1}^q\sum_{j=1}^q\sup_{n\in\mathbb{N}}{\bf{E}}_{\mathbb{X}_n}\left[\biggl(n^{\frac{1}{4}}\left|\frac{1}{n}\partial_{\theta^{(i)}}\partial_{\theta^{(j)}}{\rm{H}}_{n}(\mathbb{X}_n,\theta_0)+{\bf{I}}(\theta_0)_{ij}\right|\biggr)^{2}\right]^{\frac{1}{2}}\\
    &\qquad\qquad\qquad\qquad\qquad\qquad\qquad\qquad\qquad\times \sup_{n\in\mathbb{N}}{\bf{E}}_{\mathbb{X}_n}
    \biggl[\bigl|\hat{u}_n^{(i)}\bigr|^{4}\biggr]^{\frac{1}{4}}\sup_{n\in\mathbb{N}}{\bf{E}}_{\mathbb{X}_n}
    \biggl[\bigl|\hat{u}_n^{(j)}\bigr|^{4}\biggr]^{\frac{1}{4}}\\
    &\longrightarrow 0
\end{align*}
as $n\longrightarrow\infty$. Thus, it follows from (\ref{uIu}) that
\begin{align}
    \begin{split}
    {\bf{E}}_{\mathbb{X}_n}\Bigl[{\rm{E}}_{2,n}\Bigr]&=-\frac{1}{2}{\bf{E}}_{\mathbb{X}_n}\left[\sum_{i=1}^q\sum_{j=1}^q {\bf{I}}(\theta_0)_{ij}\hat{u}^{(i)}_n\hat{u}^{(j)}_n\right]+\frac{1}{2}{\bf{E}}_{\mathbb{X}_n}
    \left[\sum_{i=1}^q\sum_{j=1}^q {\rm{R}}_{2,n,ij}\hat{u}^{(i)}_n\hat{u}^{(j)}_n\right]\\
    &\longrightarrow -\frac{q}{2}\label{E2}
    \end{split}
\end{align}
as $n\longrightarrow\infty$. Note that 
\begin{align*}
    {\rm{E}}_{3,n}&=\sum_{i=1}^q\sum_{j=1}^q\sum_{k=1}^q{\rm{R}}_{3,n,ijk}
    \hat{u}_n^{(i)}\hat{u}_n^{(j)}\hat{u}_n^{(k)},
\end{align*}
where
\begin{align*}
    {\rm{R}}_{3,n,ijk}=\frac{1}{2n\sqrt{n}}\int_0^1(1-\lambda)^2\partial_{\theta^{(i)}}\partial_{\theta^{(j)}}\partial_{\theta^{(k)}}
    {\rm{H}}_{n}(\mathbb{X}_n,\tilde{\theta}_{n,\lambda})d\lambda
\end{align*}
for $i,j,k=1,\cdots,q$. It holds from Lemmas \ref{H3lemma} and \ref{moment} that
\begin{align*}
    \biggl|{\bf{E}}_{\mathbb{X}_n}\Bigl[{\rm{E}}_{3,n}\Bigr]\biggr|
    &\leq\frac{1}{\sqrt{n}}\sum_{i=1}^q\sum_{j=1}^q\sum_{k=1}^q{\bf{E}}_{\mathbb{X}_n}\left[\left|\frac{1}{n}\int_0^1(1-\lambda)^2\partial_{\theta^{(i)}}\partial_{\theta^{(j)}}\partial_{\theta^{(k)}}
    {\rm{H}}_{n}(\mathbb{X}_n,\tilde{\theta}_{n,\lambda})d\lambda\right|^2\right]^{\frac{1}{2}}\\
    &\qquad\qquad\qquad\qquad\qquad\qquad\quad
    \times{\bf{E}}_{\mathbb{X}_n}
    \biggl[\bigl|\hat{u}_n^{(i)}\bigr|^{4}\biggr]^{\frac{1}{4}}{\bf{E}}_{\mathbb{X}_n}
    \biggl[\bigl|\hat{u}_n^{(j)}\bigr|^{8}\biggr]^{\frac{1}{8}}{\bf{E}}_{\mathbb{X}_n}
    \biggl[\bigl|\hat{u}_n^{(k)}\bigr|^{8}\biggr]^{\frac{1}{8}}\\
    &\leq\frac{1}{\sqrt{n}}\sum_{i=1}^q\sum_{j=1}^q\sum_{k=1}^q\sup_{n\in\mathbb{N}}{\bf{E}}_{\mathbb{X}_n}\left[\biggl(\frac{1}{n}\sup_{\theta\in\Theta}\Bigr|
    \partial_{\theta^{(i)}}\partial_{\theta^{(j)}}\partial_{\theta^{(k)}}
    {\rm{H}}_{n}(\mathbb{X}_n,\theta)\Bigl|\biggr)^2\right]^{\frac{1}{2}}\\
    &\qquad\qquad\qquad\quad\ \ 
    \times\sup_{n\in\mathbb{N}}{\bf{E}}_{\mathbb{X}_n}
    \biggl[\bigl|\hat{u}_n^{(i)}\bigr|^{4}\biggr]^{\frac{1}{4}}\sup_{n\in\mathbb{N}}{\bf{E}}_{\mathbb{X}_n}
    \biggl[\bigl|\hat{u}_n^{(j)}\bigr|^{8}\biggr]^{\frac{1}{8}}\sup_{n\in\mathbb{N}}{\bf{E}}_{\mathbb{X}_n}
    \biggl[\bigl|\hat{u}_n^{(k)}\bigr|^{8}\biggr]^{\frac{1}{8}}\\
    &\longrightarrow 0
\end{align*}
as $n\longrightarrow\infty$, which yields 
\begin{align}
    {\bf{E}}_{\mathbb{X}_n}\Bigl[{\rm{E}}_{3,n}\Bigr]\longrightarrow 0 \label{E3}
\end{align}
as $n\longrightarrow\infty$. Hence, (\ref{E1}), (\ref{E2}) and (\ref{E3}) show (\ref{D1}). Next, we will prove
\begin{align}
    {\rm{D}}_{3,n}&\longrightarrow \frac{q}{2}\label{D3}
\end{align}
as $n\longrightarrow\infty$. By using the Taylor expansion, one gets
\begin{align*}
    {\rm{H}}_n(\mathbb{Z}_n,\hat{\theta}_n)-{\rm{H}}_n(\mathbb{Z}_n,\theta_0)
    &=\sum_{i=1}^q\partial_{\theta^{(i)}}{\rm{H}}_n(\mathbb{Z}_n,\theta_0)(\hat{\theta}^{(i)}_{n}-\theta^{(i)}_{0})\\
    &\quad+\frac{1}{2}\sum_{i=1}^q\sum_{j=1}^q\partial_{\theta^{(i)}}\partial_{\theta^{(j)}}
    {\rm{H}}_n(\mathbb{Z}_n,\theta_0)(\hat{\theta}^{(i)}_{n}-\theta^{(i)}_{0})(\hat{\theta}^{(j)}_{n}-\theta^{(j)}_{0})\\
    &\quad+\frac{1}{2}\sum_{i=1}^q\sum_{j=1}^q\sum_{k=1}^q\left(\int_0^1(1-\lambda)^2\partial_{\theta^{(i)}}\partial_{\theta^{(j)}}\partial_{\theta^{(k)}}{\rm{H}}_n(\mathbb{Z}_n,\tilde{\theta}_{n,\lambda})d\lambda\right)\\
    &\qquad\qquad\qquad\qquad\qquad\quad\times
    (\hat{\theta}^{(i)}_{n}-\theta^{(i)}_{0})(\hat{\theta}^{(j)}_{n}-\theta^{(j)}_{0})(\hat{\theta}^{(k)}_{n}-\theta^{(k)}_{0})\\
    &={\rm{F}}_{1,n}+{\rm{F}}_{2,n}+{\rm{F}}_{3,n},
\end{align*}
where
\begin{align*}
    {\rm{F}}_{1,n}&=\sum_{i=1}^q\partial_{\theta^{(i)}}
    {\rm{H}}_n(\mathbb{Z}_n,\theta_0)(\hat{\theta}^{(i)}_{n}-\theta^{(i)}_{0}),\\
    {\rm{F}}_{2,n}&=\frac{1}{2}\sum_{i=1}^q\sum_{j=1}^q\partial_{\theta^{(i)}}\partial_{\theta^{(j)}}
    {\rm{H}}_n(\mathbb{Z}_n,\theta_0)(\hat{\theta}^{(i)}_{n}-\theta^{(i)}_{0})(\hat{\theta}^{(j)}_{n}-\theta^{(j)}_{0}),\\
    {\rm{F}}_{3,n}&=\frac{1}{2}\sum_{i=1}^q\sum_{j=1}^q\sum_{k=1}^q\left(\int_0^1(1-\lambda)^2\partial_{\theta^{(i)}}\partial_{\theta^{(j)}}\partial_{\theta^{(k)}}{\rm{H}}_n(\mathbb{Z}_n,\tilde{\theta}_{n,\lambda})d\lambda\right)\\
    &\qquad\qquad\qquad\qquad\qquad\qquad\qquad\qquad\times (\hat{\theta}^{(i)}_{n}-\theta^{(i)}_{0})(\hat{\theta}^{(j)}_{n}-\theta^{(j)}_{0})(\hat{\theta}^{(k)}_{n}-\theta^{(k)}_{0}).
\end{align*}
Since it holds from Lemmas \ref{thetaprob1}, \ref{deltalemma} and \ref{moment} that
\begin{align*}
    {\bf{E}}_{\mathbb{Z}_n}\left[\frac{1}{\sqrt{n}}
    \partial_{\theta}{\rm{H}}_n(\mathbb{Z}_n,\theta_0)\right]
    \longrightarrow \mathbb{E}\Bigl[{\bf{I}}(\theta_0)^{\frac{1}{2}}\zeta\Bigr]
    =0
\end{align*}
and
\begin{align*}
    {\bf{E}}_{\mathbb{X}_n}\Bigl[\sqrt{n}(\hat{\theta}_{n}-\theta_{0})\Bigr]
    &\longrightarrow \mathbb{E}\Bigl[{\bf{I}}(\theta_0)^{-\frac{1}{2}}\zeta\Bigr]
    =0
\end{align*}
as $n\longrightarrow\infty$, we have
\begin{align}
    {\bf{E}}_{\mathbb{X}_n}
    \biggl[{\bf{E}}_{\mathbb{Z}_n}\Bigl[{\rm{F}}_{1,n}\Bigr]\biggr]
    &=\sum_{i=1}^q{\bf{E}}_{\mathbb{Z}_n}\Biggl[\frac{1}{\sqrt{n}}
    \partial_{\theta^{(i)}}{\rm{H}}_n(\mathbb{Z}_n,\theta_0)\Biggr]
    {\bf{E}}_{\mathbb{X}_n}\biggl[\sqrt{n}(\hat{\theta}^{(i)}_{n}-\theta^{(i)}_{0})\biggr]\longrightarrow 0
    \label{F1}
\end{align}
as $n\longrightarrow\infty$. Note that
\begin{align*}
    {\rm{F}}_{2,n}&=\frac{1}{2}\hat{u}_n^{\top}\biggl(\frac{1}{n}\partial^2_{\theta}{\rm{H}}_n(\mathbb{Z}_n,\theta_0)\biggr)\hat{u}_n
    =\frac{1}{2}\tr\Biggr\{\biggl(\frac{1}{n}\partial^2_{\theta}{\rm{H}}_n(\mathbb{Z}_n,\theta_0)
    \biggr)\hat{u}_n\hat{u}_n^{\top}\Biggr\}.
\end{align*}
Let 
\begin{align*}
    f_2(x)=xx^{\top}
\end{align*}
for $x\in\mathbb{R}^q$. Lemmas \ref{thetaprob1}, \ref{deltalemma} and \ref{moment}  deduce
\begin{align*}
    {\bf{E}}_{\mathbb{Z}_n}\Biggl[\frac{1}{n}\partial^2_{\theta}{\rm{H}}_n(\mathbb{Z}_n,\theta_0)\Biggr]\longrightarrow-{\bf{I}}(\theta_0)
\end{align*}
and
\begin{align*}
    {\bf{E}}_{\mathbb{X}_n}
    \Bigl[\hat{u}_n\hat{u}_n^{\top}\Bigr]
    &={\bf{E}}_{\mathbb{X}_n}
    \Bigl[f_2\bigl(\hat{u}_n\bigr)\Bigr]\longrightarrow\mathbb{E}\biggl[f_{2}\Bigl({\bf{I}}(\theta_0)^{-\frac{1}{2}}\zeta\Bigr)\biggr]
    ={\bf{I}}(\theta_0)^{-1}
\end{align*}
as $n\longrightarrow\infty$, which implies
\begin{align}
    {\bf{E}}_{\mathbb{X}_n}\biggl[{\bf{E}}_{\mathbb{Z}_n}\Bigl[{\rm{F}}_{2,n}\Bigr]\biggr]
    &=\frac{1}{2}\tr\left\{{\bf{E}}_{\mathbb{Z}_n}\Biggl[\frac{1}{n}\partial^2_{\theta}{\rm{H}}_n(\mathbb{Z}_n,\theta_0)\Biggr]{\bf{E}}_{\mathbb{X}_n}
    \Bigl[\hat{u}_n\hat{u}_n^{\top}\Bigr]\right\}\longrightarrow-\frac{q}{2}\label{F2}
\end{align}
as $n\longrightarrow\infty$. Moreover, we note that
\begin{align*}
    {\rm{F}}_{3,n}=\sum_{i=1}^q\sum_{j=1}^q\sum_{k=1}^q \tilde{{\rm{R}}}_{3,n,ijk}\hat{u}_n^{(i)}\hat{u}_n^{(j)}\hat{u}_n^{(k)},
\end{align*}
where
\begin{align*}
    \tilde{{\rm{R}}}_{3,n,ijk}=\frac{1}{2n\sqrt{n}}\int_{0}^{1}(1-\lambda)^2\partial_{\theta^{(i)}}\partial_{\theta^{(j)}}\partial_{\theta^{(k)}}
    {\rm{H}}_n(\mathbb{Z}_n,
    \tilde{\theta}_{n,\lambda})d\lambda
\end{align*}
for $i,j,k=1,\cdots,q$. Since
\begin{align*}
    {\bf{E}}_{\mathbb{X}_n}\biggl[\Bigl|{\bf{E}}_{\mathbb{Z}_n}\Bigl[\tilde{{\rm{R}}}_{3,n,ijk}\Bigr]\Bigr|^2\biggr]
    &\leq{\bf{E}}_{\mathbb{X}_n}\biggl[{\bf{E}}_{\mathbb{Z}_n}\Bigl[
    \bigl|\tilde{{\rm{R}}}_{3,n,ijk}\bigr|^2\Bigr]\biggr]\\
    &\leq\frac{1}{n}{\bf{E}}_{\mathbb{X}_n}\left[{\bf{E}}_{\mathbb{Z}_n}\left[\left|\frac{1}{n}\int_{0}^{1}(1-\lambda)^2\partial_{\theta^{(i)}}\partial_{\theta^{(j)}}\partial_{\theta^{(k)}}
    {\rm{H}}_n(\mathbb{Z}_n,
    \tilde{\theta}_{n,\lambda})d\lambda\right|^2\right]\right]\\
    &\leq\frac{1}{n}\sup_{n\in\mathbb{N}}
    {\bf{E}}_{\mathbb{Z}_n}\left[\biggl(\frac{1}{n}\sup_{\theta\in\Theta}\Bigl|
    \partial_{\theta^{(i)}}\partial_{\theta^{(j)}}\partial_{\theta^{(k)}}{\rm{H}}_n(\mathbb{Z}_n,\theta)
    \Bigr|\biggr)^2\right],
\end{align*}
it follows from Lemmas \ref{H3lemma} and \ref{moment} that
\begin{align*}
    \Biggl|{\bf{E}}_{\mathbb{X}_n}\biggl[{\bf{E}}_{\mathbb{Z}_n}\Bigl[{\rm{F}}_{3,n}\Bigr]
    \biggr]\Biggr|
    &\leq\sum_{i=1}^q\sum_{j=1}^q\sum_{k=1}^q
    {\bf{E}}_{\mathbb{X}_n}\biggl[\Bigl|{\bf{E}}_{\mathbb{Z}_n}\Bigl[\tilde{{\rm{R}}}_{3,n,ijk}\Bigr]\Bigr|\bigl|
    \hat{u}_n^{(i)}\bigr|\bigl|\hat{u}_n^{(j)}\bigr|\bigl|\hat{u}_n^{(k)}\bigr|\biggr]\\
    &\leq\sum_{i=1}^q\sum_{j=1}^q\sum_{k=1}^q
    {\bf{E}}_{\mathbb{X}_n}\biggl[\Bigr|{\bf{E}}_{\mathbb{Z}_n}\Bigl[ \tilde{{\rm{R}}}_{3,n,ijk}\Bigr]\Bigr|^2\biggr]^{\frac{1}{2}}
    {\bf{E}}_{\mathbb{X}_n}\biggl[\bigl|\hat{u}_n^{(i)}\bigr|^4\biggr]^{\frac{1}{4}}\\
    &\qquad\qquad\qquad\qquad\qquad\qquad\qquad\quad\times {\bf{E}}_{\mathbb{X}_n}\biggl[\bigl|\hat{u}_n^{(j)}\bigr|^8\biggr]^{\frac{1}{8}}{\bf{E}}_{\mathbb{X}_n}\biggl[\bigl|\hat{u}_n^{(k)}\bigr|^8\biggr]^{\frac{1}{8}}\\
    &\leq\frac{1}{\sqrt{n}}\sum_{i=1}^q\sum_{j=1}^q\sum_{k=1}^q\sup_{n\in\mathbb{N}}
    {\bf{E}}_{\mathbb{Z}_n}\left[\biggl(\frac{1}{n}\sup_{\theta\in\Theta}\Bigl|
    \partial_{\theta^{(i)}}\partial_{\theta^{(j)}}\partial_{\theta^{(k)}}{\rm{H}}_n(\mathbb{Z}_n,\theta)\Bigr|
    \biggr)^2\right]^{\frac{1}{2}}\\
    &\qquad\quad
    \times\sup_{n\in\mathbb{N}}{\bf{E}}_{\mathbb{X}_n}\biggl[\bigl|\hat{u}_n^{(i)}\bigr|^4\biggr]^{\frac{1}{4}}\sup_{n\in\mathbb{N}}{\bf{E}}_{\mathbb{X}_n}\biggl[\bigl|\hat{u}_n^{(j)}\bigr|^8\biggr]^{\frac{1}{8}}\sup_{n\in\mathbb{N}}{\bf{E}}_{\mathbb{X}_n}\biggl[\bigl|\hat{u}_n^{(k)}\bigr|^8\biggr]^{\frac{1}{8}}\\
    &\longrightarrow 0
\end{align*}
as $n\longrightarrow\infty$, so that one has
\begin{align}
    {\bf{E}}_{\mathbb{X}_n}\biggl[{\bf{E}}_{\mathbb{Z}_n}\Bigl[{\rm{F}}_{3,n}\Bigr]
    \biggr]\longrightarrow 0 \label{F3}
\end{align}
as $n\longrightarrow\infty$. Consequently, we see from (\ref{F1}), (\ref{F2}) and (\ref{F3}) that
\begin{align*}
    {\rm{D}}_{3,n}=-{\bf{E}}_{\mathbb{X}_n}\biggl[{\bf{E}}_{\mathbb{Z}_n}\Bigl[{\rm{F}}_{1,n}+{\rm{F}}_{2,n}+{\rm{F}}_{3,n}\Bigr]\biggr]\longrightarrow \frac{q}{2}
\end{align*}
as $n\longrightarrow\infty$, which yields (\ref{D3}). Furthermore, we have
\begin{align}
    \rm{D}_{2,n}=0 \label{D2}
\end{align}
since $\mathbb{X}_n$ and $\mathbb{Z}_n$ have the same distribution. Therefore, it holds from (\ref{D1}), (\ref{D3}) and (\ref{D2}) that
\begin{align*}
    {\bf{E}}_{\mathbb{X}_n}\biggl[\log {\rm{L}}_{n}(\mathbb{X}_{n},\hat{\theta}_{n})-{\bf{E}}_{\mathbb{Z}_n}\Bigl[\log {\rm{L}}_{n}(\mathbb{Z}_n,\hat{\theta}_{n})\Bigr]\biggr]=q+o_p(1)
\end{align*}
as $n\longrightarrow\infty$.
\end{proof}
\begin{lemma}\label{mis}
Under {\bf{[A]}} and {\bf{[B2]}}, as $n\longrightarrow\infty$,
\begin{align*}
    \frac{1}{n}{\rm{H}}_{m,n}(\mathbb{X}_n,\hat{\theta}_{m,n})\stackrel{p}{\longrightarrow}
    {\rm{H}}_{m}(\bar{\theta}_{m}).
\end{align*}
\end{lemma}
\begin{proof}[Proof of Lemma \ref{mis}]
     In an analogous manner to the proof of Theorem 4 in Kusano and Uchida \cite{Kusano(JJSD)}, we can obtain the result.
     See also Appendix \ref{appendix lemma}.
\end{proof}
\begin{proof}[\textbf{Proof of Theorem 2}]
Fix $m^*\in\mathcal{M}$. From the definition of $\hat{m}_n$, one has
\begin{align}
    \begin{split}
    {\bf{P}}\Bigl(\hat{m}_{n}\in\mathcal{M}^{c}\Bigr)&\leq
    {\bf{P}}\left(\min_{m_1\in\mathcal{M}^{c}} {\rm{QAIC}}(\mathbb{X}_n,m_1)<\min_{m_2\in\mathcal{M}} {\rm{QAIC}}(\mathbb{X}_n,m_2)\right)\\
    &={\bf{P}}\left(\bigcup_{m_1\in\mathcal{M}^c}\Bigl\{ {\rm{QAIC}}(\mathbb{X}_n,m_1)<\min_{m_2\in\mathcal{M}}{\rm{QAIC}}(\mathbb{X}_n,m_2)\Bigr\}\right)\\
    &\leq \sum_{m_1\in\mathcal{M}^c}{\bf{P}}\left({\rm{QAIC}}(\mathbb{X}_n,m_1)<\min_{m_2\in\mathcal{M}}{\rm{QAIC}}(\mathbb{X}_n,m_2)\right)\\
    &= \sum_{m_1\in\mathcal{M}^c}{\bf{P}}
    \left(\bigcap_{m_2\in\mathcal{M}}\Bigl\{ {\rm{QAIC}}(\mathbb{X}_n,m_1)<{\rm{QAIC}}(\mathbb{X}_n,m_2)\Bigr\}\right)\\
    &\leq\sum_{m_1\in\mathcal{M}^c}
    {\bf{P}}\Bigl({\rm{QAIC}}(\mathbb{X}_n,m_1)<{\rm{QAIC}}(\mathbb{X}_n,m^*)\Bigr).
    \end{split}\label{AICine}
\end{align}
For all $m_1\in\mathcal{M}^c$, it follows from Lemma \ref{mis} that
\begin{align}
    \begin{split}
    &\quad\ \frac{1}{n}{\rm{QAIC}}(\mathbb{X}_n,m_1)-\frac{1}{n}{\rm{QAIC}}(\mathbb{X}_n,m^*)\\
    &=-\frac{2}{n}\log{\rm{L}}_{m_1,n}(\mathbb{X}_n,\hat{\theta}_{m_1,n})+\frac{2}{n}q_{m_1}
    +\frac{2}{n}\log{\rm{L}}_{m^*,n}(\mathbb{X}_n,\hat{\theta}_{m^*,n})-\frac{2}{n}q_{m^*}\\
    &=-\frac{2}{n}{\rm{H}}_{m_1,n}(\mathbb{X}_n,\hat{\theta}_{m_1,n})+\frac{2}{n}{\rm{H}}_{m^*,n}(\mathbb{X}_n,\hat{\theta}_{m^*,n})+\frac{2}{n}q_{m_1}-\frac{2}{n}q_{m^*}\\
    &\stackrel{p}{\longrightarrow}c_{m_1,m^*}
    \end{split}\label{AICprob}
\end{align}
as $n\longrightarrow\infty$, where
\begin{align*}
    c_{m_1,m^*}=-2{\rm{H}}_{m_1}(\bar{\theta}_{m_1})+2{\rm{H}}_{m^*}(\theta_{m^*,0}).
\end{align*}
Define the function ${\rm{G}}:\mathcal{M}_p^{++}\longrightarrow\mathbb{R}$:
\begin{align*}
    {\rm{G}}({\bf{\Sigma}})=-\frac{1}{2}\tr\bigl({\bf{\Sigma}}^{-1}{\bf{\Sigma}}_0\bigr)-\frac{1}{2}\log\det {\bf{\Sigma}}.
\end{align*}
Note that ${\rm{G}}({\bf{\Sigma}})$ has the unique maximum point at ${\bf{\Sigma}}={\bf{\Sigma}}_0$. Since
\begin{align*}
    {\bf{\Sigma}}_0={\bf{\Sigma}}_{m^*}(\theta_{m^*,0})\neq{\bf{\Sigma}}_{m_1}(\bar{\theta}_{m_1}),
\end{align*}
we obtain
\begin{align*}
    {\rm{H}}_{m_1}(\bar{\theta}_{m_1})={\rm{G}}\bigl({\bf{\Sigma}}_{m_1}(\bar{\theta}_{m_1})\bigr)<{\rm{G}}\bigl({\bf{\Sigma}}_{m^*}(\theta_{m^*,0})\bigr)={\rm{H}}_{m^*}(\theta_{m^*,0})
\end{align*}
for any $m_1\in\mathcal{M}^c$, which yields $c_{m_1,m^*}>0$. Consequently, it holds from (\ref{AICprob}) that for all $m_1\in\mathcal{M}^c$,
\begin{align*}
    0&\leq{\bf{P}}\Bigl({\rm{QAIC}}(\mathbb{X}_n,m_1)<{\rm{QAIC}}(\mathbb{X}_n,m^*)\Bigr)\\
    &={\bf{P}}\left(\frac{1}{n}{\rm{QAIC}}(\mathbb{X}_n,m_1)-\frac{1}{n}{\rm{QAIC}}(\mathbb{X}_n,m^*)-c_{m_1,m^*}<-c_{m_1,m^*}\right)\\
    &={\bf{P}}\left(c_{m_1,m^*}-\frac{1}{n}{\rm{QAIC}}(\mathbb{X}_n,m_1)+\frac{1}{n}{\rm{QAIC}}(\mathbb{X}_n,m^*)>c_{m_1,m^*}\right)\\
    &\leq{\bf{P}}\left(\Bigl|c_{m_1,m^*}-\frac{1}{n}{\rm{QAIC}}(\mathbb{X}_n,m_1)+\frac{1}{n}{\rm{QAIC}}(\mathbb{X}_n,m^*)\Bigr|>c_{m_1,m^*}\right)\\
    &={\bf{P}}\left(\Bigl|\frac{1}{n}{\rm{QAIC}}(\mathbb{X}_n,m_1)-\frac{1}{n}{\rm{QAIC}}(\mathbb{X}_n,m^*)-c_{m_1,m^*}\Bigr|>c_{m_1,m^*}\right)\longrightarrow 0
\end{align*}
as $n\longrightarrow \infty$, which implies
\begin{align}
    \sum_{m_1\in\mathcal{M}^c}
    {\bf{P}}\Bigl({\rm{QAIC}}(\mathbb{X}_n,m_1)<{\rm{QAIC}}(\mathbb{X}_n,m^*)\Bigr)\longrightarrow 0
    \label{sumconv}
\end{align}
as $n\longrightarrow \infty$. Therefore, we see from (\ref{AICine}) and (\ref{sumconv}) that
\begin{align*}
    {\bf{P}}\Bigl(\hat{m}_{n}\in\mathcal{M}^{c}\Bigr)\longrightarrow 0
\end{align*}
as $n\longrightarrow \infty$.
\end{proof}

\newpage
\section{Appendix}
\subsection{Proofs of Lemmas}\label{appendix lemma}
\begin{proof}[\textbf{Proof of Lemma \ref{thetaprob1}}]
Since
\begin{align*}
    \partial_{\theta^{(i)}}{\rm{H}}_n(\theta)&=-\frac{n}{2}\partial_{\theta^{(i)}}\log\det{\bf{\Sigma}}(\theta)-
    \frac{n}{2}\partial_{\theta^{(i)}}\tr\Bigl({\bf{\Sigma}}(\theta)^{-1}\mathbb{Q}_{\mathbb{XX}}\Bigr)\\
    &=-\frac{n}{2}\tr\biggl\{\Bigl({\bf{\Sigma}}(\theta)^{-1}\Bigr)\Bigl(\partial_{\theta^{(i)}}{\bf{\Sigma}}(\theta)\Bigr)\biggr\}\\
    &\qquad\qquad\qquad+\frac{n}{2}\tr\biggl\{\Bigl({\bf{\Sigma}}(\theta)^{-1}\Bigr)\Bigl(\partial_{\theta^{(i)}}{\bf{\Sigma}}(\theta)\Bigr)\Bigl({\bf{\Sigma}}(\theta)^{-1}\Bigr)\mathbb{Q}_{\mathbb{XX}}\biggr\}
\end{align*}
for $i=1,\cdots, q$, it is shown that
\begin{align*}
    &\quad\ \frac{1}{\sqrt{n}}\partial_{\theta^{(i)}}{\rm{H}}_n(\theta_0)\\
    &=\frac{\sqrt{n}}{2}\tr\biggl\{\Bigl({\bf{\Sigma}}(\theta_0)^{-1}\Bigr)\Bigl(\partial_{\theta^{(i)}}{\bf{\Sigma}}(\theta_0)\Bigr)\Bigl({\bf{\Sigma}}(\theta_0)^{-1}\Bigr)\mathbb{Q}_{\mathbb{XX}}\biggr\}\\
    &\qquad\qquad\qquad\qquad\qquad\qquad\qquad\qquad\quad-\frac{\sqrt{n}}{2}\tr\biggl\{\Bigl({\bf{\Sigma}}(\theta_0)^{-1}\Bigr)\Bigl(\partial_{\theta^{(i)}}{\bf{\Sigma}}(\theta_0)\Bigr)\biggr\}\\
    &=\frac{1}{2}\tr\biggl\{\Bigl({\bf{\Sigma}}(\theta_0)^{-1}\Bigr)\Bigl(\partial_{\theta^{(i)}}{\bf{\Sigma}}(\theta_0)\Bigr)\Bigl({\bf{\Sigma}}(\theta_0)^{-1}\Bigr)\sqrt{n}\Bigl(\mathbb{Q}_{\mathbb{XX}}-{\bf{\Sigma}}(\theta_0)\Bigr)\biggr\}\\
    &=\frac{1}{2}\Bigl(\vec{\partial_{\theta^{(i)}}{\bf{\Sigma}}(\theta_0)}\Bigr)^{\top}\Bigl({\bf{\Sigma}}(\theta_0)^{-1}\otimes{\bf{\Sigma}}(\theta_0)^{-1}\Bigr)\sqrt{n}\Bigl(\vec \mathbb{Q}_{\mathbb{XX}}-\vec {\bf{\Sigma}}(\theta_0)\Bigr)\\
    &=\frac{1}{2}\Bigl(\vech{\partial_{\theta^{(i)}}{\bf{\Sigma}}(\theta_0)}\Bigr)^{\top}\mathbb{D}_p^{\top}\Bigl({\bf{\Sigma}}(\theta_0)^{-1}\otimes{\bf{\Sigma}}(\theta_0)^{-1}\Bigr)\mathbb{D}_p\sqrt{n}\Bigl(\vech \mathbb{Q}_{\mathbb{XX}}-\vech {\bf{\Sigma}}(\theta_0)\Bigr)\\
    &=\Bigl(\partial_{\theta^{(i)}}\vech{{\bf{\Sigma}}(\theta_0)}\Bigr)^{\top}{\bf{W}}(\theta_0)^{-1}\sqrt{n}\Bigl(\vech \mathbb{Q}_{\mathbb{XX}}-\vech {\bf{\Sigma}}(\theta_0)\Bigr).
\end{align*}
Thus, we see from Theorem 1 in Kusano and Uchida \cite{Kusano(2023)} that
\begin{align*}
    \frac{1}{\sqrt{n}}\partial_{\theta}{\rm{H}}_n(\theta_0)&=\Delta_0^{\top}{\bf{W}}(\theta_0)^{-1}\sqrt{n}\Bigl(\vech \mathbb{Q}_{\mathbb{XX}}-\vech {\bf{\Sigma}}(\theta_0)\Bigr)\\
    &\qquad\qquad\qquad\stackrel{d}{\longrightarrow}N_{q}\Bigl(0,\Delta_0^{\top}{\bf{W}}(\theta_0)^{-1}\Delta_0\Bigr)\sim {\bf{I}}(\theta_0)^{\frac{1}{2}}\zeta.
\end{align*}
Note that
\begin{align*}
    \frac{1}{n}\partial_{\theta^{(i)}}\partial_{\theta^{(j)}}{\rm{H}}_n(\theta)&=\frac{1}{2}\tr\biggl\{\Bigl({\bf{\Sigma}}(\theta)^{-1}\Bigr)\Bigl(\partial_{\theta^{(i)}}{\bf{\Sigma}}(\theta)\Bigr)\Bigl({\bf{\Sigma}}(\theta)^{-1}\Bigr)\Bigl(\partial_{\theta^{(j)}}{\bf{\Sigma}}(\theta)\Bigr)\biggr\}\\
    &\quad-\frac{1}{2}\tr\biggl\{\Bigl({\bf{\Sigma}}(\theta)^{-1}\Bigr)\Bigl(\partial_{\theta^{(i)}}\partial_{\theta^{(j)}}{\bf{\Sigma}}(\theta)\Bigr)\biggr\}\\
    &\quad-\frac{1}{2}\tr\biggl\{\Bigl({\bf{\Sigma}}(\theta)^{-1}\Bigr)\Bigl(\partial_{\theta^{(i)}}{\bf{\Sigma}}(\theta)\Bigr)\Bigl({\bf{\Sigma}}(\theta)^{-1}\Bigr)\Bigl(\partial_{\theta^{(j)}}{\bf{\Sigma}}(\theta)\Bigr)\Bigl({\bf{\Sigma}}(\theta)^{-1}\Bigr)\mathbb{Q}_{\mathbb{XX}}\biggr\}\\
    &\quad+\frac{1}{2}\tr\biggl\{\Bigl({\bf{\Sigma}}(\theta)^{-1}\Bigr)\Bigl(\partial_{\theta^{(i)}}\partial_{\theta^{(j)}}{\bf{\Sigma}}(\theta)\Bigr)\Bigl({\bf{\Sigma}}(\theta)^{-1}\Bigr)\mathbb{Q}_{\mathbb{XX}}\biggr\}\\
    &\quad-\frac{1}{2}\tr\biggl\{\Bigl({\bf{\Sigma}}(\theta)^{-1}\Bigr)\Bigl(\partial_{\theta^{(j)}}{\bf{\Sigma}}(\theta)\Bigr)\Bigl({\bf{\Sigma}}(\theta)^{-1}\Bigr)\Bigl(\partial_{\theta^{(i)}}{\bf{\Sigma}}(\theta)\Bigr)\Bigl({\bf{\Sigma}}(\theta)^{-1}\Bigr)\mathbb{Q}_{\mathbb{XX}}\biggr\}
\end{align*}
for $i,j=1,\cdots,q$. It follows from Theorem 1 in Kusano and Uchida \cite{Kusano(2023)} and Slutsky Theorem that
\begin{align*}
    \frac{1}{n}\partial_{\theta^{(i)}}\partial_{\theta^{(j)}}{\rm{H}}_n(\theta_0)
    &\stackrel{p}{\longrightarrow}-\frac{1}{2}
    \tr\biggl\{\Bigl({\bf{\Sigma}}(\theta_0)^{-1}\Bigr)\Bigl(\partial_{\theta^{(i)}}{\bf{\Sigma}}(\theta_0)\Bigr)\Bigl({\bf{\Sigma}}(\theta_0)^{-1}\Bigr)\Bigl(\partial_{\theta^{(j)}}{\bf{\Sigma}}(\theta_0)\Bigr)\biggr\}\\
    &\quad=-\frac{1}{2}\Bigl(\vec\partial_{\theta^{(i)}}{\bf{\Sigma}}(\theta_0)\Bigr)^{\top}\Bigl({\bf{\Sigma}}(\theta_0)^{-1}\otimes{\bf{\Sigma}}(\theta_0)^{-1}\Bigr)\Bigl(\vec\partial_{\theta^{(j)}}{\bf{\Sigma}}(\theta_0)\Bigr)\\
    &\quad=-\frac{1}{2}\Bigl(\vech\partial_{\theta^{(i)}}{\bf{\Sigma}}(\theta_0)\Bigr)^{\top}\mathbb{D}_{p}^{\top}\Bigl({\bf{\Sigma}}(\theta_0)^{-1}\otimes{\bf{\Sigma}}(\theta_0)^{-1}\Bigr)\mathbb{D}_{p}\Big(\vech\partial_{\theta^{(j)}}{\bf{\Sigma}}(\theta_0)\Bigr)\\
    &\quad=-\Bigl(\partial_{\theta^{(i)}}\vech{\bf{\Sigma}}(\theta_0)\Bigr)^{\top}{\bf{W}}(\theta_0)^{-1}\Bigl(\partial_{\theta^{(j)}}\vech{\bf{\Sigma}}(\theta_0)\Bigr)\\
    &\quad=-\bigl(\Delta_0^{\top}{\bf{W}}(\theta_0)^{-1}\Delta_0\bigr)_{ij}
\end{align*}
as $n\longrightarrow\infty$, so that we obtain
\begin{align*}
    \frac{1}{n}\partial^2_{\theta}{\rm{H}}_n(\theta_0)\stackrel{p}{\longrightarrow}-{\bf{I}}(\theta_0)
\end{align*}
as $n\longrightarrow\infty$.
\end{proof}
\begin{proof}[\textbf{Proof of Lemma \ref{thetaprob2}}]
{[\bf{B1}]} deduces
\begin{align*}
    {\rm{Y}}(\theta)=0\Longrightarrow \theta=\theta_0.
\end{align*}
For all $\varepsilon>0$, there exists $\delta>0$ such that
\begin{align*}
    |\hat{\theta}_n-\theta_0|>\varepsilon\Longrightarrow {\rm{Y}}(\theta_0)-{\rm{Y}}(\hat{\theta}_n)>\delta.
\end{align*}
In an analogous manner to Lemma 33 in Kusano and Uchida \cite{Kusano(2023)}, we obtain
\begin{align*}
    \sup_{\theta\in\Theta}\Bigl|{\rm{Y}}_n(\theta;\theta_0)-{\rm{Y}}(\theta)\Bigr|\stackrel{p}{\longrightarrow}0.
\end{align*}
Since it holds from the definition of $\hat{\theta}_n$ that
\begin{align*}
    {\rm{Y}}_n(\theta_0;\theta_0)\leq {\rm{Y}}_n(\hat{\theta}_n;\theta_0),
\end{align*}
we see
\begin{align*}
    {\bf{P}}\Bigl(|\hat{\theta}_n-\theta_0|>\varepsilon\Bigr)&\leq{\bf{P}}\biggl({\rm{Y}}(\theta_0)-{\rm{Y}}(\hat{\theta}_n)>\delta\biggr)\\
    &\leq {\bf{P}}\biggl({\rm{Y}}(\theta_0)-{\rm{Y}}_n(\theta_0;\theta_0)>\frac{\delta}{3}\biggr)\\
    &\quad+{\bf{P}}\biggl({\rm{Y}}_n(\theta_0;\theta_0)-{\rm{Y}}_n(\hat{\theta}_n;\theta_0)>\frac{\delta}{3}\biggr)\\
    &\quad+{\bf{P}}\biggl({\rm{Y}}_n(\hat{\theta}_n;\theta_0)-{\rm{Y}}(\hat{\theta}_n)>\frac{\delta}{3}\biggr)\\
    &\leq 2{\bf{P}}\Biggl(\sup_{\theta\in\Theta}\Bigl|{\rm{Y}}_n(\theta;\theta_0)-{\rm{Y}}(\theta)\Bigr|>\frac{\delta}{3}\Biggr)\longrightarrow 0
\end{align*}
as $n\longrightarrow\infty$, which yields
\begin{align}
    \hat{\theta}_n\stackrel{p}{\longrightarrow}\theta_0 \label{thetacons}
\end{align}
as $n\longrightarrow\infty$. Using the Taylor expansion, we have
\begin{align*}
    \frac{1}{\sqrt{n}}\partial_{\theta}{\rm{H}}_n(\hat{\theta}_n)
    &=\frac{1}{\sqrt{n}}\partial_{\theta}{\rm{H}}_n(\theta_0)+
    \biggl(\frac{1}{n}\int_{0}^{1}\partial^2_{\theta}{\rm{H}}_n(\tilde{\theta}_{n,\lambda})d\lambda\biggr)\sqrt{n}(\hat{\theta}_n-\theta_0),
\end{align*}
where $\tilde{\theta}_{n,\lambda}=\theta_0+\lambda(\hat{\theta}_n-\theta_0)$. 
Note that
\begin{align*}
    \frac{1}{\sqrt{n}}\partial_{\theta}{\rm{H}}_n(\hat{\theta}_n)=0
\end{align*}
on $A_n$ and ${\bf{P}}(A_n)\longrightarrow 1$ as $n\longrightarrow\infty$, where
\begin{align*}
    A_n=\Bigl\{\hat{\theta}_n\in{\rm{Int}}\Theta\Bigr\}.
\end{align*}
In a similar manner to Theorem 2 in Kusano and Uchida \cite{Kusano(JJSD)}, it holds from Lemma \ref{thetaprob2} and (\ref{thetacons}) that
\begin{align*}
    \frac{1}{n}\int_{0}^{1}\partial^2_{\theta}{\rm{H}}_n(\tilde{\theta}_{n,\lambda})d\lambda\stackrel{p}{\longrightarrow}-{\bf{I}}(\theta_0)
\end{align*}
as $n\longrightarrow\infty$. Therefore, we see from Lemma \ref{thetaprob2} that
\begin{align*}
    \sqrt{n}(\hat{\theta}_n-\theta_0)\stackrel{d}{\longrightarrow}{\bf{I}}(\theta_0)^{-\frac{1}{2}}\zeta
\end{align*}
as $n\longrightarrow\infty$.
\end{proof}
\begin{proof}[\textbf{Proof of Lemma \ref{mis}}]
In a similar way to Lemma 33 in Kusano and Uchida \cite{Kusano(2023)}, it is shown that
\begin{align*}
    \sup_{\theta_m\in\Theta_m}\biggl|\frac{1}{n}{\rm{H}}_{m,n}(\theta_m)-{\rm{H}}_m(\theta_m)\biggr|\stackrel{p}{\longrightarrow}0.
\end{align*}
Since ${\rm{H}}_m(\theta_m)$ is continuous in $\theta_m\in\Theta_m$, it holds from the continuous mapping theorem and Lemma 36 in Kusano and Uchida \cite{Kusano(2023)} that
\begin{align*}
    {\rm{H}}_m(\hat{\theta}_{m,n})\stackrel{p}{\longrightarrow}{\rm{H}}_m(\bar{\theta}_m)
\end{align*}
as $n\longrightarrow\infty$. Therefore, we see
\begin{align*}
    \biggl|\frac{1}{n}{\rm{H}}_{m,n}(\hat{\theta}_{m,n})-{\rm{H}}_m(\bar{\theta}_m)\biggr|&\leq 
    \biggl|\frac{1}{n}{\rm{H}}_{m,n}(\hat{\theta}_{m,n})-{\rm{H}}_m(\hat{\theta}_{m,n})\biggr|+
    \biggl|{\rm{H}}_m(\hat{\theta}_{m,n})-{\rm{H}}_m(\bar{\theta}_m)\biggr|\\
    &\leq\sup_{\theta\in\Theta}\biggl|\frac{1}{n}{\rm{H}}_{m,n}(\theta_m)-{\rm{H}}_m(\theta_m)\biggr|+
    \biggl|{\rm{H}}_m(\hat{\theta}_{m,n})-{\rm{H}}_m(\bar{\theta}_m)\biggr|\stackrel{p}{\longrightarrow}0
\end{align*}
as $n\longrightarrow\infty$, which yields
\begin{align*}
    \frac{1}{n}{\rm{H}}_{m,n}(\hat{\theta}_{m,n})\stackrel{p}{\longrightarrow}{\rm{H}}_m(\bar{\theta}_m)
\end{align*}
as $n\longrightarrow\infty$.
\end{proof}
\subsection{Proof of (\ref{Y1iden})}\label{identify1}
\begin{proof}
In an analogous manner to Appendix 6.2 in Kusano and Uchida \cite{Kusano(2023)}, it is shown that
\begin{align}
    {\bf{\Sigma}}(\theta)={\bf{\Sigma}}(\theta_0)\Longrightarrow\theta=\theta_0.
    \label{iden}
\end{align}
For ${\bf{\Sigma}}\in\mathcal{M}_p^{++}$, we define
\begin{align*}
    {\rm{G}}_2({\bf{\Sigma}})=\log\det {\bf{\Sigma}}-\log\det {\bf{\Sigma}}(\theta_0)+\tr\Bigl({\bf{\Sigma}}^{-1}{\bf{\Sigma}}(\theta_0)\Bigr)-p.
\end{align*}
Note that ${\rm{G}}_2({\bf{\Sigma}})$ has the unique minimum point at ${\bf{\Sigma}}={\bf{\Sigma}}(\theta_0)$. Since
\begin{align*}
    {\rm{Y}}(\theta)
    &=-\frac{1}{2}\Bigl\{\log\det{\bf{\Sigma}}(\theta)-\log\det{\bf{\Sigma}}(\theta_0)+\tr\Bigl({\bf{\Sigma}}(\theta)^{-1}{\bf{\Sigma}}(\theta_0)\Bigr)-p\Bigr\}\\
    &=-\frac{1}{2}{\rm{G}}_2\bigl({\bf{\Sigma}}(\theta)\bigr),
\end{align*}
it holds from (\ref{iden}) that ${\rm{Y}}(\theta)$ has the unique maximum point at $\theta=\theta_0$, which yields
\begin{align}
    \sup_{|\theta-\theta_0|>v}{\rm{Y}}(\theta)<{\rm{Y}}(\theta_0)=0 \label{supY}
\end{align}
for all $v>0$. The Taylor expansion of ${\rm{Y}}(\theta)$ around $\theta=\theta_0$ is given by
\begin{align}
\begin{split}
    {\rm{Y}}(\theta)&={\rm{Y}}(\theta_{0})+\partial_{\theta}{\rm{Y}}(\theta_{0})^{\top}(\theta-\theta_{0})\\
    &\qquad\qquad\qquad+(\theta-\theta_{0})^{\top}\biggl(\int_{0}^{1}(1-\lambda)\partial^2_{\theta}{\rm{Y}}(\theta_{\lambda})d\lambda\biggr)(\theta-\theta_{0})\\
    &=\frac{1}{2}(\theta-\theta_{0})^{\top}\partial^2_{\theta}{\rm{Y}}(\theta_{0})(\theta-\theta_{0})\\
    &\qquad\qquad\qquad+(\theta-\theta_{0})^{\top}\biggl(\int_{0}^{1}(1-\lambda)\partial^2_{\theta}{\rm{Y}}(\theta_{\lambda})d\lambda-\frac{1}{2}\partial^2_{\theta}{\rm{Y}}(\theta_{0})\biggr)(\theta-\theta_{0}), \label{Ytaylor}
\end{split}
\end{align}
where $\theta_{\lambda}=\theta_0+\lambda(\theta-\theta_0)$. In a similar way to Theorem 2 in Kusano and Uchida \cite{Kusano(JJSD)}, it is shown that
\begin{align*}
    \int_{0}^{1}(1-\lambda)\partial^2_{\theta}{\rm{Y}}(\theta_{\lambda})d\lambda\longrightarrow\frac{1}{2}\partial^2_{\theta}{\rm{Y}}(\theta_{0})
\end{align*}
as $\theta\longrightarrow\theta_0$, which deduces 
\begin{align*}
    (\theta-\theta_{0})^{\top}\biggl(\int_{0}^{1}(1-\lambda)\partial^2_{\theta}{\rm{Y}}(\theta_{\lambda})d\lambda-\frac{1}{2}\partial^2_{\theta}{\rm{Y}}(\theta_{0})\biggr)(\theta-\theta_{0})\longrightarrow 0
\end{align*}
as $\theta\longrightarrow\theta_0$. Hence, for all $\varepsilon>0$, there exists $\delta>0$ such that
\begin{align*}
    |\theta-\theta_0|\leq\delta\Longrightarrow\biggl|(\theta-\theta_{0})^{\top}\biggl(\int_{0}^{1}(1-\lambda)\partial^2_{\theta}{\rm{Y}}(\theta_{\lambda})d\lambda-\frac{1}{2}\partial^2_{\theta}{\rm{Y}}(\theta_{0})\biggr)(\theta-\theta_{0})\biggr|<\varepsilon,
\end{align*}
so that we see from (\ref{Ytaylor}) that
\begin{align*}
    {\rm{Y}}(\theta)<\frac{1}{2}(\theta-\theta_{0})^{\top}\partial^2_{\theta}{\rm{Y}}(\theta_{0})(\theta-\theta_{0})+\varepsilon
\end{align*}
for $\theta\in B_{\delta}(\theta_0)$, where
\begin{align*}
    B_{\delta}(\theta_0)=\Bigl\{\theta\in\Theta: |\theta-\theta_0|\leq \delta\Bigr\}.
\end{align*}
Note that it holds from the proof of Lemma \ref{thetaprob1} that ${\bf{I}}(\theta_0)=-\partial^2_{\theta}{\rm{Y}}(\theta_{0})$.
Since $\varepsilon>0$ is arbitrary, one has
\begin{align*}
    {\rm{Y}}(\theta)\leq-\frac{1}{2}(\theta-\theta_{0})^{\top}{\bf{I}}(\theta_0)(\theta-\theta_{0})
\end{align*}
as $\varepsilon\downarrow 0$ for $\theta\in B_{\delta}(\theta_0)$. Recalling that ${\bf{I}}(\theta_0)$ is a positive definite matrix, we have
\begin{align*}
    \lambda_{min}|\theta-\theta_{0}|^2\leq(\theta-\theta_{0})^{\top}{\bf{I}}(\theta_0)(\theta-\theta_{0})\leq\lambda_{max}|\theta-\theta_{0}|^2,
\end{align*}
where $\lambda_{min}>0$ and $\lambda_{max}>0$ are the minimum and maximum eigenvalues of ${\bf{I}}(\theta_0)$. There exists $C_1>0$ such that
\begin{align}
    {\rm{Y}}(\theta)\leq -C_1|\theta-\theta_{0}|^2 \label{YC1}
\end{align}
for $\theta\in B_{\delta}(\theta_0)$. Let
\begin{align*}
    {\rm{Diam}}\Theta=\sup_{\theta_1,\theta_2\in\Theta}|\theta_1-\theta_2|>0.
\end{align*}
Since
\begin{align*}
    \frac{1}{{\rm{Diam}}\Theta}|\theta-\theta_0|\leq 1,
\end{align*}
we see from (\ref{supY}) that
\begin{align*}
\begin{split}
    {\rm{Y}}(\theta)&\leq \sup_{|\theta-\theta_0|>\delta}{\rm{Y}}(\theta)\\
    &\leq -\Biggl(-\sup_{|\theta-\theta_0|>\delta}{\rm{Y}}(\theta)\Biggr)\frac{1}{({\rm{Diam}}\Theta)^2}|\theta-\theta_0|^2
\end{split}
\end{align*}
for $\theta\in B_{\delta}(\theta_0)^c$, so that there exists $C_2>0$ such that
\begin{align}
    {\rm{Y}}(\theta)\leq -C_2|\theta-\theta_{0}|^2 \label{YC2}
\end{align}
for $\theta\in B_{\delta}(\theta_0)^c$. Therefore, it follows from (\ref{YC1}) and (\ref{YC2}) that
\begin{align*}
    {\rm{Y}}(\theta)\leq -C|\theta-\theta_{0}|^2
\end{align*}
for $\theta\in\Theta$, where $C=\min (C_1,C_2)$.
\end{proof}
\subsection{Ergodic case}\label{ergodic} \ \\
{\setlength{\abovedisplayskip}{12pt}
\setlength{\belowdisplayskip}{12pt}
In this section, we consider the ergodic case. The following assumptions are made.
\begin{enumerate}
    \item[{\bf{[C]}}]
    \begin{enumerate}
        \item The diffusion process $\xi_t$ is ergodic with its invariant measure $\pi_{\xi,0}$. For any $\pi_{\xi,0}$-integrable function $f_1$, it holds that
        \begin{align*}
        \frac{1}{T}\int_{0}^{T}{f_1(\xi_t)dt}\overset{p}{\longrightarrow}\int f_1(x)\pi_{\xi,0}(dx)
        \end{align*}
        as $T\longrightarrow\infty$. 
        \item The diffusion process $\delta_t$ is ergodic with its invariant measure $\pi_{\delta,0}$. For any $\pi_{\delta,0}$-integrable function $f_2$, it holds that
        \begin{align*}
        \frac{1}{T}\int_{0}^{T}{f_2(\delta_t)dt}\overset{p}{\longrightarrow}\int f_2(x)\pi_{\delta,0}(dx)
        \end{align*}
        as $T\longrightarrow\infty$. 
        \item The diffusion process $\varepsilon_t$ is ergodic with its invariant measure $\pi_{\varepsilon,0}$. For any $\pi_{\varepsilon,0}$-integrable function $f_3$, it holds that
        \begin{align*}
        \frac{1}{T}\int_{0}^{T}{f_3(\varepsilon_t)dt}\overset{p}{\longrightarrow}\int f_3(x)\pi_{\varepsilon,0}(dx)
        \end{align*}
        as $T\longrightarrow\infty$. 
        \item The diffusion process $\zeta_t$ is ergodic with its invariant measure $\pi_{\zeta,0}$. For any $\pi_{\zeta,0}$-integrable function $f_4$, it holds that
        \begin{align*}
        \frac{1}{T}\int_{0}^{T}{f_4(\zeta_t)dt}\overset{p}{\longrightarrow}\int f_4(x)\pi_{\zeta,0}(dx)
        \end{align*}
        as $T\longrightarrow\infty$. 
    \end{enumerate}
\end{enumerate}
In the ergodic case, we have the following results similar to the non-ergodic case.
\begin{theorem}\label{theorem1er}
Let $m\in\{1,\cdots,M\}$. Under {\bf{[A]}}, {\bf{[B1]}} and {\bf{[C]}}, as $h_n\longrightarrow 0$, $nh_n\longrightarrow\infty$ and $nh_n^2\longrightarrow 0$,
\begin{align*}
    {\bf{E}}_{\mathbb{X}_n}\biggl[\log {\rm{L}}_{m,n}\bigl(\mathbb{X}_{n},\hat{\theta}_{m,n}({\mathbb{X}_{n}})\bigr)-{\bf{E}}_{\mathbb{Z}_n}\Bigl[\log {\rm{L}}_{m,n}\bigl(\mathbb{Z}_n,\hat{\theta}_{m,n}({\mathbb{X}_{n}})\bigr)\Bigr]\biggr]=q_m+o_p(1).
\end{align*}
\end{theorem}
\begin{theorem}
Under {\bf{[A]}}, {\bf{[B2]}} and {\bf{[C]}}, as $h_n\longrightarrow 0$, $nh_n\longrightarrow\infty$ and $nh_n^2\longrightarrow 0$,
\begin{align*}
    {\bf{P}}\Bigl(\hat{m}_{n}\in\mathcal{M}^{c}\Bigr)\longrightarrow 0. 
\end{align*}    
\end{theorem}}
\begin{proof}[\textbf{Proofs of Theorems 3-4}]
Since $h_n\longrightarrow 0$ and $nh_n^2\longrightarrow\infty$, we can prove the results in the same way as the proofs of Theorems 1-2.
\end{proof}
\end{document}